\documentclass{article}
\usepackage[T1]{fontenc}
\usepackage{times}
\usepackage[margin=1in]{geometry}
\usepackage{amsmath, amsfonts, amssymb, amsthm}
\usepackage{latexsym, mathtools, bm, bbm, dsfont, tocloft, mathrsfs,authblk,enumitem}
\usepackage[noadjust]{cite}
\usepackage{ytableau}
\usepackage{hyperref,cleveref}
\hypersetup{
    colorlinks=true,
    breaklinks=true,
    linkcolor= black,
    citecolor= black
    }
\usepackage{color}
\usepackage{graphicx}

\newtheorem{theorem}{Theorem}[section]
\newtheorem{lemma}[theorem]{Lemma}
\newtheorem{prop}[theorem]{Proposition}
\newtheorem{cor}[theorem]{Corollary}
\newtheorem{conj}[theorem]{Conjecture}

\theoremstyle{remark}
\newtheorem{rem}{Remark}

\def\N{\mathbb{N}}

\def\Q{\mathbb{Q}}
\def\R{\mathbb{R}}
\def\lb{[\![}
\def\rb{]\!]}
\def\dist{\mathrm{dist}}
\def\d{\mathrm{d}}
\def\dH{\mathrm{d_H}}
\def\on{\underset{\as{n}{\infty}}{o}}
\def\oeps{\underset{\varepsilon\rightarrow0}{o}}
\def\m{\!-\!}
\def\mm{\!\!-\!\!}
\def\p{\!+\!}
\def\pp{\!\!+\!\!}
\newcommand{\as}[2]{#1\rightarrow#2}
\newcommand{\cv}[2]{\underset{\as{#1}{#2}}{\longrightarrow}}

\def\E{\mathbb{E}}
\def\P{\mathbf{P}}
\def\Leb{\mathrm{Leb}}
\def\meas{\mathcal{M}_\text{fin}}
\def\prob{\mathcal{M}_\text{prob}}

\def\carre{[0,1]^2}
\def\LIS{\mathrm{LIS}}
\def\LISt{\widetilde{\LIS}}
\def\LDS{\mathrm{LDS}}
\def\LDSt{\widetilde{\LDS}}

\def\inv{\mathrm{inv}}
\def\invt{\widetilde{\inv}}
\def\id{\mathrm{id}}
\def\S{\mathfrak{S}}
\def\sample{\mathrm{Sample}}
\def\perm{\mathrm{Perm}}
\def\per{\mathcal{M}_\text{unif}}
\def\preper{\mathcal{M}_\text{cont}}
\def\RS{\mathrm{RS}}
\def\RSt{\mathrm{\widetilde{RS}}}
\def\shape{\mathrm{sh}}
\def\shapet{\mathrm{\widetilde{sh}}}
\def\lamt{\widetilde{\lambda}}
\def\row{\mathrm{row}}
\def\col{\mathrm{col}}
\def\Fom{\mathcal{F}}
\def\top{\mathrm{top}}
\def\rig{\mathrm{right}}
\def\bot{\mathrm{bot}}
\def\lef{\mathrm{left}}
\def\ord{\mathrm{ord}}
\def\block{\mathrm{blocks}}
\def\rev{\mathrm{rev}}
\def\incr{\mathrm{incr}}
\def\decr{\mathrm{decr}}
\def\sub{\mathrm{sub}}

\title{Increasing subsequences of linear size in random permutations and the Robinson--Schensted tableaux of permutons}
\author[*]{Victor Dubach}
\affil[*]{Universit\'e de Lorraine, IECL, Nancy, France}
\date{}

\begin{document}

\maketitle

\begin{abstract}
    The study of longest increasing subsequences (LIS) in permutations led to that of Young diagrams via Robinson--Schensted's (RS) correspondence.
    In a celebrated paper, Vershik and Kerov obtained a limit theorem for such diagrams and found that the LIS of a uniform permutation of size $n$ behaves as $2\sqrt{n}$. 
    Independently and much later, Hoppen et al.~introduced the theory of permutons as a scaling limit of permutations.
    In this paper, we extend in some sense the RS correspondence of permutations to the space of permutons.
    When the ``RS-tableaux'' of a permuton are non-trivial, we show that the RS-tableaux of random permutations sampled from this permuton exhibit a linear behavior, in the sense that their first rows and columns have lengths of linear order.
    In particular, the LIS of such permutations behaves as a multiple of $n$. 
    We also prove some large deviation results for these convergences.
    Finally, by studying asymptotic properties of Fomin's algorithm for permutations, we show that the RS-tableaux of a permuton satisfy a partial differential equation.
\end{abstract}


\section{Background}

\subsection{The theory of permutons and sampled permutations}\label{section_intro_permutons}

A Borel probability measure $\mu$ on $\carre$ is called a \textit{permuton} when each of its marginals is uniform, \textit{i.e.}~for any $0\leq~x\leq~y\leq1$:
\begin{equation*}
    \mu\big( [x,y]\times[0,1] \big) = \mu\big( [0,1]\times[x,y] \big) = y-x.
\end{equation*}
The theory of permutons was first studied in \cite{HKMRS13}, named in \cite{GGKK15}, and is now widely studied (see e.g.~\cite{KP12,M15,BBFGP18,BBFGMP22} and references therein). It serves as a scaling limit model for permutations: if $\sigma$ is a permutation of size $n\in\N^*$, one can associate it with the permuton $\mu_\sigma$ having density
\begin{equation*}
    n\sum_{i=1}^n \mathbf{1}_{\left[\frac{i\m1}{n},\frac{i}{n}\right] \times \left[\frac{\sigma(i)\m1}{n},\frac{\sigma(i)}{n}\right]}
\end{equation*}
with respect to Lebesgue measure on $\carre$. Putting weak convergence topology on the space of permutons, this gives a meaning to the convergence of a sequence of permutations towards a limit permuton. 
This convergence is notably characterized by the asymptotics of pattern densities, 
cementing the theory of permutons within the general framework of sampling limit theory (see e.g.~\cite{LS06} for dense graphs, \cite{J11} for posets, \cite{ET22} and \cite{G23trees} for trees).

\medskip

\begin{figure}
    \centering
    \includegraphics[scale=0.88]{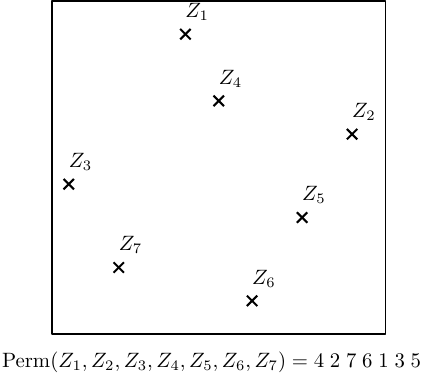}
    \caption{The permutation $\sigma$ associated to a family of points, written in one-line notation. Here $\sigma(1)=4$ because the first point from the left, $Z_3$, is the fourth point point from the bottom, and so on.}
    \label{fig_perm_points}
\end{figure}

Fundamental to permuton theory is the fact that each permuton gives rise to a model of random permutations.
Consider points $Z_1,\dots,Z_n$ in the unit square $\carre$ whose $x$-coordinates and $y$-coordinates are all distinct. 
One can then naturally define a permutation $\sigma$ of size $n$ in the following way: for any $1\leq i,j\leq n$, let $\sigma(i)=j$ whenever the point with $i$-th lowest $x$-coordinate has $j$-th lowest $y$-coordinate.
We denote by $\perm(Z_1,\dots,Z_n)$ this permutation; see \Cref{fig_perm_points} for an example.
When $\mu$ is a permuton and $Z_1,\dots,Z_n$ are random i.i.d.~points distributed under $\mu$, we write $\sample_n(\mu)$ for the law of this random permutation.
Since the marginals of $\mu$ are continuous, this permutation is almost surely well defined.

\begin{theorem}[\cite{HKMRS13}]\label{th_HKMRS}
    Let $\mu$ be a permuton and $(\sigma_n)_{n\in\N^*}$ be a sequence of random permutations such that for each $n\in\N^*, \sigma_n \sim \sample_n(\mu)$.
    Then the following weak convergence holds almost surely:
    \begin{equation*}
        \mu_{\sigma_n} \cv{n}{\infty} \mu.
    \end{equation*}
\end{theorem}

Note that this model can be generalized to more than just permutons.
As in \cite{D23} we say that a Borel probability measure on $\carre$ is a \textit{pre-permuton} when none of its marginals have any atom; then the law $\sample_n(\mu)$ is still well defined.
This notion has the advantage of being stable by induced sub-measures, which is why we use it for some results.
Going from pre-permutons to permutons and conversely is fairly easy: one can associate any pre-permuton $\mu$ with a unique permuton $\hat{\mu}$ such that random permutations sampled from $\mu$ or $\hat{\mu}$ have the same law (see e.g.~Remark 1.2 in \cite{BDMW22}).
Moreover this is done via a transformation of the unit square which is nondecreasing in each coordinate, hence most quantities defined in later sections have the same values on pre-permutons as on their associated permutons.

\subsection{Longest increasing subsequences and Robinson--Schensted's correspondence}\label{section_intro_RS}

Let $w=w_1\dots w_n$ be a finite word with positive integer letters. 
We say that a subset of indices $\{i_1<\dots<i_\ell\}$ is an \textit{increasing subsequence of $w$} if $w_{i_1} <\dots< w_{i_\ell}$.
The maximum length of such a sequence is called length of the longest increasing subsequence of $w$ and denoted by $\LIS(w)$.
More generally for $k\in\N^*$, define the length of the longest $k$-increasing subsequence of $w$ as
\begin{equation*}
    \LIS_k(w)
    := \max_{I_1,\dots,I_k} \vert I_1\cup\dots\cup I_k\vert
\end{equation*}
where the maximum is taken over all families $I_1,\dots,I_k$ (not necessarily disjoint) of increasing subsequences of $w$.
Also define $\LDS_k$ similarly by replacing ``increasing'' with ``decreasing'' in the previous definition.
The study of longest increasing subsequences is intimately related to a bijective map called Robinson--Schensted's correspondence, which we present next.

\medskip

A \textit{(Young) diagram} is a visual representation of an integer partition: it consists of a sequence of left-justified rows, containing nonincreasing numbers of boxes.
A \textit{(Young) tableau} is a filling of a diagram, called its \textit{shape}, with positive integers.
When the filling is nondecreasing (resp.~increasing) in each row and increasing in each column, we call the tableau \textit{semistandard} (resp.~\textit{standard}).
Robinson--Schensted's (RS) correspondence is a map between words and pairs of tableaux sharing the same shape, see \Cref{fig_ex_tableaux}.
The first tableau, called the P-tableau, is semistandard and filled with the letters of $w$, while the second one, called the Q-tableau, is standard and filled with the integers from $1$ to the length of $w$.
If $\sigma$ is a permutation, seen as a word via one-line notation, then its P-tableau is standard.

There are many ways to define this correspondence, for which we refer the reader to \cite{F96}.
Below we recall its equivalent description given by Greene's theorem.
For any word $w$, define
\begin{equation*}
    \shape(w) 
    = \left( \shape^\row(w) \,,\, \shape^\col(w) \right)
    := \Big( 
    \big( \LIS_k(w)-\LIS_{k\m1}(w) \big)_{k\in\N^*} \;,\;
    \big( \LDS_k(w)-\LDS_{k\m1}(w) \big)_{k\in\N^*}
    \Big)
\end{equation*}
with convention $\LIS_0(w)=\LDS_0(w)=0$.

\begin{figure}
    \centering
    $P(\sigma)=$
    \begin{ytableau}
    1&3&5\\
    2&6\\
    4&7
    \end{ytableau}
    \quad,\quad
    $Q(\sigma)=$
    \begin{ytableau}
    1&3&7\\
    2&4\\
    5&6
    \end{ytableau}
    \quad,\quad
    $\shape(\sigma)=$
    \begin{ytableau}
    \ &\ &\\
    \ &\\
    \ &
    \end{ytableau}
    \caption{The standard tableaux associated with the permutation $\sigma = 4\,2\,7\,6\,1\,3\,5$, and their common shape. The first row has length $3$, and indeed $\LIS(\sigma)=3$ (attained e.g.~with the increasing subsequence $2\,3\,5$). The second row has length $2$, and indeed $\LIS_2(\sigma)=3+2=5$ (attained e.g.~with the $2$-increasing subsequence $4\,2\,6\,3\,5$). Finally, the shape has a total of three rows, and indeed $\sigma$ is $3$-increasing (one can e.g.~decompose it into $4\,6$, $2\,7$, and $1\,3\,5$). Greene's theorem can also be checked for the column lengths.}\label{fig_ex_tableaux}
\end{figure}

\begin{theorem}[\cite{G74}]\label{th_Greene}
    For any word $w$, the sequences $\shape^\row(w)$ and $\shape^\col(w)$ are respectively the row lengths and column lengths of the RS-shape of $w$.
    In particular, those sequences are nonincreasing.
\end{theorem}

Note that the definition of $\shape$ is redundant, as the sequence of row lengths can be deduced from the sequence of column lengths and vice versa (these are usually called \textit{conjugate partitions}).
However we shall see that this property is lost in the permuton limit.

\medskip

Greene's theorem can also be used to define the RS-tableaux.
A standard tableau of size $n$ can be equivalently defined as an increasing sequence of $n$ diagrams, where the $k$-th diagram is the shape of the sub-tableau containing the boxes with values at most $k$.
For example in \Cref{fig_ex_tableaux}, the $4$-th diagram defining $P(\sigma)$ has row lengths $(2,1,1)$ and the $4$-th diagram defining $Q(\sigma)$ has row lengths $(2,2)$.

Let $\sigma$ be a permutation of $\{1,\dots,n\}$. 
For any $1\leq i,j\leq n$ define its sub-word $\sigma^{i,j}$ as the sequence of letters $\sigma(h)$ with $h\leq i$ and $\sigma(h)\leq j$.
One can see that for each $1\leq k\leq n$, the RS-shape of $\sigma^{n,k}$ is the $k$-th diagram defining $P(\sigma)$ and the RS-shape of $\sigma^{k,n}$ is the $k$-th diagram defining $Q(\sigma)$.
This motivates the following alternative definition of Robinson--Schensted's tableaux:
\begin{equation}\label{def_RS_permutations}
    \RS(\sigma) := \big( \lambda^{\sigma}(n,\cdot) , \lambda^{\sigma}(\cdot,n) \big)
\end{equation}
where, for any $1\leq i,j\leq n$,
\begin{equation*}
\lambda^{\sigma}(i,j) = \big( \lambda^{\sigma}_{k}(i,j) \big)_{k\in\N^*}
:= \shape^\row\left( \sigma^{i,j} \right).
\end{equation*}

\subsection{LIS and RS-shapes of random permutations}

Ulam formulated the following question in the 60's: consider for each $n\in\N^*$ a uniformly random permutation $\sigma_n$ of size $n$, and write
\begin{equation*}
    \ell_n := \E\big[ \LIS(\sigma_n) \big] .
\end{equation*}
Then what can we say about the asymptotic behavior of $\ell_n$ as $\as{n}{\infty}$?
The study of longest increasing subsequences has since then been a surprisingly fertile research subject with unexpected links to diverse areas of mathematics; see \cite{R15} for a nice introduction to this topic. 

A solution to Ulam's problem was found by Vershik and Kerov through the use of Robinson--Schensted's correspondence. 
More precisely they encoded diagrams $\lambda$ by some continuous functions $\omega_\lambda:\R\rightarrow\R$ and established the following limit theorem (also simultaneously proved in \cite{LS77}).
See \Cref{fig_LSKV} for an illustration.

\begin{theorem}[\cite{VK77}]
    For each integer $n$, let $\sigma_n$ be a uniformly random permutation of size $n$ and write
    \begin{equation*}
        \omega_n(\cdot) := \frac{\omega_{\shape(\sigma_n)}(\sqrt{n}\;\cdot)}{\sqrt{n}}.
    \end{equation*}
    Then the sequence $(\omega_n)_{n\in\N^*}$ converges uniformly in probability to some explicit limit curve.
\end{theorem}

\begin{figure}
    \centering
    \includegraphics[scale=0.4]{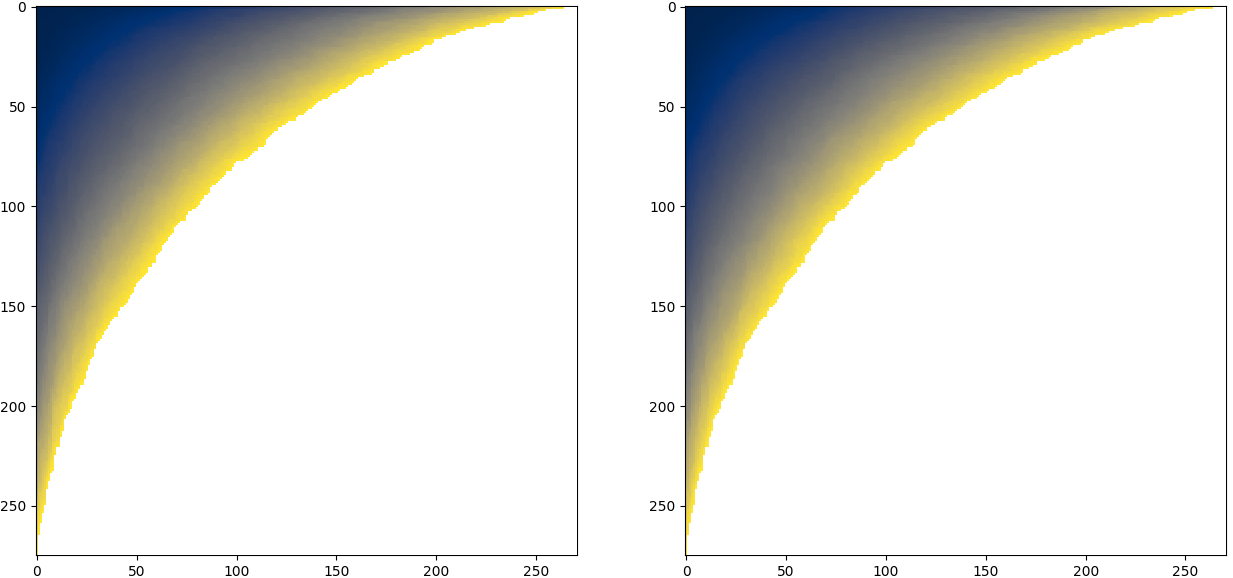}
    \caption{The RS-tableaux of a size $20000$ uniformly random permutation. Higher values in the tableaux are represented by brighter colors.}
    \label{fig_LSKV}
\end{figure}

By \Cref{th_Greene} this convergence corresponds to explicit non-trivial limits as $n$ goes to infinity for $\frac{1}{n}\LIS_{\alpha\sqrt{n}}(\sigma_n)$, $\alpha\in[0,2]$.
This also implies $\liminf \frac{1}{\sqrt n}\ell_n \ge 2$, and with an additional argument for the upper bound we can deduce:

\begin{cor}[\cite{VK77}]
For each $n$, let $\sigma_n$ be a uniformly random permutation of size $n$ and write $\ell_n:=\E[\LIS(\sigma_n)]$. Then 
\begin{equation*}
    \frac{1}{\sqrt{n}}\ell_n \cv{n}{\infty} 2.
\end{equation*}
\end{cor}

The study of $\ell_n$ later culminated with the work of Baik, Deift and Johansson \cite{BDJ99}, who found its second-order asymptotics.

\medskip

One can then naturally ask: is it possible to obtain similar asymptotics for sampled random permutations?
One of the first advances on this question was obtained by Deuschel and Zeitouni who proved:
\begin{theorem}[\cite{DZ95}]
Let $\mu_\rho$ be a permuton with a $C_b^1$ density $\rho$ which is bounded below on $\carre$. 
If for each integer $n$, $\sigma_n\sim\sample_n(\mu_\rho)$, then:
\begin{equation*}
    \frac{1}{\sqrt{n}}\LIS(\sigma_n)
    \cv{n}{\infty} K_\rho
\end{equation*}
in probability, for some positive constant $K_\rho$ defined by a variational problem.
\end{theorem}

This question of generalization is still an active research subject. 
In the recent paper \cite{S22}, the author studies the $\sqrt{n}$ scaling limit for the RS-shape of sampled permutations when the sampling permuton is absolutely continuous.

\medskip

The goal of this paper is to extend in some sense the function $\LIS$, as well as the first rows and columns of RS-shape and tableaux, to the space of permutons. 
This will allow us to deduce linear asymptotics of these quantities for sampled permutations; therefore our results are different in nature from the $\sqrt{n}$ regime of \cite{DZ95,S22}, and yield non-trivial asymptotics for certain types of permutons displaying a singular part with respect to the Lebesgue measure on $\carre$.
We introduce the RS-shape of permutons in \Cref{section_shape_results} and present these asymptotics. 
Then in \Cref{section_ldp_results} we investigate large deviation principles related to those convergences, exhibiting a difference of behavior between the upper and lower tails.
In \Cref{section_tableaux_results} we define the RS-tableaux of a permuton and explain how a discrete algorithm translates in this continuous setting, then in \Cref{section_injectivity_results} we present a few results about the injectivity of permutons' RS-tableaux.
Finally, we illustrate these objects with several examples in \Cref{section_examples}.

\section{Our results}

\subsection{The RS-shape of (pre-)permutons}\label{section_shape_results}

Throughout the paper we denote by $\per \subset \preper \subset \prob \subset \meas$
the sets of permutons, pre-permutons, probability measures and finite measures on $\carre$.
Say a subset $A\subset\carre$ is nondecreasing (resp.~nonincreasing) if
\begin{equation*}
    \text{for any }(x_1,y_1),(x_2,y_2)\in A,\quad
    (x_2-x_1)(y_2-y_1)\geq 0 
    \quad\big(\text{resp. }(x_2-x_1)(y_2-y_1)\leq 0\big)
\end{equation*}
and for any $k\in\N^*$, say $A$ is $k$-nondecreasing (resp.~$k$-nonincreasing) if it can be written as a union of $k$ (not necessarily disjoint) nondecreasing (resp.~nonincreasing) subsets of $\carre$. 
Denote by $\mathcal{P}_K^{k\nearrow}(\carre)$ and $\mathcal{P}_K^{k\searrow}(\carre)$ respectively the sets of closed $k$-nondecreasing and $k$-nonincreasing subsets, and define for any $\mu\in\meas$:
\begin{equation*}
    \LISt_k(\mu) := \sup_{A\in\mathcal{P}_K^{k\nearrow}(\carre)}\mu(A)
    \quad,\quad
    \LDSt_k(\mu) := \sup_{A\in\mathcal{P}_K^{k\searrow}(\carre)}\mu(A).
\end{equation*}
When $k=1$ we may drop the subscript $k$.
By convention we set $\LISt_0(\mu)=\LDSt_0(\mu):=0$.

\begin{prop}\label{prop_semicon_LISt}
For any $\mu\in\meas$, the supremum in the definition of $\LISt_k(\mu)$ is reached. 
Moreover the function $\LISt_k$ is upper semi-continuous on $\prob$, \textit{i.e.} for any sequence $(\mu_n)_{n\in\N}$ weakly converging to $\mu$ in $\prob$ one has $\limsup \LISt_k( \mu_n ) \leq \LISt_k(\mu)$.
The same holds for $\LDSt_k$.
\end{prop}

The functions $\LISt_k$ are not continuous:
for instance, one can check that the permutons $\mu_\sigma$ defined in \Cref{section_intro_permutons} satisfy $\LISt(\mu_\sigma)=0$ for all permutations $\sigma$, although this family is dense in $\prob$ by \Cref{th_HKMRS}.

The next result, along with \Cref{lem_lien_lis_list} used to prove it, justifies the idea that $\LISt_k,\LDSt_k$ are generalizations of $\LIS_k,\LDS_k$ to the space of (pre-)permutons:

\begin{prop}\label{prop_conv_LISk}
Let $\mu\in\preper$ and $(\sigma_n)_{n\in\N^*}$ be such that for each $n\in\N^*$, $\sigma_n\sim\sample_n(\mu)$. For any $k\in\N^*$, the following convergences hold almost surely:
\begin{equation*}
    \frac{1}{n}\LIS_k\left(\sigma_n\right)
    \cv{n}{\infty} \LISt_k(\mu)
    \quad;\quad
    \frac{1}{n}\LDS_k\left(\sigma_n\right)
    \cv{n}{\infty} \LDSt_k(\mu).
\end{equation*}
\end{prop}

This result does not hold for any sequence of permutations $\sigma_n$ such that $\mu_{\sigma_n}$ converges to $\mu$.
For example, it is easy to construct a sequence of permutations with sublinear $\LIS$ that converges towards the diagonal permuton.
We could expect this result to hold under appropriate regularity assumptions on the sequence $(\sigma_n)$, similar to the ones exhibited in \cite{M15} for the study of small cycles, but such assumptions remain to be discovered.

We can then extend the notion of RS-shape to the space of permutons by analogy with Greene's theorem.
For any $\mu\in\meas$, define:
\begin{equation*}
    \shapet(\mu) 
    = \left( \shapet^\row(\mu) , \shapet^\col(\mu) \right) 
    := \left( 
    \left( \LISt_k(\mu)-\LISt_{k\m1}(\mu) \right)_{k\in\N^*} ,
    \left( \LDSt_k(\mu)-\LDSt_{k\m1}(\mu) \right)_{k\in\N^*}
    \right).
\end{equation*}
An immediate consequence of \Cref{prop_conv_LISk} is the convergence of sampled permutations' shapes towards the sampling permuton's shape, component-wise.
Another one is that this quantity belongs to the Thoma simplex, which appears in the limit theory of integer partitions (see e.g.~\cite{FMN20} and references therein):

\begin{prop}\label{prop_Thoma}
Define the Thoma simplex $\Omega$ as
\begin{equation*}
    \Omega := \left\{ \Big( \, \big(\alpha_1\geq\alpha_2\geq\dots\geq0\big) \, , \, \big(\beta_1\geq\beta_2\geq\dots\geq0\big) \, \Big) \::\: \sum_{i\geq 1} \alpha_i +
    \sum_{j\geq 1} \beta_j \leq 1 \right\}.
\end{equation*}
Then for any $\mu\in\preper$, $\shapet(\mu) \in \Omega$.
\end{prop}

\begin{rem}\label{rem_taquin}
Propositions \ref{prop_conv_LISk} and \ref{prop_Thoma} can be related to a law of large numbers for random partitions due to Vershik and Kerov.
Let $(\alpha,\beta)\in\Omega$ and $\gamma:= 1 - \sum_i\alpha_i - \sum_i\beta_i$.
In \cite{VK81} the authors defined a measure $M_{\alpha,\beta,\gamma}^n$ on Young diagrams of size $n$ and proved that $\alpha_i$ (resp.~$\beta_i$) is the asymptotic proportion of boxes in the $i$-th row (resp.~column). 
In \cite{VK86} they showed that this measure could be interpreted as the RS-shape of some random word.
For this, fix an alphabet $\mathbb{L}=\{a_1,a_2,\dots\}\sqcup\{b_1,b_2,\dots\}\sqcup[0,1]$ where $a_i$'s are called \textit{row letters}, $b_i$'s are called \textit{column letters}, and elements of $[0,1]$ are neutral letters.
Define a measure $m_{\alpha,\beta,\gamma}^n$ on words of length $n$ by choosing independently each letter with a probability $\alpha_i$ of being $a_i$, $\beta_i$ of being $b_i$, and $\gamma$ of being uniform in $[0,1]$.
Adapt Robinson--Schensted's algorithm on such words by requiring that each row of the P-tableau contains each column letter at most once, and each column of the P-tableau contains each row letter at most once (but rows may contain one row letter multiple times and columns may contain one column letter multiple times).
Then the measure $M_{\alpha,\beta,\gamma}^n$ is the push-forward by RS-shape of $m_{\alpha,\beta,\gamma}^n$.
Equivalently, it is possible to interpret random diagrams under $M_{\alpha,\beta,\gamma}^n$ as the RS-shapes of certain sampled permutations.
For this, associate each row letter $a_i$ with an upward diagonal of mass $\alpha_i$ and each column letter $b_i$ with a downward diagonal of mass $\beta_i$, then place them in $\carre$ and interpose uniform mass according to the order on $\mathbb{L}$.
This defines a permuton $\mu_{\alpha,\beta,\gamma}$ whose sampled permutations are standardizations of words under $m_{\alpha,\beta,\gamma}^n$ (see e.g.~Section 3.7 in \cite{S14} for a definition), hence their RS-shapes follow the law $M_{\alpha,\beta,\gamma}^n$ and we can recover with \Cref{prop_conv_LISk} the asymptotics of \cite{VK81}.
See \Cref{fig_taquin} for a representation of the permuton $\mu_{\alpha,\beta,\gamma}$.
\end{rem}

\begin{figure}
    \centering
    \includegraphics[scale=0.44]{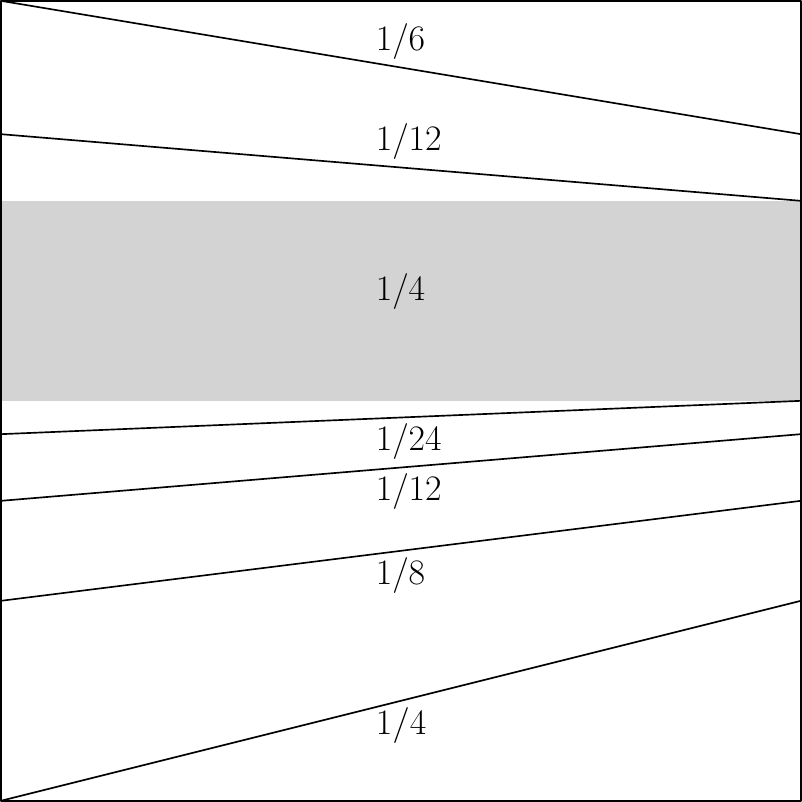}
    \caption{The permuton $\mu_{\alpha,\beta,\gamma}$ when the alphabet $\mathbb{L}$ puts the order of $\N$ on row letters, the order of $-\N$ on column letters, and row letters are lower than neutral letters which are lower than column letters.
    In this example the parameters are $\alpha=(1/4,1/8,1/12,1/24)$, $\beta=(1/6,1/12)$, $\gamma=1/4$.
    This alphabet is called the ``jeu de taquin alphabet'' in \cite{S14}.}
    \label{fig_taquin}
\end{figure}

\begin{rem}
    Our results also have some connections to \textit{graphon} theory.
    This theory was introduced in \cite{LS06} for the study of dense graph sequences and is arguably the primary motivation for the introduction of permuton theory in \cite{HKMRS13} and \cite{GGKK15}.
    These two theories share several similarities as well as a formal link; the \textit{inversion graph} of permutations (see e.g.~Section 1.2 in \cite{BBDFGMP21} for a definition) can naturally be extended into a continuous map $\invt$ from permutons to graphons.
    The function $\LISt$ is then analogous to the \textit{independence number} $\widetilde{\alpha}$, which was proved to be upper semi-continuous in \cite{HR17}.
    Moreover the counterpart of \Cref{prop_conv_LISk} for $k=1$, \textit{i.e.}~the a.s.~convergence of $\alpha\big(\sample_n(W)\big)/n$ towards $\widetilde\alpha(W)$, was established in \cite{BBDFGMP21}.
    These convergence results allow to see that $\invt$ indeed maps $\LISt$ to $\widetilde\alpha$.
\end{rem}

\medskip

The results of this section are proved in \Cref{section_topo,section_shape_proofs}.
From now on we will focus on the study of $\LISt_k$, since results on $\LDSt_k$ can be deduced by mirroring the permuton:
\begin{equation*}
    \mu^\rev := \psi_*\mu
    \quad\text{where for any }(x,y)\in\carre,\quad
    \psi(x,y):=(1\mm x,y).
\end{equation*}
Since $\psi$ exchanges nonincreasing subsets with nondecreasing ones, we have for any $k\in\N^*$ and $(x,y)\in\carre$:
\begin{equation*}
    \LISt_k\left( \mu^\rev\vert_{[0,x]\times[0,y]} \right)
    = \LDSt_k\left( \mu\vert_{[1\m x,1]\times[0,y]} \right).
\end{equation*}

\subsection{Large deviation principles}\label{section_ldp_results}

In this section we present partial large deviation principles concerning the convergences of \Cref{prop_conv_LISk}. 
First we state an upper tail in \Cref{th_upper_tail}: the simple idea is that long increasing subsequences are formed when many points appear in a maximal increasing subset of the permuton. 
However this idea does not work for a lower tail, since having few points in a maximal increasing subset does not prevent long increasing subsequences from forming somewhere else: this difference of behavior is made formal in \Cref{prop_contre_exemple}. 
In the $\sqrt{n}$ regime when the permuton has a regular density, this asymmetry translates into a change of speed between the upper and lower tails \cite{DZ99}. 
Here in the linear regime, we prove in \Cref{th_speed_equiv} that if there exists a lower tail large deviation then the speed is the same as for the upper tail.

\medskip

For any $p\in[0,1]$, define the Legendre transform of the Bernoulli law ${\rm Ber}(p)$ as:
\begin{equation*}
    \text{for any }q\in[0,1],\quad
    \Lambda_p^*(q) 
    := q\log\frac{q}{p} + (1-q)\log\frac{1-q}{1-p} \quad\in [0,+\infty] .
\end{equation*}
Cramér's theorem (recalled in \Cref{th_Cramer}) states that this is the good rate function for the large deviation principle of binomial laws. 
We show that this is also the rate function for the following upper tail:

\begin{theorem}\label{th_upper_tail}
Let $\mu\in\per$, $k\in\N^*$, and $\sigma_n\sim\sample_n(\mu)$ for each $n\in\N^*$.
Then for any $\alpha \in \left( \LISt_k(\mu),1 \right]$:
\begin{equation*}
    \frac{1}{n}\log\P\big(\LIS_k(\sigma_n) \geq \alpha n\big) 
    \cv{n}{\infty} -\Lambda_{\LISt_k(\mu)}^*(\alpha).
\end{equation*}
\end{theorem}

In particular when $\alpha=1$ and $k=1$ we get the following convergence:

\begin{cor}\label{cor_puissance_1/n}
    For any permuton $\mu$ and $\sigma_n\sim\sample_n(\mu)$:
    \begin{equation*}
        \P( \sigma_n=\id_n )^{1/n} \cv{n}{\infty} \LISt(\mu).
    \end{equation*}
\end{cor}

For the lower tail we restrict our study to $\LIS$ for convenience, but this could be generalized to $\LIS_k$. 
We begin with an upper bound in the same flavor as in \Cref{th_upper_tail}:

\begin{prop}\label{prop_lower_tail}
Let $\mu\in\per$ and $\sigma_n\sim\sample_n(\mu)$ for each $n\in\N^*$. 
Then for any $\beta\in\left( 0,\LISt(\mu) \right)$:
\begin{equation*}
    \limsup_{\as{n}{\infty}}\frac{1}{n}\log\P\big( 
    \LIS(\sigma_n) < \beta n
    \big)
    \;\leq\; -\! \sup_{j\geq1,\, j\beta<\LISt_j(\mu)}
    \Lambda_{\LISt_j(\mu)}^*(j\beta).
\end{equation*}
\end{prop}

This upper bound is in general better than the simpler one $\Lambda_{\LISt(\mu)}^*(\beta)$; take for example $\mu$ such that $\LISt_2(\mu)=2\LISt(\mu)>0$.
Since it is non-zero whenever $\LISt(\mu)>0$, it implies that the $n$ speed for a lower tail is not ``too fast''.
Conversely the following result states that, except in trivial cases, this speed is not ``too slow'' either:

\begin{theorem}\label{th_speed_equiv}
    Let $\mu\in\per$ and $\sigma_n\sim\sample_n(\mu)$ for each $n\in\N^*$.
    Let $\beta\in (0,1]$ and $k\in\N^*$ satisfying $\frac{1}{k\p1} < \beta \leq \frac{1}{k}$.
    Then the following assertions are equivalent:
    \begin{itemize}
        \item[(i)] $\liminf \frac{1}{n}\log\P\big(\LIS(\sigma_n) < \beta n\big) = -\infty$ ;
        \item[(ii)] For all $n\in\N^*$, $\P\big( \LIS(\sigma_{n}) \geq \beta n \big) = 1$;
        \item[(iii)] $\LISt_k(\mu) = 1$.
    \end{itemize}
\end{theorem}

Unfortunately the upper bound of \Cref{prop_lower_tail} is not sharp.
We actually show that a rate function for this lower tail can not be expressed as a function of $\shapet(\mu)$:

\begin{prop}\label{prop_contre_exemple}
There exist $\mu_1,\mu_2\in\per$ satisfying $\shapet(\mu_1)=\shapet(\mu_2)$ and
\begin{equation*}
    \limsup_{\as{n}{\infty}}\frac{1}{n}\log\P\big( 
    \LIS(\tau_n) < \beta n 
    \big) < \liminf_{\as{n}{\infty}}\frac{1}{n}\log\P\big(
    \LIS(\sigma_n) < \beta n 
    \big)
\end{equation*}
for some $\beta<\LISt(\mu_1)$, where $\sigma_n\sim\sample_n(\mu_1)$ and $\tau_n\sim\sample_n(\mu_2)$.
\end{prop}

We believe a lower tail rate function could be expressed with the help of $\RSt(\mu)$ or $\lamt^\mu$ defined in \Cref{section_tableaux_results}, but this question remains unanswered.
The results of this section are proved in \Cref{section_upper_proofs,section_lower_tail_proofs,section_lower_speed_proof}.

\begin{rem}
	A large deviation principle for the sequence of sampled permutations $(\sigma_n)$ was recently obtained in \cite[Theorem~1.6]{BDMW22}.
	By semi-continuity of $\LISt_k$ (\Cref{prop_semicon_LISt}), we can deduce large deviation inequalities for the sequence $\left(\LIS_k(\sigma_n)\right)$ by a variant of the contraction principle (the standard contraction principle carries large deviation results over \textit{continuous} functions, see e.g.~\cite{DZ98}). 
	This actually yields the same inequalities as if we apply the same ``semi-contraction principle'' directly to Sanov's theorem, the latter being done in \Cref{cor_Sanov}.
	We do not know whether these bounds are optimal are not.
	We do not use this approach for \Cref{th_upper_tail}, and instead opt for a self-contained proof.
	However, we use the lower tail bound of \Cref{cor_Sanov} in the proof of \Cref{th_speed_equiv}.
\end{rem}

\begin{rem}
    The counterpart of \Cref{cor_puissance_1/n} for graphons was proved in \cite{CKP22}, and it is likely that \Cref{th_upper_tail} would also hold in that context.
    However our study of the lower tail in \Cref{th_speed_equiv} seems rather specific to permutations, as it hinges on the fact that any permutation $\sigma$ of size $n$ satisfies $\LIS(\sigma)\LDS(\sigma)\geq n$.
    The analog of this property for graphs is not true in general; note that inversion graphs of permutations are \textit{perfect graphs}.
\end{rem}

\subsection{The RS-tableaux of (pre-)permutons}\label{section_tableaux_results}

By analogy with (\ref{def_RS_permutations}), define the RS-tableaux of any $\mu\in\meas$ as
\begin{equation*}
    \RSt(\mu) := \left( \lamt^{\mu}(1,\cdot) , \lamt^{\mu}(\cdot,1) \right)
\end{equation*}
where for any $(x,y)\in\carre$, we write $\mu\vert_{[0,x]\times[0,y]}$ for the measure obtained by restricting $\mu$ to $[0,x]\times[0,y]$ and:
\begin{equation*}
    \lamt^{\mu}(x,y) = 
    \left( \lamt^{\mu}_{k}(x,y) \right)_{k\in\N^*}
    := \shapet^\row\left( \mu\vert_{[0,x]\times[0,y]} \right)
    = \left( \LISt_k\left( \mu\vert_{[0,x]\times[0,y]} \right) - \LISt_{k\m1}\left( \mu\vert_{[0,x]\times[0,y]} \right) \right)_{k\in\N^*}
\end{equation*}
In the continuity of \Cref{prop_conv_LISk}, we deduce in \Cref{section_proof_cv_tableaux} linear asymptotics for the RS-tableaux of sampled permutations:

\begin{prop}\label{prop_lambda_cv}
    Let $\mu\in\per$ and for each $n\in\N^*$, $\sigma_n \sim \sample_n(\mu)$. 
    For each $k\in\N^*$, the function $\lamt^{\mu}_{k}$ is nondecreasing and $1$-Lipschitz in each variable.
    Moreover we have the following almost sure uniform convergence on $\carre$:
    \begin{equation*}
        \frac{1}{n}\lambda^{\sigma_n}_k\left( \lfloor\cdot\,n\rfloor , \lfloor\cdot\,n\rfloor \right) 
        \cv{n}{\infty}
        \lamt^{\mu}_k.
    \end{equation*}
\end{prop}

The RS-tableaux of a permutation can be constructed via an algorithm known as \textit{Fomin's local rules}, see \Cref{section_proofs_Fomin}.
Its usual version works by labeling the vertices of a grid with diagrams, as in \cite[Section 5.2]{S01}, but it is equivalently possible to label the edges of this grid with integers \cite{V18}.
This edge version is well adapted to our framework, since labels then correspond to increment indices of $\lambda^\sigma$.
The direct algorithm recursively constructs each $\lambda^\sigma(i\p1,j\p1)$ from $\lambda^\sigma(i\p1,j)$, $\lambda^\sigma(i,j\p1)$, $\lambda^\sigma(i,j)$ and $\sigma$.
Its reciprocal recursively constructs each $\lambda^\sigma(i,j)$ from $\lambda^\sigma(i\p1,j)$, $\lambda^\sigma(i,j\p1)$ and $\lambda^\sigma(i\p1,j\p1)$.

It is then natural to wonder if there exists a continuous version of this algorithm, translating as a  partial differential equation on $\lamt^\mu$.
It is indeed the case for the inverse local rules, as stated in the following result:

\begin{theorem}\label{th_derivative_per}
    Suppose $\mu$ is a permuton satisfying $\LISt_r(\mu)=1$ for some $r\in\N^*$.
    Let $(x,y)\in(0,1]^2$ be such that the left-derivatives
    \begin{equation*}
        \alpha_k := \partial^-_x \lamt_k^\mu(x,y)
        \quad\text{and}\quad
        \beta_k := \partial^-_y \lamt_k^\mu(x,y)
    \end{equation*}
    for $1\leq k\leq r$ all exist. Fix $s,t\geq0$.
    Then we have the following directional semi-derivatives for each $1\leq k\leq r$:
    \begin{equation}\label{eq_derivative_per}
        \underset{\epsilon \rightarrow 0^+}{\lim}
        \frac{\lamt_k^\mu(x,y) - \lamt_k^\mu(x-t\epsilon,y-s\epsilon)}{\epsilon}
        = \phi\big( (t\alpha_i)_{k\leq i\leq r} , (s\beta_i)_{k\leq i\leq r} \big)
    \end{equation}
    where $\phi$ is a continuous function defined in \Cref{prop_Fbot_real_words_and_continuity_of_phi}.
\end{theorem}

The proof of \Cref{th_derivative_per} comes in several steps and introduces methods that might be interesting on their own.
The general idea is to study Fomin's inverse local rules on $\sigma_n\sim\sample_n(\mu)$ and exhibit their asymptotic behavior, yielding properties of $\lamt^\mu$.
To this aim, we see sequences of edge labels as words, for which we introduce an appropriate equivalence relation.
Surprisingly enough, this equivalence relation is implied by the Knuth equivalence of words (see \Cref{section_proofs_Knuth} and \Cref{prop_Knuth_implique_Fomin}).
Along with the derivability condition of \Cref{th_derivative_per} this allows us to show in \Cref{section_proofs_diff} that edge labels of $\sigma_n$ behave, locally and asymptotically, as if they were ordered in a deterministic way.
This behavior yields the function $\phi$ introduced and studied in \Cref{section_proofs_phi}.
We refer to the discussion at the end of \Cref{section_proofs_Fomin} for more details.

\medskip

A few comments on this theorem are in order:
\begin{itemize}
    \item It is possible to generalize this result to pre-permutons, using the associated permuton and the chain rule.
    \item If $s=0$ then the left-hand side of \Cref{eq_derivative_per} is $t\alpha_k$ by definition, and indeed one can check with \Cref{eq_phi_one_coordinate,eq_phi_crossed_zero} that $\phi\big( (t\alpha_i)_{k\leq i\leq r} , (0) \big) = t\alpha_k$.
    By symmetry, an analogous remark holds when $t=0$.
    \item When $k=r$, the left-hand side of \Cref{eq_derivative_per} depends only on the ``last row derivatives'' $\alpha_r$ and $\beta_r$, and equals $\phi\big( (t\alpha_r) , (s\beta_r) \big) = \max(t\alpha_r,s\beta_r)$.
    When $k<r$, the $k$-th to $r$-th rows intervene in an intricate way through the function $\phi$.
    \item Unfortunately we have not been able to establish a continuous version of Fomin's \textit{direct} algorithm, as the dependence on permutation points makes the asymptotic study more difficult. 
    It seems likely that \Cref{th_derivative_per} could be generalized in this direction, but this question is left open.
    \item It would also be interesting to get rid of the hypothesis $\LISt_r(\mu)=1$; more precisely we would like to weaken it to $\LISt_\infty(\mu)=1$ with the notation of \Cref{section_injectivity_results}. 
    We can always assume that this is satisfied thanks to \Cref{th_decomposition}.
    However we could neither extend $\phi$ to infinite sequences nor make our techniques work when Fomin's edge labels are unbounded.
    \item \Cref{eq_derivative_per} can be related to the ``hydrodynamical approach'' used by Aldous and Diaconis in \cite{AD95} for the study of $\LIS$ in the uniform case. 
    When $\sigma_n$ is a uniform permutation of size $n$, Hammersley \cite{H72} studied the function $\lambda_1^{\sigma_n}$ (a very similar one actually) and used the subadditive ergodic theorem to obtain the asymptotic equivalence $\E[\LIS(\sigma_n)]\sim c\sqrt{n}$ for some constant $c>0$. 
    In \cite{AD95}, the authors investigated Hammersley's process in more depth, provided some heuristics as to why $c=2$, and proved it in this framework. 
    The general idea is that for any $y$, the counting process $x\mapsto \lambda_1^{\sigma_n}(x,y)$ should approximate a Poisson process. 
    Under this hypothesis, one can argue that $\frac{1}{\sqrt{n}}\E\left[\lambda_1^{\sigma_n}(x,y)\right]$ should asymptotically satisfy a PDE whose solution is $2\sqrt{xy}$. 
    Therefore \Cref{eq_derivative_per} can be interpreted as an analog of this, in the $n$ scaling limit for non-uniform permutations and for the first rows and columns of the shape. 
    It could be interesting to try this approach in the case of a permuton with ``regular'' density, in a context similar to that of \cite{DZ95} and \cite{S22}.
\end{itemize}

We end this section with an interesting consequence of \Cref{prop_lambda_cv} and \Cref{th_derivative_per}, proved in \Cref{section_proofs_diff}.

\begin{cor}\label{cor_gradient_lambda}
	Let $\mu$ be a permuton satisfying $\LISt_r(\mu)=1$ for some $r\in\N^*$.
	Then for all $1\le k\le r$, for almost every $(x,y)\in (0,1)^2$, $\lamt_k^\mu$ is differentiable at $(x,y)$ and satisfies $\partial_x \lamt_k^\mu(x,y) \cdot \partial_y \lamt_k^\mu(x,y) = 0$.
	In other words, the gradient of $\lamt_k^\mu$ is a.e.~directed along the $x$- or along the $y$-axis.
\end{cor}

\subsection{Injectivity of pre-permuton Robinson--Schensted}\label{section_injectivity_results}

Recall that Robinson--Schensted's correspondence is injective on permutations.
This begs the question of what information about a (pre-)permuton $\mu$ we can deduce from the knowledge of $\RSt(\mu)$.
To this end, we introduce a useful decomposition of pre-permutons.
Define $\LISt_\infty := \lim_{k\rightarrow\infty} \uparrow \LISt_k$ and $\LDSt_\infty := \lim_{k\rightarrow\infty} \uparrow \LDSt_k$ on $\meas$.

\begin{theorem}\label{th_decomposition}
Let $\mu\in\preper$. There exists a unique triple $\left(\mu^\incr,\mu^\decr,\mu^\sub\right) \in\meas^3$ satisfying:
\begin{equation*}
    \left\{
    \begin{array}{lll}
        \mu= \mu^\incr + \mu^\decr + \mu^\sub ;\smallskip\\
        \LISt_\infty\left(\mu^\incr\right)
        =\mu^\incr\left(\carre\right)
        =\LISt_\infty(\mu) ;\smallskip\\
        \LDSt_\infty\left(\mu^\decr\right)
        =\mu^\decr\left(\carre\right)
        =\LDSt_\infty(\mu).
    \end{array}
    \right.
\end{equation*}
Moreover the following properties hold:
\begin{itemize}
    \item[(i)] the measures $\mu^\incr, \mu^\decr, \mu^\sub$ are pairwise singular;
    \item[(ii)] for all $k\in\N^*$,
    $\LISt_k\left(\mu^\incr\right) = \LISt_k(\mu)$
    and $\LDSt_k\left(\mu^\decr\right) = \LDSt_k(\mu)$;
    \item[(iii)] $\LISt\left(\mu^\decr\right) = \LISt\left(\mu^\sub\right)
    = \LDSt\left(\mu^\sub\right) = \LDSt\left(\mu^\incr\right) = 0$ ;
    \item[(iv)] $\lamt^\mu=\lamt^{\mu^\incr}$.
\end{itemize}
\end{theorem}

Notice that Property (iv) of \Cref{th_decomposition} allows to lift the hypothesis $\LISt_r(\mu)=1$ in \Cref{th_derivative_per} to the more general one $\LISt_r(\mu)=\LISt_\infty(\mu)$.
Indeed, consider a permuton $\mu$ that satisfies $\LISt_r(\mu)=\LISt_\infty(\mu)$.
Its submeasure $\mu^\incr$ is \textit{a priori} not a permuton, but can be mapped to a pre-permuton $\delta\mu^\incr$ by dilatation.
This pre-permuton satisfies $\LISt_r(\delta\mu^\incr)=1$, so we can apply \Cref{th_derivative_per} using the chain rule.
Thanks to Property (iv) of \Cref{th_decomposition}, this finally yields a partial differential equation on $\lamt^\mu.$

\begin{conj}\label{conj_injectivity}
    Let $\mu,\nu \in\preper$ such that $\RSt(\mu) = \RSt(\nu)$. Then $\mu^\incr = \nu^\incr$.
\end{conj}

Recall that $\RSt(\mu)$ is defined using only the nondecreasing subsets of $\mu$. 
These ``row-tableaux'' capture no information about the submeasure $\mu^\decr$, since $\RSt(\mu^\decr)$ is null.
An analog of \Cref{conj_injectivity} for nonincreasing subsets could be stated by replacing $\mu$ with $\mu^\rev$ (see the discussion at the end of \Cref{section_shape_results}).

We believe that \Cref{th_derivative_per} could help establish \Cref{conj_injectivity}.
In most ``regular'' examples, the PDE of \Cref{th_derivative_per} seems to be satisfied everywhere on $(0,1]^2$ (note that $\lamt_k^\mu$ is in general not \textit{differentiable} everywhere).
We would then hope for the uniqueness of their solution $\lamt^\mu$ with boundary condition $\RSt(\mu)$, using \Cref{cor_gradient_lambda} and the form taken by \Cref{eq_derivative_per} at the points of non-differentiability of $\lamt_k^\mu$.
This is further discussed in \Cref{section_examples}, item \ref{ex_PDE_also_satisfied}.
However we could not make this idea work, and instead we present more elementary results: an injectivity property of $\mu\mapsto\lamt^\mu$, and a positive answer to \Cref{conj_injectivity} when the pre-permutons contain at most one nondecreasing subset.

\begin{prop}\label{prop_injectivity_lambda}
    Let $\mu,\nu \in\preper$ such that $\lamt^\mu = \lamt^\nu$.
    Then $\mu^\incr=\nu^\incr$.
\end{prop}

\begin{prop}\label{prop_injectivity_RS_one}
    Let $\mu,\nu\in\preper$ such that $\LISt_\infty(\mu) = \LISt_1(\mu)$ and $\RSt(\mu) = \RSt(\nu)$.
    Then $\mu^\incr = \nu^\incr$.
\end{prop}

The results of this section are proved in \Cref{section_proofs_injectivity}.

\subsection{Examples}\label{section_examples}

In this section, we illustrate our results with several examples.
Some of them are represented in \Cref{fig_ex_permutons}.
If $x,y\in\R$, we use the standard notation $x\wedge y := \min(x,y)$ and $x\vee y := \max(x,y)$.

\begin{figure}
	\centering
	\includegraphics[scale=.374]{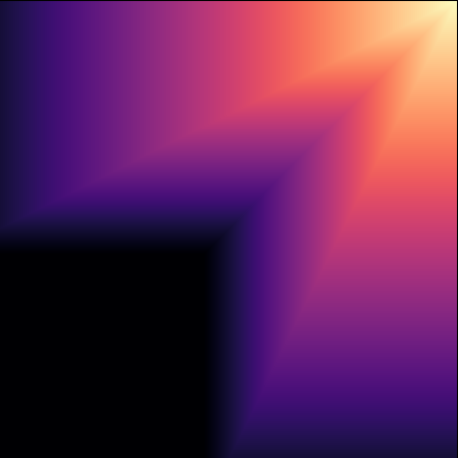}
	\qquad
	\includegraphics[scale=1]{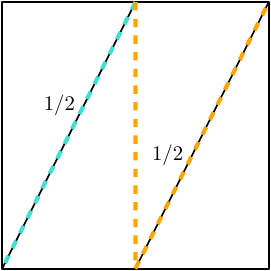}
	\qquad
	\includegraphics[scale=1]{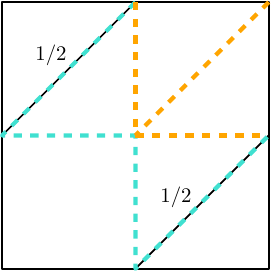}
	\caption{Representation of some of the examples of \Cref{section_examples}.
	From left to right: the locally affine function $f:\carre\to[0,1]$ constructed in \ref{ex_PDE_also_satisfied} (brighter colors indicate higher values), the permuton of \ref{ex_two_lines_first}, and the permuton of \ref{ex_two_lines_second}.
	For the last two: $\lamt_1^\mu$ is not differentiable on the dashed blue lines, and $\lamt_2^\mu$ is not differentiable on the dashed orange lines.}
	\label{fig_ex_permutons}
\end{figure}

\begin{enumerate}[label=(\alph*)]
	\item Let $\Leb_{\carre}$ be the uniform permuton, \textit{i.e.}~the Lebesgue measure on $\carre$.
	It is easy to check that for any nondecreasing $A$, we have $\Leb_{\carre}(A)=0$.
	Thus $\LISt\left(\Leb_{\carre}\right)=0$, and likewise $\LDSt\left(\Leb_{\carre}\right)=0$.
	This is consistent with the fact that the longest monotone subsequences of uniform permutations have sublinear asymptotics.
	
	More generally, let $\mu$ be a pre-permuton with a density with respect to $\Leb_{\carre}$.
	By absolute continuity, we again have $\LISt(\mu)=\LDSt(\mu)=0$.
	In particular, with the notation of \Cref{th_decomposition}: $\left(\mu^\incr,\mu^\decr,\mu^\sub\right) = (0,0,\mu)$.
	Our definitions are vacuous in this case, and our techniques are not suited to the study of such pre-permutons.
	\item\label{ex_permuton_diagonal} Let $\mu=\Leb_\Delta$ be the permuton that puts uniform mass on the diagonal $\Delta$ of $\carre$.
	Then $\LISt\left(\mu\right)=1$ and $\LDSt\left(\mu\right)=0$, and in particular $\shapet\left(\mu\right) = \big((1,0,\dots) , (0,\dots)\big)$.
	Also, $\lamt_1^\mu(x,y) = x\wedge y$ for any $(x,y)\in\carre$, and $\lamt_k^\mu(x,y)=0$ for all $k\ge2$.
	The RS-tableaux of $\mu$ are $\RSt(\mu) = \big( (\id_{[0,1]},0,\dots) , (\id_{[0,1]},0,\dots) \big)$.
	\item Let $\mu$ be the permuton that puts uniform mass on the \textit{antidiagonal} of $\carre$.
	Then $\LISt\left(\mu\right)=0$, $\LDSt\left(\mu\right)=1$, but with our definition $\RSt(\mu) = \big( (0,\dots) , (0,\dots) \big)$.
	This is because in this paper, we defined the RS-tableaux of a permuton via its \textit{row-shapes}, and not its \textit{column-shapes}.
	As discussed at the end of \Cref{section_shape_results}, it would be more relevant in this case to apply our definitions to the mirrored permuton $\mu^\rev=\Leb_\Delta$.
	\item\label{ex_two_lines_first} Let $\mu$ be the only permuton with support $\{y=2x\} \cup \{y=2x-1\}$ in $\carre$.
	This permuton has row-shape $\shapet^\row(\mu) = (1/2,1/2,0,\dots)$ and trivial column-shape.
	In particular $\mu=\mu^\incr$.
	We can compute:
	\begin{align*}
		\lamt_1^\mu(x,y) = 
		\left\{
		\begin{array}{ll}
			x \wedge (y/2) &\text{if } x\in[0,1/2];
			\\y/2 &\text{if } x\in[1/2,1];
		\end{array}
		\right.
	\end{align*}
	and 
	\begin{align*}
		\lamt_2^\mu(x,y) = 
		\left\{
		\begin{array}{ll}
			0 &\text{if } x\in[0,1/2];
			\\(x-1/2) \wedge (y/2) &\text{if } x\in[1/2,1].
		\end{array}
		\right.
	\end{align*}
	The function $\lamt_1^\mu$ is not differentiable on $\{y=2x\}$, and $\lamt_2^\mu$ is not differentiable on $\{x=1/2\} \cup \{y=2x-1\}$.
	\item\label{ex_two_lines_second} Let $\mu$ be the only permuton with support $\{y=x+1/2\} \cup \{y=x-1/2\}$ in $\carre$.
	This permuton has the same RS-shape as the previous one, but different RS-tableaux.
	We can compute:
	\begin{align*}
		\lamt_1^\mu(x,y) = 
		\left\{
		\begin{array}{llll}
			0 &\text{if } (x,y)\in[0,1/2]^2;
			\\x \wedge (y-1/2) &\text{if } (x,y)\in[0,1/2]\times[1/2,1];
			\\(x-1/2) \wedge y &\text{if } (x,y)\in[1/2,1]\times[0,1/2];
			\\(x-1/2) \vee (y-1/2) &\text{if } (x,y)\in[1/2,1]^2;
		\end{array}
		\right.
	\end{align*}
	and 
	\begin{align*}
		\lamt_2^\mu(x,y) = 
		\left\{
		\begin{array}{ll}
			0 &\text{if } (x,y)\notin[1/2,1]^2;
			\\(x-1/2) \wedge (y-1/2) &\text{if } (x,y)\in[1/2,1]^2.
		\end{array}
		\right.
	\end{align*}
	The function $\lamt_1^\mu$ is not differentiable on $\{x\le y, y=1/2\} \cup \{y=x+1/2\} \cup \{y\le x, x=1/2\} \cup \{y=x-1/2\} \cup \{y=x, y\ge 1/2\}$, and $\lamt_2^\mu$ is not differentiable on $\{x\le y, x=1/2\} \cup \{y\le x, y=1/2\} \cup \{y=x, y\ge 1/2\}$.
\end{enumerate}
Now let us illustrate \Cref{th_decomposition} with some examples of permutons having non-trivial decompositions.
\begin{enumerate}[label=(\alph*),resume]
	\item The permuton $\mu_{\alpha,\beta,\gamma}$ defined in \Cref{rem_taquin} has RS-shape $(\alpha,\beta)$ by construction.
	The decomposition of \Cref{th_decomposition} is what one would expect: $\mu_{\alpha,\beta,\gamma}^\incr$ is the restriction of $\mu_{\alpha,\beta,\gamma}$ to the upward diagonals of its support, $\mu_{\alpha,\beta,\gamma}^\decr$ is its restriction to the downward diagonals, and $\mu_{\alpha,\beta,\gamma}^\sub$ is the remaining uniform mass.
	\item Let $f:[0,1]\to[0,1]$ and define a measure $\mu_f$ as the push-forward of $\Leb_{[0,1]}$ by the mapping $x\mapsto (x,f(x))$.
	If $f$ preserves the Lebesgue measure then $\mu_f$ is a permuton, called a \textit{push-forward permuton} in \cite{FR-L23}.
	A few examples of push-forward permutons are the permutons \ref{ex_permuton_diagonal}, \ref{ex_two_lines_first}, \ref{ex_two_lines_second} here; the skew Brownian permutons \cite{B23}; and the recursive separable permutons \cite{FR-L23}.
	The RS-shapes and tableaux of push-forward permutons are not easy to compute in general, and may even be trivial (this is e.g.~the case for the Brownian separable permuton, by \cite[Theorem 1.10]{BBDFGMP21}).
	However if $f$ is a smooth function (not necessarily Lebesgue-preserving) such that the set $\{f'=0\}$ has empty interior in $[0,1]$, then $\mu_f$ is a pre-permuton with a simple decomposition:
	$\mu_f^\incr$ is the restriction of $\mu_f$ to $\{(x,f(x)) : f'(x)>0\}$ and $\mu_f^\decr$ is the restriction of $\mu_f$ to $\{(x,f(x)) : f'(x)<0\}$.
	
	Other examples of ``unidimensional'' pre-permuton can be constructed by putting uniform masses on unions of curves in $\carre$, under the assumption that those curves do not have any horizontal or vertical parts.
	\item Several permutons with non-trivial decompositions have been observed in the recent literature, and we list some of them here.
	For the runsort permuton $\mathbf{R}$ found in \cite{ADK22}, $\mathbf{R}^\incr$ is supported on a smooth increasing curve and $\mathbf{R}^\sub$ has a density above this curve.
	The limit of square permutations \cite{BS20} satisfies $\mu^\sub=0$, and its (random) RS-shape has two rows and two columns.
	Almost square permutations with a sublinear number of internal points have a similar limit \cite{BDS21}, and \cite[Conjecture~5.1.4]{B21} suggests that a linear number of internal points yields a permuton with non-trivial components $\mu^\sub$ (uniform mass inside some domain), $\mu^\incr$ and $\mu^\decr$ (supported on the boundary of this domain).
	A similar observation can be made for record-biased permutations \cite{ABNP16}: simulations suggest that, in an adequate regime, there is a permuton limit $\mu$, with $\mu^\incr$ being supported on an increasing curve and $\mu^\sub$ attributing uniform mass under this curve.
\end{enumerate}
This last example further discusses a potential proof strategy for \Cref{conj_injectivity}, mentioned in \Cref{section_injectivity_results}.
\begin{enumerate}[label=(\alph*),resume]
	\item\label{ex_PDE_also_satisfied} Take again $\mu=\Leb_\Delta$.
	\Cref{conj_injectivity} would imply that this is the only (pre-)permuton with RS-tableaux $\RSt(\mu) = \big( (\id_{[0,1]},0,\dots) , (\id_{[0,1]},0,\dots) \big)$ (this is actually true thanks to \Cref{prop_injectivity_RS_one}, but let us ignore this for the sake of the example).
	As mentioned earlier, a plausible approach would be to solve the PDE \eqref{eq_derivative_per} with boundary condition $\RSt(\mu)$.
	In this example, we aim to illustrate that this approach should rely on the behavior of $\lamt^\mu$ at its points of non-differentiability.
	More precisely, $\lamt_1^\mu(x,y)=x\wedge y$ is not the only function $f:\carre\to[0,1]$ such that:
	\begin{itemize}
	\item $f$ is Lipschitz and nondecreasing in both variables;
	\item $f$ is differentiable a.e.~and satisfies $D_{(x,y)}f(s,t) = \phi\left(\big. (t\partial_xf(x,y)) , (s\partial_yf(x,y)) \right)$ at its points of differentiability (this equation is equivalent to $\partial_xf(x,y) \cdot \partial_yf(x,y)=0$ at its points of differentiability, as proved in \Cref{cor_gradient_lambda});
	\item $f(1,\cdot)=f(\cdot,1)=\id_{[0,1]}$.
	
	\end{itemize}
	To show this, construct $f$ as follows:
	for any $z\in[0,1]$, $f$ takes the value $z$ on the broken line which goes from $(z,1)$ to $(z,\frac{z+1}{2})$, to $(\frac{z+1}{2},\frac{z+1}{2})$, to $(\frac{z+1}{2},z)$, to $(1,z)$.
	Also, set $f=0$ on $[0,1/2]^2$.
	Then $f$ is differentiable outside $\{y=\frac{x+1}{2}\} \cup \{y=x, x\ge1/2\} \cup \{x=\frac{y+1}{2}\} \cup \{y=1/2,x\le 1/2\} \cup \{x=1/2,y\le 1/2\}$, and on this domain it is locally either a function of $x$, or a function of $y$.
	Hence $f$ satisfies a.e.~the same PDE as $\lamt_1^\mu$, and the same boundary condition.
	See \Cref{fig_ex_permutons} for a representation.

\end{enumerate}

\section{Proofs of the results in Sections \ref{section_shape_results} and \ref{section_ldp_results}}

\subsection{Topological preliminaries}\label{section_topo}

Write $\d$ for the Euclidean distance on the plane.
If $A\subseteq\carre$ and $\delta>0$, define the $\delta$-halo of $A$ by:
\begin{equation*}
    A^\delta := \big\{ y\in\carre \::\: \d(y,x)\leq\delta \text{ for some } x\in A \big\}.
\end{equation*}
Note that if $A$ is a closed subset of $\carre$ then $A^\delta$ is closed as well.
We then define the Hausdorff distance by
\begin{equation*}
    \text{for all }A,B\subseteq\carre,\quad
    \dH(A,B) := \inf\left\{\delta>0 \::\: A\subseteq B^\delta \text{ and } B\subseteq A^\delta \right\}.
\end{equation*}
It is well known that the space $\mathcal{P}_K(\carre)$ of closed subsets of the unit square, equipped with the distance $\dH$, is compact.
Moreover, one can easily see that its subsets $\mathcal{P}_K^\nearrow(\carre)$ are $\mathcal{P}_K^\searrow(\carre)$ are closed, and hence compact.
The following two lemmas may be well known, but we include a short proof.

\begin{lemma}\label{lem_mesure_mesurable}
Let $\mu\in\meas$. The function $A \in \mathcal{P}_K(\carre) \mapsto \mu(A) \in [0,1]$
is upper semi-continuous for the Hausdorff distance.
\end{lemma}

\begin{proof}
Let $(A_n)_{n\in\N}$ be a sequence converging to $A$ in $\mathcal{P}_K(\carre)$. Our goal is to establish $\mu(A) \geq \limsup\mu(A_n)$.
For this, fix $\varepsilon$ > 0. 
By monotony and since $\bigcap_{\delta>0}A^\delta = A$, there exists $\delta > 0$ satisfying $\mu(A^\delta) \leq \mu(A)+\varepsilon$.
Let $n_0$ be such that for any $n$ greater than $n_0$, $A^\delta$ contains $A_n$. 
Then:
\begin{equation*}
     \limsup_{\as{n}{\infty}}\mu(A_n) \leq \mu(A^\delta) \leq \mu(A) + \varepsilon
\end{equation*}
which concludes this proof, as it holds for any $\varepsilon>0$.
\end{proof}

\begin{lemma}\label{lem_union_con}
For any $k\in\N^*$, the union map $\cup$ from $\mathcal{P}_K(\carre)^k$ to $\mathcal{P}_K(\carre)$ is continuous. 
\end{lemma}

\begin{proof}
By associativity, it suffices to consider the case $k=2$. 
Let $(A_n), (B_n)$ be two sequences converging to $A,B$ respectively in $\mathcal{P}_K(\carre)$. 
Our goal is to prove the convergence of $(A_n\cup B_n)$ to $A\cup B$. 
For this, fix $\varepsilon>0$ and take $n_0$ such that:
\begin{equation*}
    \text{for all } n>n_0 ,\quad A_n\subseteq A^\varepsilon 
    \;\text{ and }\; A\subseteq A_n^\varepsilon 
    \;\text{ and }\; B_n\subseteq B^\varepsilon 
    \;\text{ and }\; B\subseteq B_n^\varepsilon .
\end{equation*}
Hence:
\begin{equation*}
    \text{for all } n>n_0 ,\quad 
    A_n\cup B_n\subseteq (A^\varepsilon) \cup (B^\varepsilon)
    \;\text{ and }\; A\cup B \subseteq (A_n^\varepsilon) \cup (B_n^\varepsilon).
\end{equation*}
As it is immediate to check the equality $(A^\varepsilon)\cup(B^\varepsilon)=(A\cup B)^\varepsilon$ for any closed subsets of $\carre$, this concludes the proof.
\end{proof}

\begin{cor}\label{lem_compact_k}
For any $k\in\N^*$, the space $\mathcal{P}_K^{k\nearrow}(\carre)$ equipped with the distance $\dH$ is compact.
\end{cor}

Now we have all the necessary tools to prove \Cref{prop_semicon_LISt}.

\begin{proof}[Proof of \Cref{prop_semicon_LISt}]
Let $\mu\in\meas$ and $k\in\N^*$.
According to \Cref{lem_mesure_mesurable}, the map $\mu$ defined from $\mathcal{P}_K^{k\nearrow}(\carre)$ to $[0,1]$ is upper semi-continuous. 
It is well-known that any upper semi-continuous function defined on a compact set reaches its maximum. 
Hence, by \Cref{lem_compact_k}, the definition of $\LISt_k(\mu)$ is actually a maximum.

Now let us turn to the semi-continuity of $\LISt_k$ on the space $\prob$. Let $(\mu_n)_{n\in\N}$ be a sequence of probability measures weakly converging to some $\mu$.
We aim to prove:
\begin{equation}\label{ine_limsup_LISt_k}
    \LISt_k(\mu) \geq \limsup_{\as{n}{\infty}}\LISt_k(\mu_n).
\end{equation}
Up to extraction, suppose the right hand side is an actual limit.
For each $n\in\N$, fix $A_n\in\mathcal{P}_K^{k\nearrow}(\carre)$ such that $\LISt_k(\mu_n) = \mu_n(A_n)$.
Once again using \Cref{lem_compact_k}, up to extraction (which does not change the right hand side limit in (\ref{ine_limsup_LISt_k})), suppose the sequence $(A_n)_{n\in\N}$ converges to some $A$ in $\mathcal{P}_K^{k\nearrow}(\carre)$.
Fix an arbitrary $\varepsilon>0$. 
Recalling the definition of $\dH$, we have:
\begin{equation*}
    \lim_{\as{n}{\infty}} \LISt_k(\mu_n)
    = \lim_{\as{n}{\infty}} \mu_n(A_n)
    \leq \limsup_{\as{n}{\infty}} \mu_n\big( A^\varepsilon \big)
    \leq \mu\big( A^\varepsilon \big)
\end{equation*}
where in the fourth inequality we used Portmanteau theorem and the fact that $A^\varepsilon$ is closed. 
Since this holds for any $\varepsilon>0$ and $\mu\big( A^\varepsilon \big) \cv{\varepsilon}{0} \mu(A)$, we deduce
\begin{equation*}
    \lim_{\as{n}{\infty}} \LISt_k(\mu_n) \leq \mu(A) \leq \LISt_k(\mu).
\end{equation*}
This concludes the proof of (\ref{ine_limsup_LISt_k}), and thus of \Cref{prop_semicon_LISt}. 
The study of $\LDSt_k$ is completely similar.
\end{proof}

\subsection{Convergence of shapes}\label{section_shape_proofs}

To prove \Cref{prop_conv_LISk}, we associate each permutation $\sigma\in\S_n$ with a permuton that is appropriate to the study of longest increasing subsequences. 
The permuton $\mu_\sigma$ defined in \Cref{section_intro_permutons} satisfies $\LISt(\mu_\sigma)=0$, so we need to introduce a different one.
If $a\leq b$ are integers, we use the notation $\lb a,b\rb$ for the set $\{a,a+1\dots,b\}$.
For any $i,j\in\lb1,n\rb$, let
\begin{equation*}
    C_{i,j} := \left[\frac{i-1}{n},\frac{i}{n}\right]\times\left[\frac{j-1}{n},\frac{j}{n}\right]
\end{equation*}
and $\Delta_{i,j}$ be the diagonal of $C_{i,j}$.
Then define a permuton by putting uniform mass on these diagonals:
\begin{equation*}
    \mu_\sigma^\nearrow := \frac{1}{n}\sum_{i=1}^n \Leb_{\Delta_{i,\sigma(i)}}
\end{equation*}
where $\Leb_{\Delta_{i,\sigma(i)}}$ denotes the normalized Lebesgue probability measure on $\Delta_{i,\sigma(i)}$.
Since the Wasserstein distance between $\mu_\sigma^\nearrow$ and $\mu_\sigma$ is a $O(1/n)$, \Cref{th_HKMRS} implies:
\begin{lemma}\label{lem_HKMRS_increasing}
    Let $\mu\in\per$ and $\sigma_n\sim\sample_n(\mu)$ for each $n\in\N^*$. Then $\mu_{\sigma_n}^\nearrow \cv{n}{\infty} \mu$ almost surely.
\end{lemma}
Notice that this almost sure convergence holds regardless of the joint distribution of $\sigma_n$'s, as this is the case in \Cref{th_HKMRS}.

Instead of these permutons, we could use the empirical measures of i.i.d.~points distributed under $\mu$ (see \Cref{section_lower_speed_proof}) to prove \Cref{prop_conv_LISk}.
One advantage of introducing $\mu_\sigma^\nearrow$ is to also obtain an appropriate embedding of the set of permutations into the space of permutons.
This embedding is adapted to the study of longest increasing subsequences, as stated in the following lemma:
\begin{lemma}\label{lem_lien_lis_list}
Let $\sigma\in\S_n$. Then for any $k\in\N^*$:
\begin{equation*}
    \LISt_k(\mu_\sigma^\nearrow) = \frac{\LIS_k(\sigma)}{n}.
\end{equation*}
\end{lemma}

\begin{proof}
First, let $I_1,\dots,I_k$ be increasing subsequences of $\sigma$ such that $\LIS_k(\sigma) = \vert I_1\cup\dots\cup I_k\vert$.
Then define for each $j\in\lb1,k\rb$:
\begin{equation*}
    A_j := \bigcup_{i\in I_j}\Delta_{i,\sigma(i)}.
\end{equation*}
Since $I_j$ is an increasing subsequence of $\sigma$, $A_j$ is in $\mathcal{P}_K^\nearrow(\carre)$. 
As $\mu_\sigma^\nearrow$ puts mass $1/n$ on each diagonal of its support, one has:
\begin{equation}\label{premiere_ine_lis_list}
    \frac{\LIS_k(\sigma)}{n}=\frac{\vert I_1\cup\dots\cup I_k\vert }{n} =
    \mu_\sigma^\nearrow(A_1\cup\dots\cup A_k)
    \leq\LISt_k(\mu_\sigma^\nearrow).
\end{equation}
Next, using \Cref{prop_semicon_LISt}, let $B_1,\dots,B_k$ be closed nondecreasing subsets of $\carre$ such that:
\begin{equation*}
    \LISt_k(\mu_\sigma^\nearrow) = \mu_\sigma^\nearrow(B_1\cup\dots\cup B_k).
\end{equation*}
Then define for each $i\in\lb1,k\rb$:
\begin{equation*}
    J_i := \{ j\in\lb1,n\rb \::\: B_i\cap \Delta_{j,\sigma(j)}\neq\emptyset \}.
\end{equation*}
Since $B_i$ is nondecreasing, $J_i$ is an increasing subsequence of $\sigma$. 
Hence:
\begin{equation}\label{seconde_ine_lis_list}
    \LISt_k(\mu_\sigma^\nearrow) = \mu_\sigma^\nearrow(B_1\cup\dots\cup B_k)
    \leq \frac{\vert J_1\cup\dots\cup J_k\vert }{n}\leq\frac{\LIS_k(\sigma)}{n}.
\end{equation}
The lemma follows from \Cref{premiere_ine_lis_list,seconde_ine_lis_list}.
\end{proof}

\medskip

\begin{proof}[Proof of \Cref{prop_conv_LISk}]
First notice that if $\mu$ is a pre-permuton, its corresponding permuton $\hat{\mu}$ satisfies $\shape\left( \hat{\mu} \right) = \shape(\mu)$ and for each $n\in\N^*$, $\sample_n\left( \hat{\mu} \right) = \sample_n(\mu)$ (see e.g.~Remark 1.2 in \cite{BDMW22}).
Thus we can assume for this proof that $\mu$ is a permuton.

On the one hand, using \Cref{lem_lien_lis_list,lem_HKMRS_increasing} and \Cref{prop_semicon_LISt}:
\begin{equation}\label{conv_LISk_upper}
    \underset{\as{n}{\infty}}{\limsup}\;
    \frac{1}{n}\LIS_k(\sigma_n)
    = \underset{\as{n}{\infty}}{\limsup}\;
    \LISt_k\left( \mu_{\sigma_n}^\nearrow\right)
    \leq \LISt_k(\mu) 
    \quad\text{almost surely.}
\end{equation}
On the other hand, using \Cref{prop_semicon_LISt}, let $A$ be a closed $k$-nondecreasing subset of $\carre$ such that $\LISt_k(\mu)=\mu(A)$.
Let $\sigma_n := \perm\left(Z_1,\dots,Z_n\right)$ where $Z_1,\dots,Z_n$ are random i.i.d.~points distributed under $\mu$, and define $S_n$ as the number of $Z_i$'s lying in $A$. 
Since $A$ is $k$-nondecreasing, those $S_n$ points can be written as a union of $k$ nondecreasing subsets. Hence:
\begin{equation}\label{points_k_croiss}
    \LIS_k( \sigma_n ) \geq S_n.
\end{equation}
However $S_n$ follows a binomial law of parameter $\left(n,\mu(A)\right)$, so the law of large numbers imply 
\begin{equation}\label{lgn_points}
    \frac{S_n}{n} \cv{n}{\infty} \mu(A) = \LISt_k(\mu) \quad\text{almost surely.}
\end{equation}
We point out the fact that this law of large numbers comes along with exponential concentration inequalities, so almost sure convergence holds regardless of the joint distribution of $S_n$'s.
Using (\ref{points_k_croiss}) and (\ref{lgn_points}), we get:
\begin{equation}\label{conv_LISk_lower}
    \underset{\as{n}{\infty}}{\liminf}\;
    \frac{1}{n}\LIS_k(\sigma_n)
    \geq \LISt_k(\mu) 
    \quad\text{almost surely}.
\end{equation}
Inequalities (\ref{conv_LISk_upper}) and (\ref{conv_LISk_lower}) conclude the proof for $\LIS_k$.
The proof of the analogous convergence for $\LDS_k$ works the same, replacing $\mu_\sigma^\nearrow$ with a similar permuton $\mu_\sigma^\searrow$ putting mass on antidiagonals.
\end{proof}

To conclude this section, we explain how to derive \Cref{prop_Thoma} from \Cref{prop_conv_LISk}. 
While the quantity $\shape(\sigma)$ for $\sigma\in\S_n$ does not quite belong to the Thoma simplex, this becomes the case in the $\as{n}{\infty}$ limit:

\begin{lemma}\label{lem_asymp_thom}
For each $n\in\N^*$ consider a permutation $\sigma_n$ of size $n$, and define for all $k\in\N^*$:
\begin{equation*}
    \alpha_k := \limsup_{\as{n}{\infty}}\frac{\LIS_k(\sigma_n)-\LIS_{k\m1}(\sigma_n)}{n}
    \;\text{ and }\;
    \beta_k := \limsup_{\as{n}{\infty}}\frac{\LDS_k(\sigma_n)-\LDS_{k\m1}(\sigma_n)}{n}.
\end{equation*}
Then the sequences $\alpha:=(\alpha_k)_{k\geq1}$ and $\beta:=(\beta_k)_{k\geq1}$ satisfy $(\alpha,\beta) \in \Omega$.
\end{lemma}

\begin{proof}
The monotony condition is a consequence of \Cref{th_Greene}. 
To prove the sum condition, we fix an arbitrary $k\in\N^*$ and show the following:
\begin{equation}\label{ine_k_fixe}
    \sum_{i= 1}^k \alpha_i +
    \sum_{j= 1}^k \beta_j \leq 1.  
\end{equation}
For each $n\in\N^*$, denote by $A_{k,n}$ the number of boxes shared by the first $k$ rows and first $k$ columns of $\sigma_n$'s RS-shape.
By \Cref{th_Greene}, $\LIS_k(\sigma_n)$ and $\LDS_k(\sigma_n)$ are respectively the numbers of boxes in the first $k$ rows and columns of this diagram.
Therefore by inclusion-exclusion:
\begin{equation*}
    \LIS_k(\sigma_n) + \LDS_k(\sigma_n)
    \leq n + A_{k,n}.
\end{equation*}
Moreover $A_{k,n} / n \leq k^2 / n \cv{n}{\infty} 0$, so inequality (\ref{ine_k_fixe}) follows.
Since this holds for any integer $k$, the lemma is proved.
\end{proof}

Finally, \Cref{prop_Thoma} is a direct consequence of \Cref{prop_conv_LISk} and \Cref{lem_asymp_thom}.

\subsection{Upper tail large deviation}\label{section_upper_proofs}

The proof of \Cref{th_upper_tail} heavily hinges on the following celebrated result:

\begin{theorem}[\cite{C38,DZ98}]\label{th_Cramer}
Fix $p\in[0,1]$ and consider for each $n\in\N^*$ a random variable $S_n$ following a binomial law of parameter $(n,p)$. Then for any $q\in(p,1]$:
\begin{equation*}
    \frac{1}{n}\log\P\left( S_n \geq nq \right) \cv{n}{\infty} -\Lambda_p^*(q)
\end{equation*}
with the notation of \Cref{section_ldp_results}. Likewise for any $q\in[0,p)$:
\begin{equation*}
    \frac{1}{n}\log\P\left( S_n < nq \right) \cv{n}{\infty} -\Lambda_p^*(q).
\end{equation*}
\end{theorem}

\medskip

\begin{proof}[Proof of \Cref{th_upper_tail}]
If $\LISt_k(\mu)=1$, there is nothing to prove. Now suppose $\LISt_k(\mu)<1$. 
To prove the lower bound, let $A$ be a closed $k$-nondecreasing subset of $\carre$ such that $\LISt_k(\mu)=\mu(A)$.
Let $\sigma_n := \perm\left(Z_1,\dots,Z_n\right)$ where $Z_1,\dots,Z_n$ are random i.i.d.~points distributed under $\mu$, and define $S_n$ as the number of $Z_i$'s lying in $A$. 
Since $S_n$ follows a binomial law of parameter $\big( n,\LISt_k(\mu) \big)
$, by \Cref{th_Cramer} and inclusion of events:
\begin{equation*}
    \liminf_{\as{n}{\infty}}\frac{1}{n}
    \log\P\big( \LIS_k(\sigma_n)\geq \alpha n \big) 
    \geq \liminf_{\as{n}{\infty}}\frac{1}{n}
    \log\P(S_n\geq \alpha n)
    \geq -\Lambda_{\LISt_k(\mu)}^*(\alpha).
\end{equation*}

\medskip

Next, let us prove the upper bound. 
The general idea is that if sampled permutations have a large $\LIS_k$ then a lot of points were sampled in a large $k$-increasing subset of $\mu$.
To find such a subset and make it deterministic, we condition on this rare event and extract a limit.
Define
\begin{equation*}
    x := -\limsup_{\as{n}{\infty}}\frac{1}{n}
    \log\P\big( \LIS_k(\sigma_n)\geq \alpha n \big)
    \quad \in[0,+\infty]
\end{equation*}
and suppose, up to extraction, that this $\limsup$ is an actual limit. 
Our goal is to prove 
\begin{equation}\label{ine_goal_upper}
    x \geq \Lambda_{\LISt_k(\mu)}^*(\alpha).
\end{equation}
If $x=+\infty$ this is trivial.
Now suppose $x<+\infty$, which implies the existence of $n_0\in\N^*$ such that for all $n\geq n_0$, $\P\big( \LIS_k(\sigma_n)\geq \alpha n \big) > 0$.
Thus for any $n\geq n_0$, we can consider random variables $Z_1^{(n)},\dots,Z_n^{(n)}$ i.i.d.~under $\mu$ conditioned to satisfy
\begin{equation*}
    \LIS_k\left( Z_1^{(n)},\dots,Z_n^{(n)} \right)
    \geq \alpha n,
\end{equation*}
where the notation $\LIS_k\left( Z_1^{(n)},\dots,Z_n^{(n)} \right)$ is shortcut for $\LIS_k\left( \perm\left( Z_1^{(n)},\dots,Z_n^{(n)} \right) \right)$.
Let $A_n$ be the maximal $k$-nondecreasing subset of $\left\{Z_1^{(n)},\dots,Z_n^{(n)}\right\}$ (if there are several such subsets, choose one with an arbitrary rule). 
Then denote by $\nu_n$ the law of $A_n$ in $\mathcal{P}_K^{k\nearrow}(\carre)$. 
Up to extraction, the sequence $(\nu_n)_{n\geq n_0}$ weakly converges to some probability measure $\nu$ on $\mathcal{P}_K^{k\nearrow}(\carre)$. 
Let $B_0$ be an arbitrary element in the support of $\nu$ and $\ell \in \left(\LISt_k(\mu),\alpha\right)$. 
Since $\mu(B_0) \leq \LISt_k(\mu) < \ell$, by monotone continuity there exists $\varepsilon>0$ such that:
\begin{equation}\label{domin_mean}
    \mu\left( B_0^\varepsilon \right) < \ell.
\end{equation}
Then by Portmanteau theorem and the fact that $B_0$ is in the support of $\nu$, we get:
\begin{equation}\label{borne_inf}
    \liminf_{\as{n}{\infty}}\nu_n\left(
    \lbrace B\in\mathcal{P}_K^{k\nearrow}(\carre) \; : \; \dH(B,B_0) < \varepsilon \rbrace
    \right) \geq \nu\left(
    \lbrace B\in\mathcal{P}_K^{k\nearrow}(\carre) \; : \; \dH(B,B_0) < \varepsilon \rbrace
    \right) > 0.
\end{equation}
However:
\begin{align*}
    &\nu_n\left(
    \lbrace B\in\mathcal{P}_K^{k\nearrow}(\carre) \; : \; \dist(B,B_0) < \varepsilon \rbrace
    \right)
    \\ \nonumber&\leq \nu_n\left(
    \lbrace B\in\mathcal{P}_K^{k\nearrow}(\carre) \; : \; B\subseteq B_0^\varepsilon \rbrace
    \right)
    \\&= \P_{\mu^{\otimes n}}\big( 
    \text{the maximal $k$-nondecreasing subset of }Z_1,\dots,Z_n\text{ is contained in }B_0^\varepsilon
    \;\, \big\vert \;\,
    \LIS_k\left( Z_1,\dots,Z_n\right) \geq \alpha n \big)
    \\&\leq \P_{\mu^{\otimes n}}\left( 
    \text{at least }\alpha n\text{ points among }Z_1,\dots,Z_n\text{ are in }B_0^\varepsilon
    \;\, \big\vert \;\,
    \LIS_k\left( Z_1,\dots,Z_n\right) \geq \alpha n
    \right)
\end{align*}
thus, writing $S_n$ for a binomial random variable of parameter $\big(n,\mu(B_0^\varepsilon)\big)$:
\begin{align*}
    \nu_n\left(
    \lbrace B\in\mathcal{P}_K(\carre) \; : \; \dist(B,B_0) < \varepsilon \rbrace
    \right)
    \leq \frac{\P( S_n \geq \alpha n )}
    {\P\big(\LIS_k(Z_1,\dots,Z_n) \geq \alpha n\big)}.
\end{align*}
Now consider another binomial random variable $S_n'$ of parameter $(n,\ell)$. 
By (\ref{domin_mean}), stochastic domination, \Cref{th_Cramer} and the fact that $\ell<\alpha$:
\begin{align*}
    \frac{1}{n}\log\P( S_n \geq \alpha n )
    \leq \frac{1}{n}\log\P( S_n' \geq \alpha n )
    \cv{n}{\infty} -\alpha\log\frac{\alpha}{\ell} - (1-\alpha)\log\frac{1-\alpha}{1-\ell}.
\end{align*}
Then, using \Cref{borne_inf} for the first equality:
\begin{align*}
    0 &= 
    \lim_{\as{n}{\infty}}\frac{1}{n}\log\nu_n\left(
    \lbrace B\in\mathcal{P}_K(\carre) \; : \; \dist(B,B_0) < \varepsilon \rbrace
    \right)
    \\&\leq\lim_{\as{n}{\infty}}\frac{1}{n}\log\P( S_n'\geq\alpha n ) - \lim_{\as{n}{\infty}}\frac{1}{n}\log\P\big(\LIS_k(Z_1,\dots,Z_n) \geq \alpha n\big)
    \\&= x - \alpha\log\frac{\alpha}{\ell} - (1-\alpha)\log\frac{1-\alpha}{1-\ell}.
\end{align*}
Since this holds for any $\ell\in\left(\LISt_k(\mu),\alpha\right)$,
we deduce (\ref{ine_goal_upper}) as desired. 
\end{proof}

\subsection{Lower tail speed}\label{section_lower_speed_proof}

The goal of this section is to establish \Cref{prop_lower_tail} and \Cref{th_speed_equiv} stating that, except in trivial cases, a lower tail large deviation principle for $\LIS(\sigma_n)/n$ where $\sigma_n\sim\sample_n(\mu)$ should indeed have speed $n$. 

\begin{proof}[Proof of \Cref{prop_lower_tail}]
Let $j\geq 1$ and $A_j$ be a maximal $j$-nondecreasing subset for $\mu$. 
Notice that all permutations $\sigma$ satisfy $\LIS_j(\sigma) \leq j\LIS(\sigma)$ by union bound.
We can then write:
\begin{multline*}
    \P\big( \LIS(\sigma_n) < \beta n \big)
    \leq \P\big( \LIS_j(\sigma_n) < j\beta n \big)
    \leq \P_{\mu^{\otimes n}}\left(
    \text{less than } j\beta n \text{ points are in } A_j
    \right)
    =\P(S_n < j\beta n)
\end{multline*}
where $S_n$ is a binomial random variable of parameter $\left( n,\LISt_j(\mu) \right)$. 
If $j$ satisfies $j\beta<\LISt_{j}(\mu)$, \Cref{th_Cramer} implies:
\begin{equation*}
    \limsup_{\as{n}{\infty}}\frac{1}{n}
    \log\P\big( \LIS(\sigma_n) < \beta n \big)
    \leq - \Lambda_{\LISt_{j}(\mu)}^*\big(j\beta\big).
\end{equation*}
Since this bound is valid for all such $j$, \Cref{prop_lower_tail} follows.
\end{proof}

In order to prove \Cref{th_speed_equiv} we recall a well known large deviation principle on empirical measures and apply it to our framework. 
See \cite{DZ98} for a proof and \cite{BDMW22} for another application to permuton theory.

\begin{theorem}[Sanov's theorem]\label{sanov}
Let $\mu\in\prob$ and $(Z_i)_{i\in\N^*}$ be a family of i.i.d.~random points distributed under $\mu$. Define for each $n\in\N^*$ the empirical measure
\begin{equation*}
    \Delta_n := \frac{1}{n}(\delta_{Z_1}+\dots+\delta_{Z_n})
\end{equation*}
Then the sequence $(\Delta_n)_{n\in\N^*}$ satisfies a large deviation principle
\begin{equation*}
    \left\{
    \begin{array}{ll}
    \text{for each open set $E$ of }\prob,
    &\underset{\as{n}{\infty}}{\liminf}\frac{1}{n}\log\P(\Delta_n\in E) \geq -\underset{\nu\in E}{\inf} D(\nu|\mu);\smallskip\\
    \text{for each closed set $F$ of }\prob,
    &\underset{\as{n}{\infty}}{\limsup}\frac{1}{n}\log\P(\Delta_n\in F) \leq -\underset{\nu\in F}{\inf} D(\nu|\mu);
    \end{array}
    \right.
\end{equation*}
with good rate function $D(\cdot\vert\mu)$, called the Kullback-Leibler divergence, defined by:
\begin{equation*}
    \text{for all }\nu\in\prob,\quad D(\nu|\mu):=
    \left\{
    \begin{array}{ll}
    \int_{\carre}\log\left(\frac{d\nu}{d\mu}\right)d\nu
    & \text{if }\nu\text{ is absolutely continuous with respect to }\mu;\\
    +\infty 
    &\text{otherwise}.
    \end{array}
    \right.
\end{equation*}
\end{theorem}

\medskip

Let $k\in\N^*$. 
By upper semi-continuity, for any $\alpha,\beta\in[0,1]$, 
$\left\{ \nu\in\prob \,:\, \LISt_k(\nu) \geq \alpha \right\}$
is closed and
$\left\{ \nu\in\prob \,:\, \LISt_k(\nu) < \beta \right\}$
is open.
Since $\LISt_k(\Delta_n)=\LIS_k(Z_1,\dots,Z_n)/n$, we deduce the following corollary:

\begin{cor}\label{cor_Sanov}
Let $\mu\in\per$, $\sigma_n\sim\sample_n(\mu)$ for each $n\in\N^*$, and $\alpha,\beta\in[0,1]$. Then:
\begin{equation*}
    \left\{
    \begin{array}{ll}
    \underset{\as{n}{\infty}}{\liminf}\frac{1}{n}\log
    \P\big( \LIS_k(\sigma_n)<\beta n \big)
    \;\geq\; -\underset{\nu\in\prob\; ;\; \LISt_k(\nu)<\beta}{\inf} D(\nu|\mu);\smallskip\\
    \underset{\as{n}{\infty}}{\limsup}\frac{1}{n}\log 
    \P\big( \LIS_k(\sigma_n)\geq\alpha n \big)
    \;\leq\; -\underset{\nu\in\prob\; ;\; \LISt_k(\nu)\geq\alpha}{\inf} D(\nu|\mu).
    \end{array}
    \right.
\end{equation*}
\end{cor}

With the help of \Cref{cor_Sanov}, we can now prove \Cref{th_speed_equiv}.

\begin{proof}[Proof of \Cref{th_speed_equiv}]
We prove both equivalences $(i)\Longleftrightarrow(ii)$ and $(ii)\Longleftrightarrow(iii)$.
Notice that $(ii)\implies(i)$ is immediate.

\medskip

\underline{Proof of $(i)\implies(ii)$.}
We proceed by contraposition, and suppose that $\P\big(\LIS(\sigma_{n_0})<\beta n_0\big)>0$ for some $n_0\in\N^*$.
By Robinson--Schensted correspondence (or Erd\H{o}s-Szekeres lemma):
\begin{equation*}
    \text{for any }n\in\N^*\text{ and }\sigma\in\S_n,\quad
    \LIS(\sigma)\LDS(\sigma)\geq n.
\end{equation*}
Hence we get:
\begin{equation*}
    \P\big(\LDS(\sigma_{n_0})\geq k\pp1\big) \geq 
    \P\left(\LIS(\sigma_{n_0})<\frac{n_0}{k}\right) \geq
    \P\big( \LIS(\sigma_{n_0}) < \beta n_0 \big) > 0
\end{equation*}
which in particular imposes $n_0\geq k\pp1$, and then:
\begin{multline*}
    0<\P_{\mu^{\otimes n_0}}\left(\exists i_1<\dots<i_{k\p1},\; \{Z_{i_1},\dots,Z_{i_{k\p1}}\}\in\mathcal{P}_K^\searrow(\carre)\right)
    \\ \leq \genfrac(){0pt}{}{n_0}{k\pp1}
    \P_{\mu^{\otimes k\p1}}\left( \{Z_{1},\dots,Z_{k\p1}\} \in\mathcal{P}_K^\searrow(\carre)\right).
\end{multline*}
Thus, since the marginals of $\mu$ are continuous:
\begin{equation*}
    \P_{\mu^{\otimes k\p1}}\left(\{Z_{1},\dots,Z_{k\p1}\}\in\mathcal{P}_{k+1}^{\searrow\!\!\searrow}(\carre)\right) >0
\end{equation*}
where $\mathcal{P}_{k\p1}^{\searrow\!\!\searrow}(\carre)$ denotes the set of nonincreasing subsets of $\carre$ of cardinal $k\pp1$ whose elements share no common $x$- or $y$-coordinate.
As a consequence, we can define the measure $\mu'$ as the law of $(Z_1,\dots,Z_{k\p1})$ under $\mu^{\otimes k\p1}$ conditioned on the event
$\{Z_{1},\dots,Z_{k\p1}\}\in \mathcal{P}_{k+1}^{\searrow\!\!\searrow}(\carre)$.
Now take an element $(z_1,\dots,z_{k\p1})$ in the support of $\mu'$. 
There exists $\varepsilon>0$ satisfying:
\begin{equation}\label{décr_supp_boules}
    \text{for any } z_1'\in B_\d(z_1,\varepsilon),\dots,z_{k\p1}'\in B_\d(z_{k\p1},\varepsilon),\quad
    \{z_1',\dots,z_{k\p1}'\}\text{ is nonincreasing}.
\end{equation}
For all $i\in\lb 1,k\pp1\rb$ the point $z_i$ is in the support of $\mu$, therefore $m_i:=\mu\big(B_\d(z_i,\varepsilon)\big)>0$.
Now define a measure $\nu$ as:
\begin{equation*}
    \nu(\cdot) := \sum_{i=1}^{k\p1} \frac{1}{(k\p1)m_i}\mu\big( B_\d(z_i,\varepsilon)\cap\cdot\big).
\end{equation*}
This probability measure is absolutely continuous with respect to $\mu$, with bounded density
\begin{equation*}
    \frac{d\nu}{d\mu}=
    \sum_{i=1}^{k+1} \frac{1}{(k\p1)m_i}\mathbf{1}_{B(z_i,\varepsilon)}.
\end{equation*}
Hence, using the notation of \Cref{sanov}, $D(\nu\vert\mu) < +\infty$.
Moreover, by (\ref{décr_supp_boules}) and the fact that $\nu$ puts mass $1/(k\p1)$ on each ball of its support,
$\LISt(\nu)\leq {1}/({k\p1}) < \beta$.
Thus we can conclude with the first inequality of \Cref{cor_Sanov}:
\begin{equation*}
    \liminf_{\as{n}{\infty}}\frac1n \log
    \P\big( \LIS(\sigma_n)<\beta n \big)
    \geq -D(\nu|\mu) > -\infty.
\end{equation*}

\medskip

\underline{Proof of $(ii)\implies(iii)$.}
Suppose (ii) holds.
First let us show that under this hypothesis:
\begin{equation}\label{eq_1_trivial}
    \text{for all }n\in\N^*,\quad
    \P\big( \LDS(\sigma_{n}) \leq k \big) =1.
\end{equation}
For this, suppose there exists $n_0\in\N^*$ such that $\P\big( \LDS(\sigma_{n_0}) \geq k\p1 \big) >0$.
Then necessarily $n_0\geq k\p1$. Thus we can write, as before:
\begin{equation*}
    0<\P\big( \LDS(\sigma_{n_0}) \geq k\p1 \big) 
    \leq \genfrac(){0pt}{}{n_0}{k\p1}
    \P\big( \LDS(\sigma_{k\p1}) = k\p1 \big)
\end{equation*}
whence $\P\big( \LIS(\sigma_{k\p1})=1 \big) >0$.
Since $\beta(k+1)>1$ this contradicts (ii), and (\ref{eq_1_trivial}) is established.
Then by Robinson--Schensted's correspondence, $\P\big( \LIS_k(\sigma_{n})=n \big) =1$ for all $n\in\N^*$.
Using \Cref{prop_conv_LISk} we deduce:
\begin{equation*}
    \LISt_k(\mu) = \lim_{\as{n}{\infty}}
    \frac{\LIS_k(\sigma_{n})}{n} = 1 
    \quad\text{almost surely.}
\end{equation*}

\medskip

\underline{Proof of $(iii)\implies(ii)$.}
Suppose $\LISt_k(\mu)=1$ and fix $n\in\N^*$. 
By \Cref{prop_semicon_LISt} there exist nondecreasing subsets $A_1,\dots,A_k\in\mathcal{P}_K^\nearrow(\carre)$ such that $\mu(A_1\cup\dots\cup A_k)=1$. 
Thus if $Z_1,\dots,Z_n$ are i.i.d.~random points distributed under $\mu$, by pigeonhole principle there must be at least $n/k$ of them in some $A_i$, implying:
\begin{equation*}
    \P\big( \LIS(\sigma_n) \geq \beta n \big) \geq \P\big( \LIS(\sigma_n) \geq n/k \big) = 1
\end{equation*}
as desired.
\end{proof}

\subsection{Lower tail rate function}\label{section_lower_tail_proofs}

Before proving \Cref{prop_contre_exemple}, we state a simple lemma about conditional concentration of binomial laws.

\begin{lemma}\label{lem_concentration_cond}
Let $S_n$ be a binomial random variable of parameter $(n,p)$ for some $0<p<1$. Then for any $q\in(0,p)$ and $\varepsilon>0$ such that $q-\varepsilon>0$,
there exists $\delta>0$ satisfying
\begin{equation*}
    \P\left( q-\varepsilon \leq S_n/n < q \;\big\vert\; S_n/n < q  \right)
    = 1 - \exp\big( -n\delta + o(n) \big)
    \quad\text{as }\as{n}{\infty}.
\end{equation*}
Likewise for any $q\in(p,1)$ and $\varepsilon>0$ such that $q+\varepsilon<1$,
there exists $\delta>0$ satisfying
\begin{equation*}
    \P\left( q \leq S_n/n < q+\varepsilon \;\big\vert\; S_n/n \geq q  \right)
    = 1 - \exp\big( -n\delta + o(n) \big)
    \quad\text{as }\as{n}{\infty}.
\end{equation*}
\end{lemma}

\begin{proof}
Let us prove the first assertion, as the second one is analogous. 
Applying \Cref{th_Cramer}, we get as $\as{n}{\infty}$:
\begin{equation*}
    \P(S_n/n < q) = \exp\left( -n\Lambda_p^*(q) + o(n) \right)
    \quad\text{and}\quad
    \P(S_n/n < q-\varepsilon) = \exp\left( -n\Lambda_p^*(q-\varepsilon) + o(n) \right).
\end{equation*}
Since $\Lambda_p^*(q-\varepsilon)>\Lambda_p^*(q)$, the assertion follows directly. 
\end{proof}

\medskip

\begin{proof}[Proof of \Cref{prop_contre_exemple}]
Define the following diagonals:
\begin{equation*}
    \left\{
    \begin{array}{lll}
        D_1^{1}\text{ from }(0,0.5)\text{ to }(0.2,0.7);\smallskip\\
        D_2^{1}\text{ from }(0.2,0)\text{ to }(0.5,0.3);\smallskip\\
        D_3^{1}\text{ from }(0.5,0.7)\text{ to }(0.8,1);\smallskip\\
        D_4^{1}\text{ from }(0.8,0.3)\text{ to }(1,0.5);
    \end{array}
    \right.
    \quad\text{and}\quad
    \left\{
    \begin{array}{lll}
        D_1^{2}\text{ from }(0,0.42)\text{ to }(0.28,0.7);\smallskip\\
        D_2^{2}\text{ from }(0.28,0)\text{ to }(0.58,0.3);\smallskip\\
        D_3^{2}\text{ from }(0.58,0.7)\text{ to }(0.88,1);\smallskip\\
        D_4^{2}\text{ from }(0.88,0.3)\text{ to }(1,0.42);
    \end{array}
    \right.
\end{equation*}
then define $\mu_1$ and $\mu_2$ as the only permutons having support $\cup_{i=1}^4D_i^{1}$ and $\cup_{i=1}^4D_i^{2}$
respectively. See \Cref{fig_contre_exemple} for a representation. 
\begin{figure}
    \centering
    \includegraphics[scale=0.7]{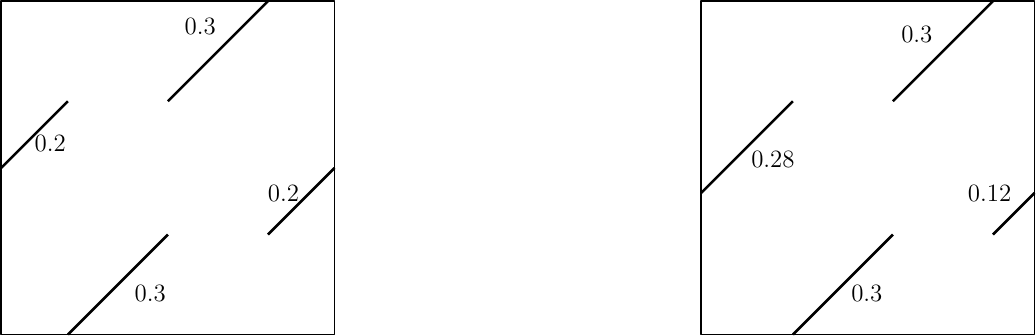}
    \caption{Representation of both permutons $\mu_1$ (on the left) and $\mu_2$ (on the right) used to prove \Cref{prop_contre_exemple}.}
    \label{fig_contre_exemple}
\end{figure}
These permutons satisfy
$\LISt(\mu_1)=\LISt(\mu_2)=0.6$,
$\LISt_2(\mu_1)=\LISt_2(\mu_2)=1$,
and in particular $\shapet(\mu_1)=\shapet(\mu_2)$.
Now fix $\beta:= 0.55$. 
Our goal is to show that
\begin{equation*}
    \limsup_{\as{n}{\infty}}\frac{1}{n}
    \log\P\big( \LIS(\tau_n) < \beta n \big) 
    < \liminf_{\as{n}{\infty}}\frac{1}{n}
    \log\P\big( \LIS(\sigma_n) < \beta n \big)
\end{equation*}
where $\sigma_n\sim\sample_n(\mu_1)$ and $\tau_n\sim\sample_n(\mu_2)$.
Let $n\in\N^*$ and, for each $j\in\{1,2\}$, $Z_1^{j},\dots,Z_n^{j}$ be random i.i.d.~points distributed under $\mu_j$. 
Also define for each  $i\in\lb1,4\rb$, $S_{n,i}^{j} := \big\vert \{Z_1^{j},\dots,Z_n^{j}\}\cap D_i^{j} \big\vert$.
This is a binomial random variable of parameter $\big(n,p_i^{j}\big)$, where
\begin{equation*}
    (p_1^{1}, p_2^{1}, p_3^{1}, p_4^{1}) = (0.2,0.3,0.3,0.2)
    \quad\text{and}\quad
    (p_1^{2}, p_2^{2}, p_3^{2}, p_4^{2}) = (0.28,0.3,0.3,0.12).
\end{equation*}
Moreover, these variables satisfy
\begin{equation}\label{somme_a_n}
    \text{almost surely for each }j\in\{1,2\},\quad \sum_{i=1}^4S_{n,i}^{j}=n.
\end{equation}
Now for each $j\in\{1,2\}$, define the events
\begin{equation*}
        A^{j} := \left\{ \LIS\big(Z_1^{j},\dots,Z_n^{j}\big) < 0.55 n \right\}
        \quad\text{and}\quad
        B^{j} := \lbrace S_{n,2}^{j} + S_{n,3}^{j} < 0.55 n \rbrace
\end{equation*}
and notice the inclusion $A^{j}\subseteq B^{j}$. 
Let us study for each $j\in\{1,2\}$ the probability of $A^{j}$ conditionally to $B^{j}$.

\bigskip

\underline{Study of $\mu_1$.} 
Since $S_{n,2}^{1}+S_{n,3}^{1}$ follows a binomial law of parameter $(n,0.6)$, \Cref{lem_concentration_cond} implies:
\begin{equation}\label{premiere_ine_cond}
    \P\left( 
    0.53 \leq \frac{S_{n,2}^{1}+S_{n,3}^{1}}{n} < 0.55 
    \:\Big\vert\: B^{1} \right) 
    = 1-\exp\big(-n\delta+o(n)\big)
    \quad\text{as }\as{n}{\infty}
\end{equation}
for some $\delta>0$ whose value might change along the proof. 
Moreover, since $S_{n,1}^{1}+S_{n,4}^{1}$ follows a binomial law of parameter $(n,0.4)$ and $B^{1} = \left\{ S_{n,1}^{1}+S_{n,4}^{1} \geq 0.45 n \right\}$ by (\ref{somme_a_n}), \Cref{lem_concentration_cond} also implies:
\begin{equation*}
    \P\left( 0.45 \leq \frac{S_{n,1}^{1}+S_{n,4}^{1}}{n} < 0.47 \:\Big\vert\: B^{1} \right) = 1-\exp\big(-n\delta+o(n)\big)
    \quad\text{as }\as{n}{\infty}.
\end{equation*}
Note that the law of $S_{n,2}^{1}$ conditional to the variable $S_{n,2}^{1}+S_{n,3}^{1}$ is binomial of parameter $\big(S_{n,2}^{1}+S_{n,3}^{1},0.5\big)$ (indeed the points appearing in $D_{n,2}^{1}\cup D_{n,3}^{1}$ can be sampled by first determining their number $S_{n,2}^{1}+S_{n,3}^{1}$ and then placing them uniformly at random). 
We can thus write:
\begin{align*}
    &\P\left(\left. 0.26\leq \frac{S_{n,2}^{1}}{n} <0.28 \:\right\vert\: B^{1} \right)
    \\&\geq \frac{\P\left( 0.26\leq \frac{S_{n,2}^{1}}{n} <0.28 \text{ and } 0.53\leq \frac{S_{n,2}^{1}+S_{n,3}^{1}}{n} <0.55 \right)}{\P(B^{1})}
    \\&\geq \P\left( \frac{S_{n,2}^{1}+S_{n,3}^{1}}{2n} -0.005\leq \frac{S_{n,2}^{1}}{n} < \frac{S_{n,2}^{1}+S_{n,3}^{1}}{2n}+0.005 \:\left\vert\: 0.53\leq \frac{S_{n,2}^{1}+S_{n,3}^{1}}{n} <0.55 \right.\right)
    \\&\quad\times \frac{\P\left( 0.53\leq \frac{S_{n,2}^{1}+S_{n,3}^{1}}{n} <0.55 \right)}{\P(B^{1})}
    \\&\geq 1-\exp\big(-n\delta+o(n)\big)
    \quad\text{as }\as{n}{\infty}
\end{align*}
for some $\delta>0$, by \Cref{th_Cramer} and (\ref{premiere_ine_cond}). 
More generally we have:
\begin{equation}\label{controles_points_contrex_pgd}
    \left\{
    \begin{array}{ll}
        \text{for }i\in\{2,3\},\quad
        \P\left(\left. 0.26\leq \frac{S_{n,i}^{1}}{n} <0.28 \:\right\vert\: B^{1} \right)
        \geq 1-\exp\big(-n\delta+o(n)\big);\smallskip\\
        \text{for }i\in\{1,4\},\quad
        \P\left(\left. 0.22\leq \frac{S_{n,i}^{1}}{n} <0.24 \:\right\vert\: B^{1} \right)
        \geq 1-\exp\big(-n\delta+o(n)\big);
    \end{array}
    \right.
\end{equation}
as $\as{n}{\infty}$ and for some $\delta>0$. 
Under the events of (\ref{controles_points_contrex_pgd}), the longest increasing subsequence is formed by the points in $D_2^1$ and $D_3^1$ (see \Cref{fig_contre_exemple}). 
Subsequently:
\begin{equation*}
    \P\Big( \LIS\left( Z_1^{1},\dots,Z_n^{1} \right) = {S_{n,2}^{1}+S_{n,3}^{1}} \:\Big\vert\: B^{1} \Big) 
    \geq 1-\exp\big(-n\delta+o(n)\big)
    \quad\text{as }\as{n}{\infty},
\end{equation*}
whence, using \Cref{premiere_ine_cond}:
\begin{equation*}
    \P(A^{1}\vert B^{1}) \geq 1-\exp\big(-n\delta+o(n)\big)
    \quad\text{as }\as{n}{\infty}.
\end{equation*}
Thanks to \Cref{th_Cramer} we finally deduce:
\begin{equation}\label{seconde_ine_cond}
    \liminf_{\as{n}{\infty}}\frac{1}{n}
    \log\P\big( \LIS(\sigma_n) < 0.55 n \big)
    \geq \liminf_{\as{n}{\infty}} \frac{1}{n}\log\left(\big(
    1-e^{-n\delta+o(n)}\big) \P(B^{1})
    \right) = -\Lambda_{0.6}^*(0.55).
\end{equation}

\medskip

\underline{Study of $\mu_2$.} 
We adapt the previous strategy.
However, a key difference here is that the law of $S_{n,1}^{2}$ conditional to the variable $S_{n,1}^{2}+S_{n,4}^{2}$ is binomial of parameter $\big(S_{n,1}^{2}+S_{n,4}^{2},0.7\big)$. 
We get:
\begin{equation*}
    \left\{
    \begin{array}{ll}
        \P\left(\left. 0.53\leq \frac{S_{n,2}^{2}+S_{n,3}^{2}}{n} <0.55 \:\right\vert\: B^{2} \right)
        \geq 1-\exp\big(-n\delta+o(n)\big);\smallskip\\
        \P\left(\left. 0.45\leq \frac{S_{n,1}^{2}+S_{n,4}^{2}}{n} <0.47 \:\right\vert\: B^{2} \right)
        \geq 1-\exp\big(-n\delta+o(n)\big);\smallskip\\
        \P\left(\left. 0.26\leq \frac{S_{n,i}^{2}}{n} <0.28 \:\right\vert\: B^{2} \right)
        \geq 1-\exp\big(-n\delta+o(n)\big)
        \quad \text{for }i\in\{2,3\};\smallskip\\
        \P\left(\left. 0.31\leq \frac{S_{n,1}^{2}}{n} <0.33 \:\right\vert\: B^{2} \right)
        \geq 1-\exp\big(-n\delta+o(n)\big);\smallskip\\
        \P\left(\left. 0.13\leq \frac{S_{n,4}^{2}}{n} <0.15 \:\right\vert\: B^{2} \right)
        \geq 1-\exp\big(-n\delta+o(n)\big);
    \end{array}
    \right.
\end{equation*}
as $\as{n}{\infty}$, for some $\delta>0$.
Under these events, the longest increasing subsequence is this time formed by the points in $D_1^2$ and $D_3^2$.
Subsequently:
\begin{equation*}
    \P\Big( \LIS\left( Z_1^{2},\dots,Z_n^{2} \right) = {S_{n,1}^{2}+S_{n,3}^{2}} \:\Big\vert\: B^{2} \Big) 
    \geq 1-\exp\big(-n\delta+o(n)\big)
    \quad\text{as }\as{n}{\infty}
\end{equation*}
whence
\begin{equation*}
    \P(A^{2}\vert B^{2}) \leq \exp\big(-n\delta+o(n)\big)
    \quad\text{as }\as{n}{\infty}.
\end{equation*}
From this, (\ref{seconde_ine_cond}) and \Cref{th_Cramer} we finally deduce:
\begin{multline*}
    \limsup_{\as{n}{\infty}}\frac{1}{n}
    \log\P\big( \LIS(\tau_n) < 0.55 n \big)
    \leq \limsup_{\as{n}{\infty}} \frac{1}{n}\log\big(
    e^{-n\delta+o(n)} \P(B^{2})\big)
    \\= -\delta-\Lambda_{0.6}^*(0.55)
    < -\Lambda_{0.6}^*(0.55)
    \leq \liminf_{\as{n}{\infty}}\frac{1}{n}
    \log\P\big( \LIS(\sigma_n) < 0.55 n \big)
\end{multline*}
as announced.
\end{proof}

\section{Proofs of the results in Sections \ref{section_tableaux_results} and \ref{section_injectivity_results}}\label{section_tableaux_proofs}

\subsection{Convergence of tableaux}\label{section_proof_cv_tableaux}

\begin{proof}[Proof of \Cref{prop_lambda_cv}]
	Recall that for any permutation $\sigma$, its subword $\sigma^{i,j}$ consists of letters with position at most $i$ and value at most $j$.
	Also recall that $\lambda_k^\sigma(i,j) = \LIS_k(\sigma^{i,j}) - \LIS_{k\m1}(\sigma^{i,j})$, and that for any permuton $\mu$, we have $\lamt_k^\mu(x,y) = \LISt_k\left( \mu\vert_{[0,x]\times[0,y]} \right) - \LISt_{k\m1}\left( \mu\vert_{[0,x]\times[0,y]} \right)$.
	Let us first establish pointwise convergence. 
    We aim to show that for any fixed $(x,y)\in\carre$, almost surely:
    \begin{equation}\label{pointwise_lambda_convergence}
        \frac{1}{n}\LIS_k\left( \sigma_n^{\lfloor nx\rfloor,\lfloor ny\rfloor} \right)
        \cv{n}{\infty}
        \LISt_k\left( \mu\vert_{[0,x]\times[0,y]} \right) .
    \end{equation}
    Let $Z_1,\dots,Z_n$ be random i.i.d.~points under $\mu$, let $\sigma_n := \perm(Z_1,\dots,Z_n)$, and define
    \begin{equation*}
        I := \lbrace i\in\lb1,n\rb \;:\; Z_i \in [0,x]\times[0,1] \rbrace
        \quad ; \quad
        J := \lbrace i\in\lb1,n\rb \;:\; Z_i \in [0,1]\times[0,y] \rbrace.
    \end{equation*}
    Then $\perm\left( (Z_i)_{i\in I\cap J} \right) = \sigma_n^{|I| , |J|}$. 
    Moreover, conditional to its size $|I\cap J|$, the permutation $\perm\left( (Z_i)_{i\in I\cap J} \right)$ follows the law $\sample_{|I\cap J|}(\nu)$ where $\nu$ is the pre-permuton induced by $\mu$ on $[0,x]\times[0,y]$. 
    As a consequence of \Cref{prop_conv_LISk} and the law of large numbers, we get:
    \begin{multline*}
        \frac{1}{n}\LIS_k\left( \sigma_n^{|I|,|J|} \right)
        =
        \frac{|I\cap J|}{n}\frac{1}{|I\cap J|}\LIS_k\left( \sigma_n^{| I|,|J|} \right)
         \cv{n}{\infty}
        \mu([0,x]\times[0,y]) \, \LISt_k\left( \nu \right)
        =
        \LISt_k\left( \mu\vert_{[0,x]\times[0,y]} \right)
    \end{multline*}
    almost surely. 
    Since the words $\sigma_n^{\lfloor nx\rfloor,\lfloor ny\rfloor}$ and $\sigma_n^{|I|,|J|}$ differ by at most $\big| \lfloor nx\rfloor- |I| \big| + \big| \lfloor ny\rfloor- |J| \big|$ letters, we have the following rough upper bound :
    \begin{equation*}
        \left| \LIS_k\left(\sigma_n^{\lfloor nx\rfloor,\lfloor ny\rfloor}\right) - \LIS_k\left(\sigma_n^{|I|,|J|}\right) \right| \leq \big| \lfloor nx\rfloor-|I| \big| + \big| \lfloor ny\rfloor-|J| \big|.
    \end{equation*}
    Using the fact that $\mu$ is a permuton, \textit{i.e.}~that $\P( Z_i \in[0,x]\times[0,1])=x$ and $\P( Z_i \in[0,1]\times[0,y])=y$, the law of large numbers yields
    \begin{equation*}
        \frac{\lfloor nx\rfloor-|I|}{n} \cv{n}{\infty} 0
        \quad ; \quad
        \frac{\lfloor ny\rfloor-|J|}{n} \cv{n}{\infty} 0
    \end{equation*}
    almost surely, and (\ref{pointwise_lambda_convergence}) follows.

    \medskip

    We can already use this pointwise convergence to deduce properties of $\lamt^\mu$.
    Let $j\in\lb1,n\rb$ and $i\in\lb1,n\m1\rb$.
    Since the P-tableau of $\sigma_n^{i+1,j}$ is obtained after Schensted-insertion of at most one letter in the P-tableau of $\sigma_n^{i,j}$, the shape of $\sigma_n^{i,j}$ is included in the shape of $\sigma_n^{i+1,j}$ and they differ by at most one box. 
    This fact is also clear from Fomin's construction described later.
    Hence for any $k\in\N^*$:
    \begin{equation*}
        \lambda_k^{\sigma_n}(i,j) \leq \lambda_k^{\sigma_n}(i\pp1,j) \leq \lambda_k^{\sigma_n}(i,j)+1.
    \end{equation*}
    Since for any $k\in\N^*$ and any $(x,y)\in\carre$:
    \begin{equation}\label{pointwise_lambda_convergence_bis}
        \frac{1}{n}\lambda^{\sigma_n}_k\left( \lfloor xn\rfloor , \lfloor yn\rfloor \right) 
        \cv{n}{\infty}
        \lamt^{\mu}_k(x,y)
    \end{equation}
    almost surely by (\ref{pointwise_lambda_convergence}), we deduce that $\lamt^\mu_k$ in nondecreasing and $1$-Lipschitz in its first variable.
    The same holds for its second variable.
    
    \medskip

    To obtain uniform convergence, we simply need to adapt the proof of Dini's theorem. 
    Fix $\epsilon>0$ and $p := \lceil 1/\epsilon \rceil$. 
    Using (\ref{pointwise_lambda_convergence_bis}), almost surely for large enough $n$ we have:
    \begin{equation*}
        \text{for any }i,j\in\lb0,p\rb,\quad
        \left\vert \frac{1}{n}\lambda_k^{\sigma_n}\big( \lfloor ni/p\rfloor,\lfloor nj/p\rfloor \big) - \lamt_k^{\mu}\big( i/p,j/p \big) \right\vert \leq \epsilon
    \end{equation*}
    Let $(x,y)\in\carre$ and $i,j\in\lb1,p\rb$ such that $(i\mm1)/p\leq x\leq i/p$ and $(j\mm1)/p\leq y\leq j/p$. Then:
    \begin{multline*}
        \lamt_k^{\mu}\big( (i\mm1)/p,(j\mm1)/p \big) - \epsilon
        \leq \frac{1}{n}\lambda_k^{\sigma_n}\big( \lfloor n(i\mm1)/p\rfloor,\lfloor n(j\mm1)/p\rfloor \big)
        \\ \leq \frac{1}{n}\lambda_k^{\sigma_n}\big( \lfloor nx\rfloor,\lfloor ny\rfloor \big)
        \leq \frac{1}{n}\lambda_k^{\sigma_n}\big( \lfloor ni/p\rfloor,\lfloor nj/p\rfloor \big)
        \leq \lamt_k^{\mu}\big( i/p,j/p \big) + \epsilon.
    \end{multline*}
    Since $\lamt^\mu_k$ in nondecreasing and $1$-Lipschitz in both variables,
    we deduce:
    \begin{equation*}
        \left| \frac{1}{n}\lambda_k^{\sigma_n}(\lfloor nx\rfloor,\lfloor ny\rfloor) - \lamt_k^\mu(x,y) \right|
        \leq \epsilon+2/p \leq 3\epsilon.
    \end{equation*}
    We proved that for any $\epsilon>0$, almost surely
    \begin{equation*}
        \limsup_{\as{n}{\infty}} \left\| 
        \frac{1}{n}\lambda_k^{\sigma_n}(\lfloor n\,\cdot\rfloor,\lfloor n\,\cdot\rfloor)
        - \lamt_k^{\mu} \right\|_\infty
        \leq 3\epsilon.
    \end{equation*}
    By restricting to $\epsilon\in\Q_+^*$ we conclude that almost surely, uniform convergence holds.
\end{proof}

\subsection{Inverse Fomin dynamics}\label{section_proofs_Fomin}

We start this section by recalling the ``edge local rules'' construction of Robinson--Schensted's correspondence, as presented by Viennot in \cite{V18}.
Let $\sigma\in\S_n$ and consider the grid $\lb0,n\rb\times\lb0,n\rb$.
For each $1\leq i,j\leq n$, fill the cell $c_{i,j} :=[ i\mm1, i ]\times[ j\mm1,j ]$ with a point if and only if $\sigma(i)=j$.
The algorithm labels the edges of each cell of the grid with nonnegative integers.
Write $S(c), W(c), N(c), E(c)$ for the south, west, north, east edge labels of a cell $c$.
We initialize the algorithm with
\begin{equation*}
    \text{for each }1\leq k\leq n,\quad
    S(c_{k,1})=0 \quad\text{and}\quad W(c_{1,k})=0
\end{equation*}
and the local rules are the following for each cell $c$:
\begin{itemize}
    \item if $S(c)=W(c)=0$ and the cell is empty then $N(c)=E(c)=0$;
    \item if $S(c)=W(c)=0$ and the cell is filled with a point then $N(c)=E(c)=1$;
    \item if $S(c)=W(c)>0$ then $N(c)=E(c)=S(c)+1$;
    \item if $S(c) \neq W(c)$ then $N(c)=S(c)$ and $E(c)=W(c)$;
\end{itemize}
see \Cref{fig_Fomin_direct_rules}. 

\begin{figure}
    \centering
    \includegraphics[scale=0.66]{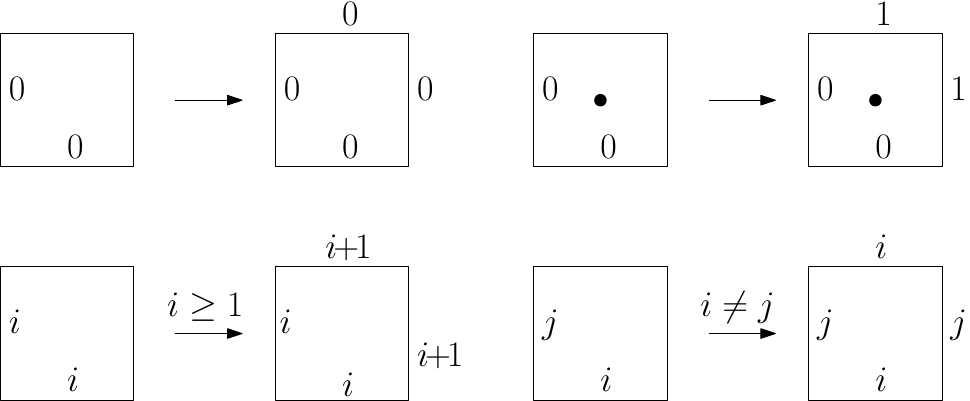}
    \caption{Illustration of Fomin's direct local rules.}
    \label{fig_Fomin_direct_rules}
\end{figure}

These edge rules are equivalent to Fomin's growth diagrams algorithm, described in Section 5.2 of \cite{S01}, which labels the vertices of each cell $c$ of the grid with diagrams $SE(c), SW(c), NW(c), NE(c)$.
This algorithm starts with empty diagrams on the south and west borders, and uses local rules to deduce for each cell $c$ the diagram $NE(c)$ from the other three diagrams and the presence or not of a permutation point.
The local rules for Fomin's growth diagrams are equivalent to Viennot's edge local rules, in the sense that the label of an edge is $k>0$ if and only if the diagrams at its endpoints differ by exactly one box in their $k$-th row, and the label is $0$ if and only if these diagrams are equal.
It is possible to prove \cite{S01} that the sequence of diagrams on the north border (resp.~east border) of the grid encodes the Q-tableau (resp.~P-tableau) of the permutation.
Subsequently one can see that the diagram labeling any vertex $(i,j)$ is the RS-shape of the word $\sigma^{i,j}$.
See \Cref{fig_Fomin_rep} for an example.

\begin{figure}
    \begin{minipage}[c]{0.45\linewidth}
        \centering
        \includegraphics[scale=0.72]{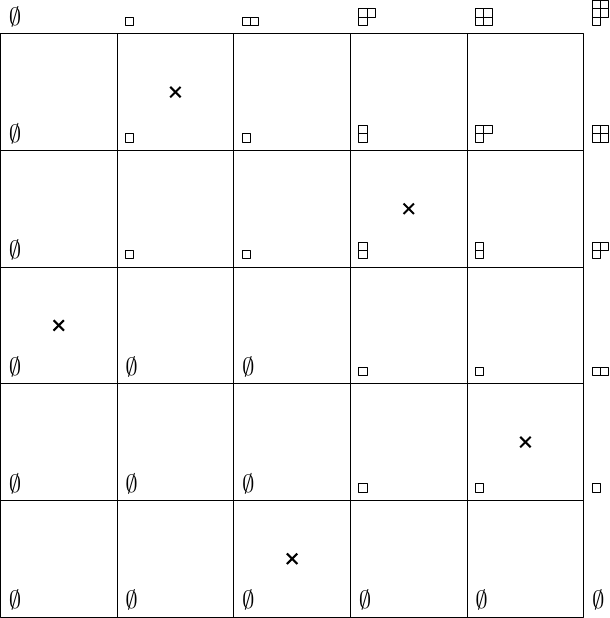}
    \end{minipage}
    \hfill
    \begin{minipage}[c]{0.45\linewidth}
        \centering
        \includegraphics[scale=0.72]{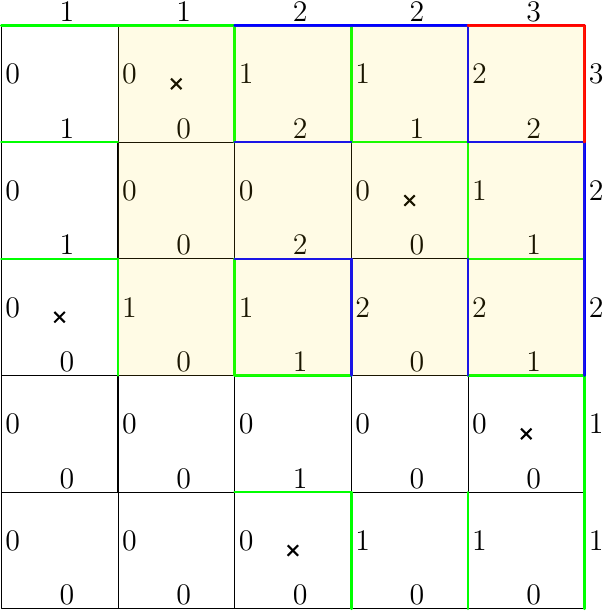}
    \end{minipage}
    \caption{Fomin's rules applied to the permutation $\sigma = 3\,5\,1\,4\,2$. 
    	On the left, diagram labels on vertices are used to represent the shapes of sub-words. 
    	On the right, integer labels on edges are used to represent the row length increments. 
    	For better readability, each integer label is associated with a color (green for $1$, blue for $2$, red for $3$). 
    	We can recover the P- and Q-tableau from the integer labels on the east and north sides of the grid:
    	$P(\sigma)$ contains the letters $1$ and $2$ in its first row, $3$ and $4$ in its second row, and $5$ in its third row (in this example, $Q(\sigma)=P(\sigma)$ because $\sigma$ is an involution).
    	A rectangle is highlighted for later reference.}
    \label{fig_Fomin_rep}
\end{figure}

We can sum up these observations with our notation: for any $0\leq i\leq i'\leq n$, $0\leq j\leq j'\leq n$ and $k\in\N^*$, the number of edges labeled $k$ on any up-right path from $(i,j)$ to $(i',j')$ is equal to $\lambda^\sigma_k(i',j')-\lambda^\sigma_k(i,j)$.

\medskip

In the previous algorithm, labels on the north and east sides of the grid read $Q(\sigma)$ and $P(\sigma)$, in the sense that the label of the $i$-th top edge (from the left) indicates which row of $Q(\sigma)$ contains the letter $i$, and the label of the $j$-th right edge (from the bottom) indicates which row of $P(\sigma)$ contains the letter $j$.
Since Robinson--Schensted's correspondence is bijective, it is possible to retrieve the permutation and the labels on the whole grid from these north and east labels.
This is done with the following edge rules, which we will refer to as \textit{Fomin's inverse local rules}:
\begin{itemize}
    \item if $N(c)=E(c)=0$, then $S(c)=W(c)=0$;
    \item if $N(c)=E(c)>0$ then $S(c)=W(c)=N(c)-1$;
    \item if $N(c) \neq E(c)$ then $S(c)=N(c)$ and $W(c)=E(c)$;
\end{itemize}
see \Cref{fig_Fomin_rules}.
This recursively reconstructs the labels on all edges of the grid, and permutation points are exactly in the cells with south/west labels $0$ and north/east labels $1$.

\begin{figure}
    \begin{minipage}{0.33\linewidth}
        \centering
        \includegraphics[scale=0.66]{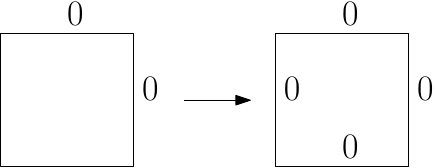}
    \end{minipage}
    \begin{minipage}{0.33\linewidth}
        \centering
        \includegraphics[scale=0.66]{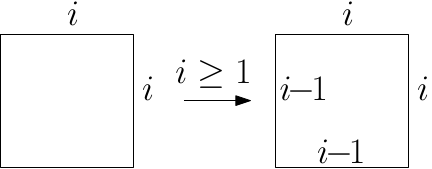}
    \end{minipage}
    \begin{minipage}{0.33\linewidth}
        \centering
        \includegraphics[scale=0.66]{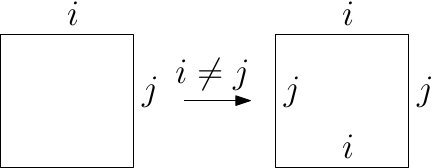}
    \end{minipage}
    \caption{Illustration of Fomin's inverse local rules.}
    \label{fig_Fomin_rules}
\end{figure}

\bigskip

Let us introduce some formalism related to these rules.
We shall study finite words on nonnegative integer letters, representing sequences of labels on vertical or horizontal lines of rectangular grids.
For instance if the edges of a rectangular grid $\lb i,i'\rb\times\lb j,j'\rb$ are labeled with integers then the associated top-word is the ``right-to-left'' sequence of labels from $(i',j')$ to $(i,j')$. 
Similarly the associated right-word is the ``top-to-bottom'' sequence of labels from $(i',j')$ to $(i',j)$. 
Associated bottom- and left-words are defined analogously.
For example on \Cref{fig_Fomin_rep} the highlighted rectangle yields top-word $3\,2\,2\,1$, right-word $3\,2\,2$, bottom-word $1\,0\,1\,0$ and left-word $0\,0\,1$.

\medskip

\begin{figure}
    \centering
    \includegraphics[scale=0.5]{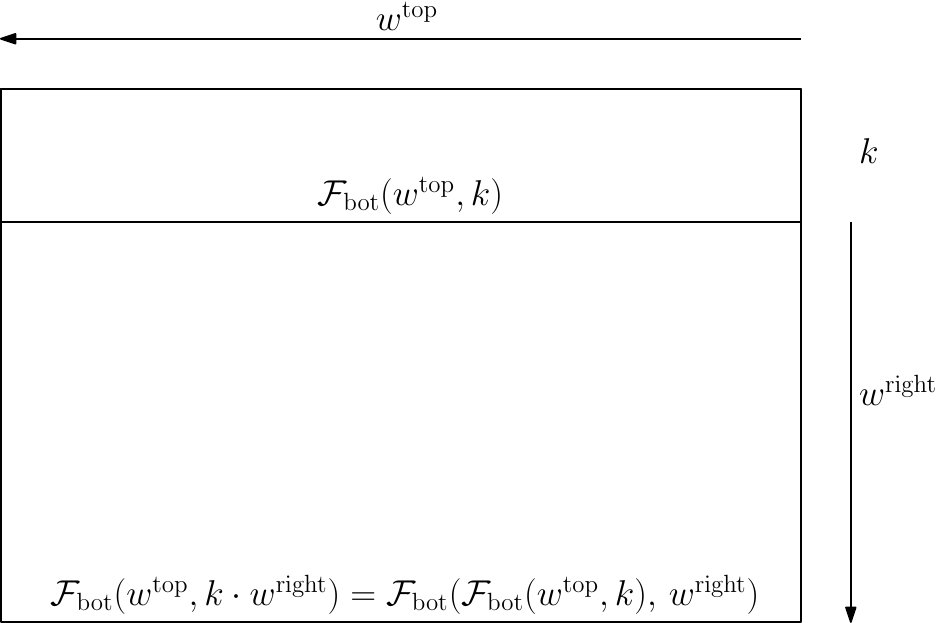}
    \caption{Illustration of Property (\ref{letter_by_letter_property}).}
    \label{fig_letter_by_letter}
\end{figure}

Fomin's inverse rules can be interpreted as a map $\Fom =\left( \Fom_\bot ,\, \Fom_\lef \right)$ with input a pair of top- and right-words and output a pair of bottom- and left-words. 
In the previous example: $\Fom(3\,2\,2\,1 , 3\,2\,2) = (1\,0\,1\,0 , 0\,0\,1)$.
Notice that the symmetry of Fomin's inverse local rules implies the following property of $\Fom$:
\begin{equation*}
    \text{for any words }w,w',\quad
    \Fom(w',w) = \left( \Fom_\lef(w,w') ,\, \Fom_\bot(w,w') \right).
\end{equation*}
If $w^\top,w^\rig$ are words we sometimes say that $\Fom_\bot(w^\top,w^\rig)$ is obtained by \textit{applying} $w^\rig$ to $w^\top$.
In most reasonings we shall apply $w^\rig$ to $w^\top$ letter by letter. This can be done since
\begin{equation}\label{letter_by_letter_property}
    \text{for any }k\in\N,\quad
    \Fom_\bot\left( w^\top ,\, k\cdot w^\rig \right) =
    \Fom_\bot\left( \Fom_\bot(w^\top,k) ,\, w^\rig \right),
\end{equation}
see \Cref{fig_letter_by_letter}.
We can therefore reformulate Fomin's inverse local rules as follows:
\begin{itemize}
    \item applying the letter $0$ to a word does not change it;
    \item applying a letter $k\geq1$ to a word containing no $k$ does not change it. 
Otherwise it changes its first $k$ into a $k\mm1$, and then applies $k\mm1$ to the rest of the word;
\end{itemize}
see \Cref{fig_apply_letter}.

\begin{figure}
    \centering
    \includegraphics[scale=0.6]{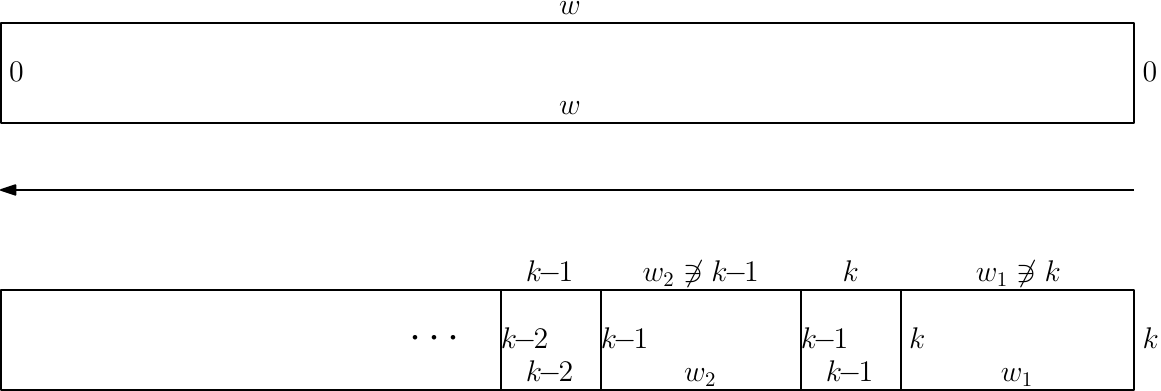}
    \caption{The effects of applying one letter to a word.}
    \label{fig_apply_letter}
\end{figure}

\medskip

If $w$ is a word and $h\in\N^*$, denote by $\#_h w$ the number of $h$'s in $w$.
Since the number of edges labeled $h$ between two vertices does not depend on the up-right path taken, we deduce the following ``mass conservation'' property of $\Fom$:
\begin{equation}\label{mass_conservation}
    \text{for any words }w^\top,w^\rig,\quad
    \#_h \Fom_\lef\left( w^\top,w^\rig \right) + \#_h w^\top
    = \#_h \Fom_\bot\left( w^\top,w^\rig \right) + \#_h w^\rig.
\end{equation}
This will be useful in \Cref{section_proofs_diff}.
We then define a pseudo-distance between words $w,w'$ by the following formula:
\begin{equation*}
    \dist_\Fom(w,w') := \sup_{h\in\N^*} \sup_{w^\rig \text{ word}}
    \left| \#_h \Fom_\bot\left( w ,\, w^\rig\right) 
    - \#_h \Fom_\bot\left( w' ,\, w^\rig\right) \right|.
\end{equation*}
The symmetry property of $\Fom$ implies that
\begin{equation*}
    \dist_\Fom(w,w') = \sup_{h\in\N^*} \sup_{w^\top \text{ word}}
    \left| \#_h \Fom_\lef\left( w^\top ,\, w \right) 
    - \#_h \Fom_\lef\left( w^\top ,\, w' \right) \right|.
\end{equation*}

Let us conclude this section by explaining informally how we shall prove \Cref{th_derivative_per}.
By \Cref{prop_lambda_cv}, $\lamt^\mu$ is approximated with $\lambda^{\sigma_n}$, whose increments are described by the edge labels of Fomin's construction.
Therefore we aim to study the number of each label in the word $\Fom_\bot\big( w_n^\top,w_n^\rig \big)$ where $w_n^\top$ and $w_n^\rig$ are edge words of $\sigma_n$ on a local rectangle.
To do this, our technique is to simplify the top- and right-words without changing ``too much'' the resulting bottom- and left-words.
More precisely, we try to control the pseudo-distance $\dist_\Fom$, whose associated equivalence relation we call ``Fomin equivalence''.
In \Cref{section_proofs_Knuth} we prove that Fomin equivalence is implied by the well-known Knuth equivalence (for which we refer to Section 2.1 in \cite{F96} and Section 3.4 in \cite{S01}).
This is done by checking that Fomin equivalence is preserved by the elementary Knuth transformations, and we can then use the fundamental fact that Knuth equivalence means equality of the P-tableaux.
Independently, thanks to the derivability condition of \Cref{th_derivative_per}, the labels of $w_n^\top$ and $w_n^\rig$ are well equidistributed.
One can see that the P-tableau of a word with bounded, equidistributed letters is similar to the P-tableau of this word's decreasing reordering.
We deduce that $w_n^\top$ and $w_n^\rig$ are ``almost'' Fomin equivalent to their decreasing reorderings; for this we need the regularity property of $\dist_\Fom$ established in \Cref{section_regularity}.
This finally allows us to approximate the local increments of $\lambda^{\sigma_n}$ with the simpler case of decreasing words, studied in \Cref{section_proofs_phi}, and we conclude the proof of \Cref{th_derivative_per} in \Cref{section_proofs_diff}.

\subsection{Regularity of inverse Fomin dynamics}\label{section_regularity}

In this section we investigate a regularity property of the map $\Fom$.
This will allow us to control how much the output words of Fomin's inverse algorithm can change
when the input words are replaced by approximations.

Fix $r\in\N^*$.
We call \textit{block} a (possibly empty) word with nondecreasing integer letters between $0$ and $r$.
If $\ell\in\N^*$ and $\mathbf{q} := (q_{i,k})_{1\leq i\leq \ell , 0\leq k\leq r}\in\N^{\ell (r\p1)}$ is an array of nonnegative integers, define the word $\block(\mathbf{q})$ as the following concatenation of $\ell$ blocks:
\begin{equation*}
    \block(\mathbf{q}) :=
    ( 0^{q_{1,0}}1^{q_{1,1}}\dots r^{q_{1,r}} )
    \dots
    ( 0^{q_{i,0}}\dots k^{q_{i,k}}\dots r^{q_{i,r}} )
    \dots
    ( 0^{q_{\ell,0}}1^{q_{r,1}}\dots r^{q_{\ell,r}} ).
\end{equation*}
Also define a pseudo-distance on such arrays:
\begin{equation*}
    \text{for any }\mathbf{q,q'}\in\N^{\ell (r\p1)},\quad
    \delta\left( \mathbf{q} , \mathbf{q}' \right) :=
    \sum_{i=1}^\ell\sum_{k=1}^r \vert q_{i,k} - q_{i,k}' \vert 
    k \left( 8k^2 \right)^{\ell-i}.
\end{equation*}

\begin{lemma}\label{lem_block_distance_decrease}
    Let $\mathbf{q} \in \N^{\ell (r\p1)}, w:=\block(\mathbf{q})$, and $w^\rig$ be a word. 
    Then there exists a distinguished $\mathbf{\tilde{q}}\in\N^{\ell (r\p1)}$ such that $\Fom_\bot( w ,\, w^\rig ) = \block(\mathbf{\tilde{q}})$.
    With a slight abuse of notation, we denote this by $\Fom_\bot( \mathbf{q} ,\, w^\rig ) = \mathbf{\tilde{q}}$.
    Moreover for any array $\mathbf{q}'\in\N^{\ell (r\p1)}$:
    \begin{equation*}
        \delta\Big( 
        \Fom_\bot\big( \mathbf{q} ,\, w^\rig \big) , 
        \Fom_\bot\big( \mathbf{q}' ,\, w^\rig \big)
        \Big) \leq
        \delta\left( \mathbf{q} , \mathbf{q}' \right).
    \end{equation*}
\end{lemma}

This last bound means that the pseudo-distance $\delta$ can not increase after applying a word, regardless of this word.

\begin{proof}
    By induction it suffices to prove the lemma when $w^\rig = k$ is a single positive letter. 
    The word $\Fom_\bot( w,k )$ is obtained from $w$ by changing the first $k$ of some block to $k\mm1$, the first $k\mm1$ of another block to $k\mm2$, and so on.
    Since each new letter can be seen in the same block as the letter that produced it, this will allow us to write $\Fom_\bot( w,k ) = \block(\bf\tilde q)$ where $\bf\tilde q$ is obtained by doing local changes in $\bf q$.

    Denote by $h$ the number of decrements that happen when applying $k$ to $w$, and by $i_k \leq\dots\leq i_{k\m h\p 1}$ the block indices of these decrements.
    Since each block is nondecreasing, two decrements cannot happen in the same block, \textit{i.e.}~we have $i_k<i_{k\m1}<\dots<i_{k\m h\p1}$.
    We can thus write $\Fom_\bot( w ,k ) = \block(\mathbf{\tilde{q}})$ where $\mathbf{\tilde{q}}\in\N^{\ell (r\p1)}$ satisfies
    \begin{equation*}
        \tilde{q}_{i_k,k} = q_{i_k,k}\m1;\;
        \tilde{q}_{i_k,k\m1} = q_{i_k,k\m1}\p1;\quad
        \dots\quad;\quad
        \tilde{q}_{i_{k\m h\p1},k\m h\p1} = q_{i_{k\m h\p1},k\m h\p1}\m1;\;
        \tilde{q}_{i_{k\m h\p1},k\m h} = q_{i_{k\m h\p1},k\m h}\p1;
    \end{equation*}
    and everywhere else $\mathbf{q}$ and $\mathbf{\Tilde{q}}$ coincide.
    This establishes the first assertion of the lemma.

    For the second assertion write $w':=\block(\mathbf{q}')$ and $\mathbf{\tilde{q}}' := \Fom_\bot( \mathbf{q}',k )$.
    Then define analogously $h'\in\lb0,k\rb$ and $1\leq i_k'<i_{k\m1}'<\dots<i_{k\m h'\p1}' \leq \ell$.
    If the decrement indices are the same as for $\mathbf{q}$ \textit{i.e.}~$h'=h$ and $i_k=i_k'$, ..., $i_{k\m h\p1}=i_{k\m h\p1}'$, then $\delta\left( \mathbf{\tilde{q}} , \mathbf{\tilde{q}}' \right) = \delta\left( \mathbf{q} , \mathbf{q}' \right)$.
    The only possibility for the indices not to be the same is that some positive letter $j\in\lb1,k\rb$ in some block $i\in\lb1,k\rb$ of one word was decremented while the same block of the other word did not contain this letter.
    Suppose this happens and take $i$ minimal for this property. 
    Then
    \begin{equation*}
        \text{either}\quad
        \tilde{q}_{i,j} = q_{i,j}\m1 
        \text{ and }
        \tilde{q}_{i,j}' = q_{i,j}'=0
        ,\qquad\text{or}\quad
        \tilde{q}_{i,j}' = q_{i,j}'\m1 
        \text{ and }
        \tilde{q}_{i,j} = q_{i,j}=0.
    \end{equation*}
    Only one of these occurs, and by symmetry we might suppose the first one does.
    Minimality of $i$ means that applying $k$ to $w$ and $w'$ starts with the same dynamics, up until their $i$-th block where the first $j$ of $w$ becomes $j\mm1$ and $j\mm1$ is applied to the rest of $w$ while this block in $w'$ remains unchanged and $j$ is applied to the rest of $w'$.
    These changes in the $i$-th block yield the equations:
    \begin{equation}\label{changes_block_i}
        | \tilde{q}_{i,j} - \tilde{q}_{i,j}' | = | q_{i,j} - q_{i,j}' | -1 
        \quad;\quad
        | \tilde{q}_{i,j\m1} - \tilde{q}_{i,j\m1}' | \leq | q_{i,j\m1} - q_{i,j\m1}' | +1.
    \end{equation}
    In the following blocks, at most $2(j\mm1)$ values may change in $\mathbf{q}$ and at most $2j$ may in $\mathbf{q}'$.
    This yields a family of ``disagreement'' indices $\mathcal{A}\subseteq \lb i\pp1,\ell\rb \times \lb1,j\rb$ of size $|\mathcal{A}|\leq 4j$ and such that:
    \begin{equation}\label{changes_block_idem}
        \text{for any }(i',j')\in \lb1,\ell\rb \times \lb1,r\rb \setminus 
            \left( \mathcal{A} \cup \{(i,j)\} \cup \{(i,j\m1)\} \right),\quad
            |\tilde{q}_{i,j} - \tilde{q}_{i,j}'| = |q_{i,j} - q_{i,j}'|.
    \end{equation}
    On disagreement indices we use the simple bound:
    \begin{equation}\label{changes_block_disagree}
        \text{for any }(i',j')\in \mathcal{A},\quad
            | \tilde{q}_{i',j'} - \tilde{q}_{i',j'}' |
            \leq | q_{i',j'} - q_{i',j'}' | +2.
    \end{equation}
    Using \Cref{changes_block_i,changes_block_idem,changes_block_disagree} we finally deduce that:
    \begin{multline*}
        \delta\left( \mathbf{\tilde{q}} , \mathbf{\tilde{q}}' \right)
        \leq \delta\left( \mathbf{q} , \mathbf{q}' \right)
        - j\left( 8j^2 \right)^{\ell-i} 
        + (j\m1)\left( 8(j\m1)^2 \right)^{\ell-i}
        + 8j \cdot j\left( 8j^2 \right)^{\ell-i-1}
        \\\leq \delta\left( \mathbf{q} , \mathbf{q}' \right)
        - j\left( 8j^2 \right)^{\ell-i} 
        + (j\m1)\left( 8j^2 \right)^{\ell-i}
        + 8j \cdot j\left( 8j^2 \right)^{\ell-i-1}
        \leq \delta\left( \mathbf{q} , \mathbf{q}' \right)
    \end{multline*}
    as desired.
\end{proof}

\begin{cor}\label{cor_regularite_Fomin_inverse_sur_blocs}
    Let $\mathbf{q},\mathbf{q}'\in\N^{\ell (r\p1)}$ and $w:=\block(\mathbf{q}) , w':=\block(\mathbf{q}')$. The following holds:
    \begin{equation*}
        \dist_\Fom(w,w') \leq \delta(\mathbf{q},\mathbf{q}').
    \end{equation*}
\end{cor}

\begin{proof}
    Fix $h\in\N^*$ and $w^\rig$ a word.
    Write $\mathbf{\tilde{q}} := \Fom_\bot( \mathbf{q} ,\, w^\rig )$ and $\mathbf{\tilde{q}}' := \Fom_\bot( \mathbf{q}' ,\, w^\rig )$.
    Then:
    \begin{align*}
        \left| \#_h \Fom_\bot( w ,\ w^\rig )
        - \#_h \Fom_\bot( w' ,\, w^\rig )  \right|
        = \left| \sum_{i=1}^\ell \tilde{q}_{i,h}
        - \sum_{i=1}^\ell \tilde{q}_{i,h}' \right|
        \leq \delta\left( \mathbf{\tilde{q}} , \mathbf{\tilde{q}}' \right)
        \leq \delta\left( \mathbf{q} , \mathbf{q}' \right)
    \end{align*}
    where the last inequality is given by \Cref{lem_block_distance_decrease}.
\end{proof}

\subsection{Fomin and Knuth equivalence}\label{section_proofs_Knuth}

We say that two words $w , w'$ are \textit{Fomin equivalent} or simply equivalent, denoted $w\equiv w'$, when $\dist_\Fom(w,w')=0$.
This means that applying any word to $w$ and $w'$ yields resulting words with the same amount of each non-zero label.
In particular two equivalent words have the same number of $k$'s for any $k\in\N^*$, but not necessarily the same number of $0$'s. 

\begin{lemma}\label{lem_factorisation_equivalence}
    Let $w,w',w^\rig$ be words. If $w\equiv w'$ then:
    \begin{equation*}
        \Fom_\bot\left( w ,\, w^\rig \right) \equiv \Fom_\bot\left( w' ,\, w^\rig \right).
    \end{equation*}
\end{lemma}

\begin{proof}
    It suffices to observe that for any word $w^\rig_2$:
    \begin{equation*}
        \Fom_\bot\left( \Fom_\bot\left( w ,\, w^\rig \right) ,\, w_2^\rig \right)
        = \Fom_\bot\left( w ,\, w^\rig w_2^\rig \right)
    \end{equation*}
    and likewise with $w'$.
    The lemma then follows directly from definition.
\end{proof}

\begin{lemma}\label{lem_ajout_equivalence}
    If $w,w'$ are words and $i\in\N$:
    \begin{equation*}
        w \equiv w' \implies 
        i\cdot w \equiv i\cdot w'
        \text{ and }
        w\cdot i \equiv w'\cdot i
    \end{equation*}
\end{lemma}

\begin{proof}
    We aim to prove that for any word $w^\rig$ the following holds: for any pair of equivalent words $w\equiv w'$, any $i\in\N$ and any $h\in\N^*$,
    \begin{equation}\label{induction_equivalence_configurations}
        \#_h \Fom_\bot\left( i\cdot w ,\, w^\rig\right) 
        = \#_h \Fom_\bot\left( i\cdot w' ,\, w^\rig\right)
        \quad\text{and}\quad
        \#_h \Fom_\bot\left( w\cdot i ,\, w^\rig\right) 
        = \#_h \Fom_\bot\left( w'\cdot i ,\, w^\rig\right).
    \end{equation}
    We do this by induction on the length of $w^\rig$. 
    If $w^\rig$ is empty, (\ref{induction_equivalence_configurations}) is immediate since equivalent words have the same amount of each positive letter. 
    Now suppose this holds for some word $w^\rig$ and let us prove it for every word $j\cdot w^\rig$ with $j\in\N$. Fix $w\equiv w'$ and $i\in\N$.

    For the first equality of (\ref{induction_equivalence_configurations}) we distinguish between two cases.
    If $i = j$ then, using Property (\ref{letter_by_letter_property}):
    \begin{equation}\label{before_equivalence}
        \Fom_\bot\left( i\cdot w ,\, j\cdot w^\rig\right)
        = \Fom_\bot\left( i\cdot w ,\, i\cdot w^\rig\right)
        = \Fom_\bot\left( (i\mm1)^+\cdot \Fom_\bot\left(w ,\, (i\mm1)^+\right) ,\, w^\rig\right)
    \end{equation}
    where we used the standard notation $(i\m1)^+ = \max(i\m1,0)$, and likewise for $w'$. 
    See the left part of \Cref{fig_ajout_equivalence} for a representation.
    \Cref{lem_factorisation_equivalence} implies $\Fom_\bot\left(w ,\, (i\mm1)^+\right) \equiv \Fom_\bot\left(w' ,\, (i\mm1)^+\right)$ and the first part of (\ref{induction_equivalence_configurations}) follows by induction.

    Otherwise $i \neq j$ and
    \begin{equation*}
        \Fom_\bot\left( i\cdot w ,\, j\cdot w^\rig\right)
        = \Fom_\bot\left( i\cdot \Fom_\bot(w,j) ,\, w^\rig\right),
    \end{equation*}
    likewise for $w'$. \Cref{lem_factorisation_equivalence} implies $\Fom_\bot(w ,\, j) \equiv \Fom_\bot(w' ,\, j)$ and the first part of (\ref{induction_equivalence_configurations}) follows once again by induction.

    \smallskip
    
    For the second equality of (\ref{induction_equivalence_configurations}), observe that applying $j$ to $w\cdot i$ decrements a number of letters in $w$, say $k\in\lb0,j\rb$, and then applies $j\m k = \Fom_\lef(w,j)$ to $i$ (resp. $w'$, $k'$).
    Therefore
    \begin{equation}\label{apply_letter_to_end}
        \Fom_\bot(w\cdot i ,\, j) =  \Fom_\bot(w ,\, j) \cdot \Fom_\bot(i ,\, j\m k)
        \quad\text{and}\quad
        \Fom_\bot(w'\cdot i ,\, j) =  \Fom_\bot(w' ,\, j) \cdot \Fom_\bot(i ,\, j\m k').
    \end{equation}
    Since $k$ is a number of decrements, it can be expressed as the sum of all letters of $w$ minus the sum of all letters of $\Fom_\bot(w,\, j)$.
    As such:
    \begin{equation*}
        k = \sum_{h\geq1} h \,\#_h w - h\,\#_h \Fom_\bot(w,j)
        = \sum_{h\geq1} h \,\#_h w' - h\,\#_h \Fom_\bot(w',j) = k'
    \end{equation*}
    where the second equality is a consequence of $w\equiv w'$.
    Finally, using (\ref{apply_letter_to_end}) and Property (\ref{letter_by_letter_property}) (see the right part of \Cref{fig_ajout_equivalence}):
    \begin{equation}\label{after_equivalence}
        \left\{
        \begin{array}{ll}
            \Fom_\bot\left( w\cdot i ,\, j\cdot w^\rig\right)
            = \Fom_\bot\left( \Fom_\bot(w ,\, j) \cdot i' ,\, w^\rig\right) ;\smallskip\\
            \Fom_\bot\left( w'\cdot i ,\, j\cdot w^\rig\right)
            = \Fom_\bot\left( \Fom_\bot(w' ,\, j) \cdot i' ,\, w^\rig\right);
        \end{array}
        \right.
    \end{equation}
    where $i'= \Fom_\bot(i,j\m k)$, and we conclude by induction.
\end{proof}

\begin{figure}
    \centering
    \begin{minipage}{0.49\linewidth}
        \centering
        \includegraphics[scale=0.56]{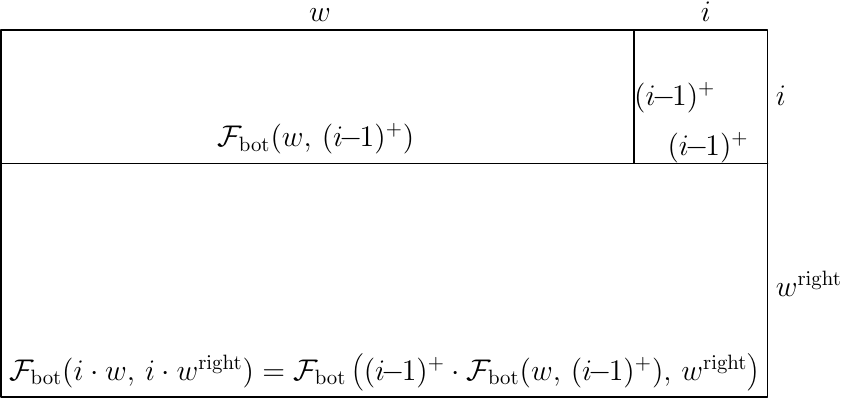}
    \end{minipage}
    \hspace{0.35em}
    \begin{minipage}{0.49\linewidth}
        \centering
        \includegraphics[scale=0.56]{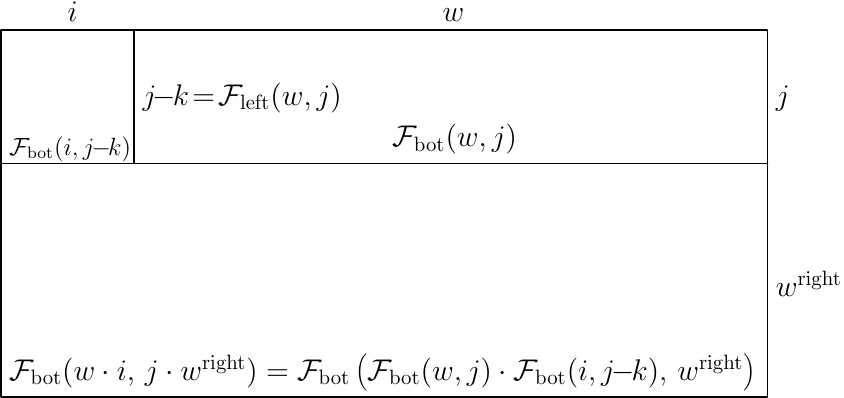}
    \end{minipage}
    \caption{Illustration of equalities (\ref{before_equivalence}) and (\ref{after_equivalence}) used to prove \Cref{lem_ajout_equivalence}.}
    \label{fig_ajout_equivalence}
\end{figure}

\begin{cor}\label{cor_concatenation_equivalence}
    For any words $w_1,w_2,w_1',w_2'$ such that $w_1 \equiv w_1'$ and $w_2 \equiv w_2'$, we have $w_1 w_2 \equiv w_1' w_2'$.
\end{cor}

\begin{proof}
    Successive application of \Cref{lem_ajout_equivalence} yield both $w_1 w_2 \equiv w_1 w_2'$ and $w_1' w_2' \equiv w_1 w_2'$, whence the result.
\end{proof}

Using $0\equiv\emptyset$, this yields the following fact (which was actually quite clear from the beginning):
 
\begin{cor}\label{cor_equivalent_without_zeros}
    Any word is equivalent to itself without its $0$'s.
\end{cor}

\begin{lemma}\label{lem_equivalence_jik_jki}
    For any integers $i<j\leq k$, $j\,i\,k \equiv j\,k\,i$.
\end{lemma}

\begin{proof}
    If $i=0$ then this is a direct consequence of \Cref{cor_equivalent_without_zeros}.
    Now suppose $i\geq1$.
    The strategy is the same as for the proof of \Cref{lem_ajout_equivalence}.
    We aim to prove that for any word $w^\rig$ the following holds: for any integers $i<j\leq k$ and any $h\in\N^*$,
    \begin{equation}\label{induction_jik_jki}
        \#_h \Fom_\bot\left( j\,i\,k ,\, w^\rig\right) 
        = \#_h \Fom_\bot\left( j\,k\,i ,\, w^\rig\right).
    \end{equation}
    We do this by induction on the length of $w^\rig$. 
    If $w^\rig$ is empty, (\ref{induction_jik_jki}) is immediate. 
    Now suppose this holds for some word $w^\rig$ and let us prove it for every word $m\cdot w^\rig$ with $m\in\N$. 
    
    If $m = i$ then 
    \begin{equation*}
        \Fom_\bot\left( j\,i\,k ,\, m\cdot w^\rig\right)
        = \Fom_\bot\left( j\;i\mm1\;k ,\, w^\rig\right)
        \; ; \;
        \Fom_\bot\left( j\,k\,i ,\, m\cdot w^\rig\right)
        = \Fom_\bot\left( j\;k\;i\mm1 ,\, w^\rig\right).
    \end{equation*}
    Since $i\mm1<j\leq k$, (\ref{induction_jik_jki}) follows by induction.

    If $m = j$ then 
    \begin{equation*}
        \Fom_\bot\left( j\,i\,k ,\, m\cdot w^\rig\right)
        = \Fom_\bot\left( j\mm1\;i'\;k ,\, w^\rig\right)
        \; ; \;
        \Fom_\bot\left( j\,k\,i ,\, m\cdot w^\rig\right)
        = \Fom_\bot\left( j\mm1\;k\;i' ,\, w^\rig\right)
    \end{equation*}
    where $i'= i\wedge(j\mm2)$.
    Since $i'<j\mm1\leq k$, (\ref{induction_jik_jki}) follows by induction.

    If $m = k$ we distinguish between two cases.
    If $j=k$ then $m=j$ and this case was already treated.
    If $j\leq k\mm1$ then $i\leq k\mm2$ and
    \begin{equation*}
        \Fom_\bot\left( j\,i\,k ,\, m\cdot w^\rig\right)
        = \Fom_\bot\left( j\;i\;k\mm1 ,\, w^\rig\right)
        \; ; \;
        \Fom_\bot\left( j\,k\,i ,\, m\cdot w^\rig\right)
        = \Fom_\bot\left( j\;k\mm1\;i ,\, w^\rig\right).
    \end{equation*}
    Since $i<j\leq k\mm1$, (\ref{induction_jik_jki}) follows by induction.

    Otherwise $m \notin \{i,j,k\}$ and
    \begin{equation*}
        \Fom_\bot\left( j\,i\,k ,\, m\cdot w^\rig\right)
        = \Fom_\bot\left( j\,i\,k ,\, w^\rig\right),
    \end{equation*}
    likewise for $j\,k\,i$. (\ref{induction_jik_jki}) follows once again by induction, and this concludes the proof.
\end{proof}

\begin{lemma}\label{lem_equivalence_ikj_kij}
    For any integers $i\leq j< k$, $i\,k\,j \equiv k\,i\,j$.
\end{lemma}

\begin{proof}
    If $i=0$ then this is a direct consequence of \Cref{cor_equivalent_without_zeros}.
    Now suppose $i\geq1$.
    We proceed as before and aim to prove that for any word $w^\rig$ the following holds: for any integers $i\leq j< k$ and any $h\in\N^*$,
    \begin{equation}\label{induction_ikj_kij}
        \#_h \Fom_\bot\left( i\,k\,j ,w^\rig\right)
        = \#_h \Fom_\bot\left( k\,i\,j ,w^\rig\right) .
    \end{equation}
    We do this by induction on the length of $w^\rig$. 
    If $w^\rig$ is empty, (\ref{induction_ikj_kij}) is immediate.
    Now suppose this holds for some word $w^\rig$ and let us prove it for every word $m\cdot w^\rig$ with $m\in\N$. 
    
    The cases $m=i$, $m=j$ and $m \notin \{i,j,k\}$ are handled just like in the proof of \Cref{lem_equivalence_jik_jki}.
    The only subtlety comes when $m = k$.
    In this case, if $i=k\mm1$ then $j=k\mm1$ and
    \begin{equation*}
        \Fom_\bot\left( i\,k\,j , m\cdot w^\rig\right)
        = \Fom_\bot\left( k\mm1\;k\mm1\;k\mm2 , w^\rig\right)
    \end{equation*}
    while
    \begin{equation*}
        \Fom_\bot\left( k\,i\,j , m\cdot w^\rig\right)
        = \Fom_\bot\left( k\mm1\;k\mm2\;k\mm1 , w^\rig\right).
    \end{equation*}
    We can thus apply \Cref{lem_equivalence_jik_jki} to get $k\mm1\;k\mm1\;k\mm2 \equiv k\mm1\;k\mm2\;k\mm1$, and directly conclude.
    If $i < k\mm1$ then
    \begin{equation*}
        \Fom_\bot\left( i\,k\,j , m\cdot w^\rig\right)
        = \Fom_\bot\left( i\,k\mm1\,j' , w^\rig\right)
        \; ; \;
        \Fom_\bot\left( k\,i\,j , m\cdot w^\rig\right)
        = \Fom_\bot\left( k\mm1\,i\,j' , w^\rig\right)
    \end{equation*}
    where $j'=j\wedge (k\mm2)$.
    Since $i\leq j'< k$, (\ref{induction_ikj_kij}) follows by induction.
\end{proof}

\Cref{lem_equivalence_ikj_kij,lem_equivalence_jik_jki} together with \Cref{cor_concatenation_equivalence} tell us transforming within a word any triplet $j\,i\,k$ where $i<j\leq k$ into $j\,k\,i$ or $i\,k\,j$ where $i\leq j<k$ into $k\,i\,j$ preserves Fomin equivalence.
In other words, two adjacent letters can be switched if they are followed or preceded by an intermediate value, and equality cases are handled by considering that among two letters with the same value, the greatest is the one that comes last in the word.
These transformations are known as \textit{elementary Knuth transformations}. 
Two words are said to be \textit{Knuth equivalent} when they differ by successive elementary transformations.
We have thus proved:

\begin{prop}\label{prop_Knuth_implique_Fomin}
    If two finite words are Knuth equivalent then they are Fomin equivalent.
\end{prop}

We have not been able to determine whether the converse is true or not, but we believe this could be an interesting question.
Now, we shall use the following well-known result about Knuth equivalence:

\begin{theorem}[\cite{F96,S01}]\label{th_Knuth}
    Two words are Knuth equivalent if and only if they have the same P-tableau.
\end{theorem}

This theorem is the last puzzle piece we needed to combine our previous results into \Cref{prop_approximation_par_decreasing_si_bonnes_LIS}, which shall prove useful in the permuton setting.
First we state some key observation.

\medskip

Fix $r\in\N^*$.
Let $w$ be a word on letters between $1$ and $r$ and $P(w)$ its P-tableau.
For any $i\in\lb1,r\rb$ and $k\in\lb0,r\rb$, denote by $q_{i,k}(w)$ the number of $k$'s in the $(r\mm i\pp1)$-st row of $P(w)$
(notice that $q_{i,k}(w)=0$ whenever $i>k$).
By \Cref{th_Knuth}, the array 
$$\mathbf{q}(w) = \big(q_{i,k}(w)\big)_{1\leq i\leq r, 0\leq k\leq r} \in\N^{r(r\p1)}$$
characterizes the Knuth equivalence class of $w$.
Moreover \Cref{th_Greene} allows to describe this array in terms of nondecreasing subsequences of $w$.
Indeed, writing $w^{\leq k}$ for the sub-word of $w$ consisting of its letters with values at most $k$, recall that the P-tableau of $w^{\leq k}$ is the restriction of $P(w)$ to its values at most $k$. 
Thus $\LIS_i\left(w^{\leq k}\right)$ is the number of boxes with values at most $k$ in the first $i$ rows of $P(w)$, and we deduce:
\begin{equation}\label{lien_LIS_tableau}
    q_{r\m i\p1,k}(w) =
    \LIS_i\left(w^{\leq k}\right) - \LIS_i\left(w^{\leq k\m1}\right) -
    \LIS_{i\m1}\left(w^{\leq k}\right) + \LIS_{i\m1}\left(w^{\leq k\m1}\right).
\end{equation}
With the notation of \Cref{section_proofs_Fomin}, the word $\block\big(\mathbf{q}(w)\big)$ is the concatenation of all rows of $P(w)$, from last to first.
It is usually called the \textit{row word} of $P(w)$, and one can see that its P-tableau is $P(w)$ (see Lemma 3.4.5 in \cite{S01} or Section 2.1 in \cite{F96}).
Hence $\block\big(\mathbf{q}(w)\big)$ is a distinguished word in the Knuth equivalence class of $w$.
These facts lead to the following result:

\begin{prop}\label{prop_approximation_par_decreasing_si_bonnes_LIS}
    Fix $r\in\N^*$, $\gamma_1,\dots,\gamma_r\in\N$ and set $w_\ord := r^{\gamma_r}\dots1^{\gamma_1}$. 
    For $k\in\lb1,r\rb$ denote by $\gamma_{(1)}^{k} \geq\dots\geq \gamma_{(k)}^{k}$ the weakly decreasing order statistics of $\gamma_1,\dots,\gamma_k$. 
    Let $w$ be a word on letters between $1$ and $r$ and define for each $i\in\lb1,r\rb$:
    \begin{equation*}
        \eta_{i,k} :=
        \left| \gamma_{(1)}^{k} +\dots+ \gamma_{(i)}^{k} 
        - \LIS_i\left(w^{\leq k}\right) \right|
    \end{equation*}
    where $\gamma_{(i)}^{k} := 0$ if $i>k$. 
    Then we have the following upper bound:
    \begin{equation*}
        \dist_\Fom\left( w,w_\ord \right) \leq 
        4 r (8r^2)^r \max_{1\leq i,k\leq r}\eta_{i,k}.
    \end{equation*}
\end{prop}

\begin{proof}
    Let $i,k\in\lb1,r\rb$. Since any nondecreasing sequence of $w_\ord$ can only have one type of letter, it is easy to see that
    \begin{equation*}
        \LIS_i\left( w_\text{ord}^{\leq k} \right) 
        = \gamma_{(1)}^{k} +\dots+ \gamma_{(i)}^{k}.
    \end{equation*} 
    Thus \Cref{lien_LIS_tableau} yields
    \begin{equation}\label{bound_q_chaos}
        \left| q_{r\m i\p1,k}(w) - q_{r\m i\p1,k}\left( w_\ord \right) \right|
        \leq \eta_{i,k} + \eta_{i,k\m1} + \eta_{i\m1,k} + \eta_{i\m1,k\m1}.
    \end{equation} 
    Since $w$ is Knuth equivalent to $\block\big(\mathbf{q}(w)\big)$ and likewise for $w_\ord$, \Cref{prop_Knuth_implique_Fomin} implies
    \begin{equation*}
        \dist_\Fom\left( w ,\, w_\ord \right)
        = \dist_\Fom\Big( \block\big(\mathbf{q}(w)\big) ,\, \block\big(\mathbf{q}(w_\ord)\big) \Big).
    \end{equation*}
    We can then apply \Cref{cor_regularite_Fomin_inverse_sur_blocs} with (\ref{bound_q_chaos}) to obtain:
    \begin{equation*}
        \dist_\Fom\left( w ,\, w_\ord \right)
        \leq \delta\big( \mathbf{q}(w) , \mathbf{q}(w_\ord) \big)
        \leq 4 r (8r^2)^r \max_{1\leq i,k\leq r}\eta_{i,k} .
        \qedhere
    \end{equation*}
\end{proof}

\subsection{Study of decreasing labels}\label{section_proofs_phi}

For any $r\in\N^*$ and $(\alpha_k)_{1\leq k\leq r} , (\beta_k)_{1\leq k\leq r} \in \N^r$, define
\begin{equation}\label{def_phi}
    \phi\big( (\alpha_k)_{1\leq k\leq r} ,\, (\beta_k)_{1\leq k\leq r} \big)
    := \beta_1 + \#_1 \Fom_\bot\left( r^{\alpha_r}\dots1^{\alpha_1} ,\, r^{\beta_r}\dots1^{\beta_1} \right).
\end{equation}
For example $\phi\big( (\alpha_1),(\beta_1) \big) = \beta_1 + \left( \alpha_1-\beta_1 \right)^+ = \alpha_1\vee\beta_1$, and
\begin{equation*}
    \phi\big( (\alpha_1,\alpha_2) , (\beta_1,\beta_2) \big)
    = \beta_1 + \left( \left( \alpha_1 - \left( \alpha_2 \wedge \beta_2 \right) \right)^+ + \left( \alpha_2 \wedge \beta_2 \right) - \beta_1 \right)^+
    = \alpha_1 \vee \beta_1 \vee \left( \alpha_2 \wedge \beta_2 \right).
\end{equation*}
Notice that the index $r$ can be omitted on $\phi$, since
\begin{equation*}
    \phi\big( (\alpha_1,\dots,\alpha_r,0,\dots,0) ,\, (\beta_1,\dots,\beta_r,0,\dots,0) \big) 
    = \phi\big( (\alpha_1,\dots,\alpha_r) ,\, (\beta_1,\dots,\beta_r) \big).
\end{equation*}
This section is devoted to the study of this function and its continuous extension to sequences of real numbers.

\begin{prop}\label{prop_Fbot_real_words_and_continuity_of_phi}
For each $r\geq1$, the function $\phi$ has a unique continuous extension from $\left(\N^r\right)^2$ to $\left(\R_+^r\right)^2$ satisfying the following homogeneity property:
    \begin{equation}\label{homogeneity_phi}
        \text{for any }c\in\R_+ 
        \text{ and }(\alpha_k)_{1\leq k\leq r} , (\beta_k)_{1\leq k\leq r} \in \R_+^{r},\quad
        \phi\big( (c\alpha_k) \,,\, (c\beta_k) \big)
        = c \phi\big( (\alpha_k) \,,\, (\beta_k) \big).
    \end{equation}
\end{prop}

\begin{proof}
    For stability reasons, we prove this by studying the dynamics of block words as defined in \Cref{section_proofs_Fomin}.
    
    Let $r\geq1$ and $\mathbf{a} = (a_0,\dots,a_r), \mathbf{b} = (b_0,\dots,b_r)\in\N^{r+1}$.
    Define
    \begin{align}\label{varphi_condition_initiale}
        \varphi\left( \mathbf{a} , \mathbf{b} \right) 
        \nonumber &= \left( \varphi_0(\mathbf{a},\mathbf{b}) , \varphi_1(\mathbf{a},\mathbf{b}) , \dots , \varphi_{r-1}(\mathbf{a},\mathbf{b}) , \varphi_r(\mathbf{a},\mathbf{b}) \right)
        \\&:= \left(
        a_0 \p (a_1 \wedge b_1) ,\;
        \left(a_1 \m b_1\right)^+ \p (a_2 \wedge b_2) ,\;
        \dots,\;
        \left(a_{r\m1} \m b_{r\m1}\right)^+ \p (a_r \wedge b_r) ,\;
        \left(a_r \m b_r\right)^+
        \right).
    \end{align}
    One can then see that the following holds:
    \begin{equation*}\label{Fom_condition_initiale}
        \left\{
        \begin{array}{ll}
            \Fom_\bot\left( 0^{a_0}\dots r^{a_r} ,\, 0^{b_0}\dots r^{b_r} \right) :=
            0^{\varphi_0\left( \mathbf{a} ,
            \mathbf{b} \right)}
            \dots r^{\varphi_r\left( \mathbf{a} ,
            \mathbf{b} \right)} ;\\
            \Fom_\lef\left( 0^{a_0}\dots r^{a_r} ,\, 0^{b_0}\dots r^{b_r} \right) :=
            0^{\varphi_0\left( \mathbf{b} ,
            \mathbf{a} \right)}
            \dots r^{\varphi_r\left( \mathbf{b} ,
            \mathbf{a} \right)}.
        \end{array}
        \right.
    \end{equation*}
    Now let us study the effect of applying a single block to a sequence of blocks.
    For any $\ell\in\N^*$ and any array $\mathbf{q} = (q_{i,k})_{1\leq i\leq \ell,0\leq k\leq r} \in\N^{\ell (r\p1)}$, one has:
    \begin{align*}
        &\Fom_\bot\left( 0^{q_{1,0}}\dots r^{q_{1,r}} \,\dots\, 0^{q_{\ell,0}}\dots r^{q_{\ell,r}} ,\, 
        0^{b_0}\dots r^{b_r} \right)
        \\&= \Fom_\bot\left( 0^{q_{1,0}}\dots r^{q_{1,r}} ,\, 
        0^{b_0}\dots r^{b_r} \right)
        \cdot
        \Fom_\bot\Big( 0^{q_{2,0}}\dots r^{q_{2,r}} \,\dots\, 0^{q_{\ell,0}}\dots r^{q_{\ell,r}} ,\,
        \Fom_\lef\left( 0^{q_{1,0}}\dots r^{q_{1,r}} ,\, 
        0^{b_0}\dots r^{b_r} \right) \Big)
        \\&= 0^{\varphi_0\left( \mathbf{q_1} , \mathbf{b} \right)}\dots r^{\varphi_r\left( \mathbf{q_1} , \mathbf{b} \right)}
        \cdot \Fom_\bot\left( 0^{q_{2,0}}\dots r^{q_{2,r}} \,\dots\, 0^{q_{\ell,0}}\dots r^{q_{\ell,r}} ,\,
        0^{\varphi_0\left( \mathbf{b} , \mathbf{q_1} \right)}\dots r^{\varphi_r\left( \mathbf{b} , \mathbf{q_1}  \right)} \right).
    \end{align*}
    where $\mathbf{q_1} := \left( q_{1,k} \right)_{0\leq k\leq r}$.
    We see by induction that this resulting word is yet again a concatenation of $\ell$ blocks.
    With a slight abuse of notation, define $\varphi\left( \mathbf{q} , \mathbf{b} \right) = \big( \varphi_{i,k}\left( \mathbf{q} , \mathbf{b} \right) \big)_{1\leq i\leq \ell,0\leq k\leq r} := \Fom_\bot\left( \mathbf{q} , 0^{b_0}\dots r^{b_r} \right)$
    with the notation of \Cref{lem_block_distance_decrease}.
    Last calculation showed the following recursive formula for $\varphi$:
    \begin{equation}\label{varphi_recurrence_1}
        \left\{
        \begin{array}{ll}
        \varphi_{1,\cdot}\left( 
        \mathbf{q} , \mathbf{b} 
        \right)
        = \varphi\left( 
        \mathbf{q_1} , \mathbf{b}
        \right) \smallskip\\
        \big( \varphi_{i,k}\left( 
        \mathbf{q} , \mathbf{b} \right)
        \big)_{2\leq i\leq \ell, 0\leq k\leq r}
        = \varphi\left( 
        \mathbf{q} \setminus \mathbf{q_1} \,,\,
        \varphi\big( \mathbf{b} , \mathbf{q_1} \big)
        \right)
        \end{array}
        \right.
    \end{equation}
    where $\mathbf{q} \setminus \mathbf{q_1} := \left(q_{i,k}\right)_{2\leq i\leq \ell, 0\leq k\leq r}$.
    Finally for any $\ell'\in\N^*$ and $\mathbf{q}' = (q_{i,k}')_{1\leq i\leq \ell',0\leq k\leq r} \in\N^{\ell' (r\p1)}$, define $\varphi\left( \mathbf{q} , \mathbf{q}' \right) = \big( \varphi_{i,k}\left( \mathbf{q} , \mathbf{q}' \right) \big)_{1\leq i\leq \ell,0\leq k\leq r} := \Fom_\bot\big( \mathbf{q} , \block(\mathbf{q}') \big)$.
    Then $\varphi$ satisfies the following recursion if $\ell'\geq2$:
    \begin{equation}\label{varphi_recurrence_2}
        \varphi\left( \mathbf{q} , \mathbf{q}' \right)
        = \varphi\big( \varphi \left( \mathbf{q} , \mathbf{q'_1} \right) \,,\, \mathbf{q}'\setminus\mathbf{q'_1} \big).
    \end{equation}
    One advantage of this point of view is that \Cref{varphi_condition_initiale,varphi_recurrence_1,varphi_recurrence_2} can be used to continuously extend $\varphi$ from $\N^{\ell (r\p1)} \times \N^{\ell' (r\p1)}$ to $\R_+^{\ell (r\p1)} \times \R_+^{\ell' (r\p1)}$.
    Then $\phi$ can be continuously extended to $\left(\R_+^r\right)^2$ by letting, for any $r\geq 1$ and $(\alpha_k)_{1\leq k\leq r} , (\beta_k)_{1\leq k\leq r} \in\R_+^r$:
    \begin{equation*}
        \phi\big( (\alpha_k)_{1\leq k\leq r} , (\beta_k)_{1\leq k\leq r} \big)
        = \beta_1 + \sum_{i=1}^r \varphi_{i,1}( \mathbf{q}_\alpha , \mathbf{q}_\beta )
    \end{equation*}
    where $\mathbf{q}_\alpha,\mathbf{q}_\beta \in \R_+^{r(r\p1)}$ are such that $r^{\alpha_r}\dots1^{\alpha_1} = \block\left( \mathbf{q}_\alpha \right)$ and $r^{\beta_r}\dots1^{\beta_1} = \block\left( \mathbf{q}_\beta \right)$.
    From homogeneity of $\varphi$, seen in (\ref{varphi_condition_initiale}) and conserved in \Cref{varphi_recurrence_1,varphi_recurrence_2}, 
    we deduce homogeneity of $\phi$ \textit{i.e.}~Property (\ref{homogeneity_phi}).

    For the uniqueness claim, observe that \Cref{def_phi,homogeneity_phi} serve to characterize $\phi$ on each $\left(\Q_+^r\right)^2$, and continuity implies uniqueness on $\left(\R_+^r\right)^2$.
\end{proof}

Many properties of $\phi$ can be extended, by homogeneity and continuity, from $\left(\N^r\right)^2$ to $\left(\R_+^r\right)^2$.
For instance: 
\begin{equation}\label{eq_phi_one_coordinate}
\text{for any } \alpha,\beta \in \R_+ ,\quad
\phi\big( (\alpha) , (\beta) \big) = \alpha \vee \beta .
\end{equation}
One can also see that $\Fom_\bot\left( (r+1)^{\alpha_{r+1}}\, r^{\alpha_r} \dots 1^{\alpha_1} , r^{\beta_r} \dots 1^{\beta_1} \right) = (r+1)^{\alpha_{r+1}} \cdot \Fom_\bot\left( r^{\alpha_r} \dots 1^{\alpha_1} , r^{\beta_r} \dots 1^{\beta_1} \right)$ and $\Fom_\bot\left( r^{\alpha_r} \dots 1^{\alpha_1} , (r+1)^{\beta_{r+1}}\, r^{\beta_r} \dots 1^{\beta_1} \right) = \Fom_\bot\left( r^{\alpha_r} \dots 1^{\alpha_1} , r^{\beta_r} \dots 1^{\beta_1} \right)$, therefore for any $r\in\N^*$ and any $(\alpha_k)_{1\leq k\leq r} , (\beta_k)_{1\leq k\leq r} \in \R_+^{r}$:
\begin{multline}\label{eq_phi_crossed_zero}
\phi\big( (\alpha_1,\dots,\alpha_r,\alpha_{r+1}) \,,\, (\beta_1,\dots,\beta_r,0) \big)
= \phi\big( (\alpha_1,\dots,\alpha_r,0) \,,\, (\beta_1,\dots,\beta_r,\beta_{r+1}) \big)
\\= \phi\big( (\alpha_1,\dots,\alpha_r) \,,\, (\beta_1,\dots,\beta_r) \big) .
\end{multline}
We can use the last two properties of $\phi$ to prove a third one which we will use to establish \Cref{cor_gradient_lambda}.

\begin{lemma}\label{lem_phi_si_diff}
	Let $r\in\N^*$.
	Consider $(\alpha_k)_{1\leq k\leq r} , (\beta_k)_{1\leq k\leq r} \in \R_+^{r}$ such that
	\begin{equation}\label{eq_phi_si_diff}
	\phi\big( (\alpha_i)_{k\le i\le r} , (\beta_i)_{k\le i\le r} \big) = \alpha_k + \beta_k
	\end{equation}
	for all $1\le k\le r$.
	Then $\alpha_k \cdot \beta_k = 0$ for all $1\le k\le r$.
\end{lemma}

\begin{proof}
	Let us prove the lemma by induction on $r\in\N$.
	For $r=1$, \eqref{eq_phi_one_coordinate} implies $\alpha_{1}\vee\beta_{1} = \alpha_{1} + \beta_{1}$, i.e.~$\alpha_{1} \cdot \beta_{1} = 0$.
	Now let $r>1$ and assume that the lemma holds for $r-1$.
	Use \Cref{eq_phi_si_diff} for $k=r$ along with \eqref{eq_phi_one_coordinate} to get $\alpha_{r}\vee\beta_{r} = \alpha_{r} + \beta_{r}$, which implies $\alpha_{r} \cdot \beta_{r} = 0$.
	Then, thanks to \eqref{eq_phi_crossed_zero}, \Cref{eq_phi_si_diff} becomes
	\begin{equation*}
	\phi\big( (\alpha_i)_{k\le i\le r-1} , (\beta_i)_{k\le i\le r-1} \big) = \alpha_k + \beta_k
	\end{equation*}
	for all $1\le k\le r-1$.
	The proof follows by induction.
\end{proof}

Let us present one last property of $\phi$.
Notice that for any $h\in\N^*$, the function $\#_h\Fom_\bot$ does not depend on letters less than $h$ in the input words. 
Indeed, applying a letter less than $h$ can only decrement letters that are less than $h$.
Subsequently for any $r\geq h$ and $(\alpha_k)_{1\leq k\leq r} , (\beta_k)_{1\leq k\leq r} \in \N^r$:
\begin{equation*}
    \#_h \Fom_\bot\left( r^{\alpha_r}\dots1^{\alpha_1} \,,\, r^{\beta_r}\dots1^{\beta_1} \right)
    = \#_h \Fom_\bot\left( r^{\alpha_r}\dots h^{\alpha_h} \,,\, r^{\beta_r}\dots h^{\beta_h} \right).
\end{equation*}
Combined with the following ``translation-invariance'' property of Fomin's inverse rules:
\begin{equation*}
    \#_h \Fom_\bot\left( r^{\alpha_r}\dots h^{\alpha_h} \,,\, r^{\beta_r}\dots h^{\beta_h} \right)
    = \#_1 \Fom_\bot\left( (r-h+1)^{\alpha_r}\dots1^{\alpha_h} \,,\, (r-h+1)^{\beta_r}\dots1^{\beta_h} \right),
\end{equation*}
we deduce that:
\begin{equation}\label{phi_k}
    \phi\big( (\alpha_k)_{h\leq k\leq r} \,,\, (\beta_k)_{h\leq k\leq r} \big)
    = \beta_h + \#_h \Fom_\bot\left( r^{\alpha_r}\dots1^{\alpha_1} \,,\, r^{\beta_r}\dots1^{\beta_1} \right).
\end{equation}

\subsection{Continuous inverse growth rules}\label{section_proofs_diff}

The aim of this section is to prove \Cref{th_derivative_per} and \Cref{cor_gradient_lambda}, using all the previously developed tools.

\begin{proof}[Proof of \Cref{th_derivative_per}]

Fix an integer $r\in\N^*$ and a permuton $\mu$ satisfying $\LISt_r(\mu)=1$.
The function $\lamt^\mu$ is then entirely described by $\lamt^{\mu}_1,\dots,\lamt^{\mu}_r$, for which we will omit the exponent. 
Moreover $\mu$'s sampled permutations are almost surely $r$-increasing, meaning their diagrams consist of at most $r$ rows. 
The tableaux of $\sigma_n \sim \sample_n(\mu)$ are then entirely described by the functions $\lambda^{\sigma_n}_1,\dots,\lambda^{\sigma_n}_r$, for which we will also omit the exponent.

Let  $(x,y)\in(0,1]^2$ be such that the left-derivatives
\begin{equation*}
    \alpha_k := \partial^-_x \lamt_k(x,y)
    \quad\text{and}\quad
    \beta_k := \partial^-_y \lamt_k(x,y)
\end{equation*}
for $k\in\lb1,r\rb$ all exist.
Fix $s,t\geq0$, and $\epsilon>0$ small enough to have $x-t\epsilon , y-s\epsilon \geq 0$. 
According to \Cref{prop_lambda_cv}, we can and will work on an event where the convergence
\begin{equation*}
    \text{for all }k\in\lb1,r\rb,\quad
    \left\| \lamt_k - \frac{1}{n}\lambda_k\big( \lfloor\cdot\,n\rfloor , \lfloor\cdot\,n\rfloor \big)  \right\| = \on(1)
\end{equation*}
holds.
We can use this convergence along with the derivability condition of $\lamt$ to obtain decent control over the disposition of Fomin's edge labels. 
Fix an integer $p\in\N^*$ and define regular subdivisions
\begin{equation*}
    x-t\epsilon = x_p < x_{p-1} < \dots < x_0 = x
    \quad \text{and} \quad
    y-s\epsilon = y_p < y_{p-1} < \dots < y_0 = y
\end{equation*}
of respective steps $t\epsilon / p$ and $s\epsilon / p$. 
Then for any $k\in\lb1,r\rb$ and $j\in\lb1,p\rb$:
\begin{align*}
    &\lamt_k(x_{j\m1},y) - \lamt_k(x_j,y)
    \\& = \lamt_k(x,y) - (j\m1)t\epsilon\alpha_k/p + \oeps\left(\epsilon(j\m1)/p\right) - \lamt_k(x,y) + jt\epsilon\alpha_k/p + \oeps\left(\epsilon j/p\right)
    \\& = \epsilon t\alpha_k/p + \oeps(\epsilon)
\end{align*}
whence
\begin{equation}\label{lambda_on_subdivision}
     \lambda_k\big( \lfloor nx_{j\m1}\rfloor,\lfloor ny\rfloor \big) 
     - \lambda_k\big( \lfloor nx_j\rfloor,\lfloor ny\rfloor \big)
     = n \epsilon t\alpha_k/p + n\oeps(\epsilon) + \on(n).
\end{equation}
This means that edge labels of $\sigma_n$ on the north side of $\lb \lfloor n(x\m t\epsilon) \rfloor , \lfloor nx \rfloor \rb \times \lb \lfloor n(y\m s\epsilon) \rfloor , \lfloor ny \rfloor \rb$ are quite well equidistributed. 
Denote by $w_\lambda^\top,w_\lambda^\rig,w_\lambda^\bot,w_\lambda^\lef$ the words associated with the edge labels of $\sigma_n$ on this rectangular grid, as defined in \Cref{section_proofs_Fomin}. 
Write the top-word as a concatenation $w_\lambda^\top = w_{(1)}^\top \dots w_{(p)}^\top$, where each $w_{(j)}^\top$ consists of labels from $x$-coordinate $\lfloor nx_{j\m1} \rfloor$ to $\lfloor nx_j \rfloor$.
Then (\ref{lambda_on_subdivision}) translates as:
\begin{equation}\label{amounts_on_subdivision}
    \#_k w_{(j)}^\top
    = n \epsilon t\alpha_k/p + n\oeps(\epsilon) + \on(n).
\end{equation}
In particular, by summing (\ref{amounts_on_subdivision}) over $j\in\lb1,p\rb$:
\begin{equation}\label{amounts_global}
    \#_k w_\lambda^\top
    = n \epsilon t\alpha_k + n\oeps(\epsilon) + \on(n).
\end{equation}
Similar statements hold for $w_\lambda^\rig$.
Next, let us explain how to relate this property of equidistribution to \Cref{prop_approximation_par_decreasing_si_bonnes_LIS}.
In what follows, we consider that nondecreasing subsequences can only contain positive letters.

\bigskip

\noindent\underline{Claim 1}:
For any $i,k\in\lb1,r\rb$ the following holds:
\begin{equation*}\label{approximation_LIS_si_equiréparti}
    \left\vert \LIS_i\left( w_\lambda^{\top,\leq k} \right)
    - nt\epsilon \left( \alpha_{(1)}^k + \dots + \alpha_{(i)}^k \right) \right\vert
    \leq nt\epsilon r^2/p + n\oeps(\epsilon) + \on(n)
\end{equation*}
with the notation of \Cref{prop_approximation_par_decreasing_si_bonnes_LIS}.

\medskip

To see this, note that a nondecreasing subsequence of $w_\lambda^{\top, \leq k}$ reads as a subsequence of $1$'s, followed by a subsequence of $2$'s, ..., followed by a subsequence of $k$'s. 
Let $A=A_1\sqcup\dots\sqcup A_i$ be an $i$-nondecreasing subsequence (\textit{i.e.} a union of $i$ nondecreasing subsequences, which we can suppose disjoint) of $w_\lambda^{\top, \leq k}$ with maximum length.

We say $j\in\lb1,p\rb$ is ``valid'' if for each $h\in\lb1,i\rb$, the nondecreasing subsequence $A_h$ takes letters of at most one value in the subword $w_{(j)}^{\top, \leq k}$.
In this case, at best, $A$ can thus take the $i$ most represented letters in $w_{(j)}^{\top, \leq k}$.
Using (\ref{amounts_on_subdivision}), this allows to bound the number of letters picked up by $A$ in any valid subword $w_{(j)}^{\top, \leq k}$ by
\begin{equation*}
    n\epsilon t \left( \alpha_{(1)}^k + \dots + \alpha_{(i)}^k \right)/p + n\oeps(\epsilon) + \on(n).
\end{equation*}

When an index $j\in\lb1,p\rb$ is not valid, it means that some $A_h$ takes letters of at least two different values in $w_{(j)}^{\top, \leq k}$.
Since each $A_h$ is a nondecreasing subsequence, it can ``invalidate'' at most $k-1$ subwords.
Therefore at most $i(k-1)$ subwords are not valid, and in each of them $A$ can potentially take every letter.
This observation, along with the previous analysis, leads to the following upper bound:
\begin{equation*}
    \LIS_i\left( w_\lambda^{\top, \leq k} \right) \leq 
    n\epsilon t \left( \alpha_{(1)}^k + \dots + \alpha_{(i)}^k \right)
    + i(k-1) nt\epsilon /p
    + n\oeps(\epsilon) + \on(n).
\end{equation*}

Conversely we can construct an $i$-nondecreasing subsequence $B=B_1\sqcup\dots\sqcup B_{i\wedge k}$ of $w_\lambda^{\top, \leq k}$ by letting each $B_h$ take all $b_h$'s of $w_\lambda^{\top, \leq k}$, where $b_h$ is such that $\alpha_{b_h} = \alpha_{(h)}^k$.
Subsequently, using (\ref{amounts_on_subdivision}):
\begin{equation*}
    \LIS_i\left( w_\lambda^{\top, \leq k} \right) \geq 
    n\epsilon t \left( \alpha_{(1)}^k + \dots + \alpha_{(i\wedge k)}^k \right)
    + n\oeps(\epsilon) + \on(n).
\end{equation*}
This concludes the proof of Claim 1.

\bigskip

\noindent\underline{Claim 2}:
Define $w_\ord^\top := r^{\lfloor n\epsilon t\alpha_r \rfloor} \dots 1^{\lfloor n\epsilon t\alpha_1 \rfloor}$.
There exists a constant $\chi_r>0$ such that for any word $w^\rig$ and any $h\in\N^*$:
\begin{equation*}
    \left\{
    \begin{array}{ll}
    \left| \#_h \Fom_\bot\left( w_\lambda^\top , w^\rig \right)
    - \#_h \Fom_\bot\left( w_\ord^\top , w^\rig \right) \right|
    \leq \chi_r t\epsilon n/p + n\oeps(\epsilon) + \on(n) ;\vspace{0.2em}\\
    \left| \#_h \Fom_\lef\left( w_\lambda^\top , w^\rig \right)
    - \#_h \Fom_\lef\left( w_\ord^\top , w^\rig \right) \right|
    \leq \chi_r t\epsilon n/p + n\oeps(\epsilon) + \on(n).
    \end{array}
    \right.
\end{equation*}

\medskip

Indeed with the notation of \Cref{prop_approximation_par_decreasing_si_bonnes_LIS}, Claim 1 rewrites
\begin{equation*}
    \text{for any }i,k\in\lb1,r\rb,\quad
    \eta_{i,k} \leq nt\epsilon r^2/p + n\oeps(\epsilon) + \on(n)
\end{equation*}
thus, letting $\chi_r := 4r^3(8r^2)^r$:
\begin{equation*}
    \dist_\Fom\left( w_\lambda^\top,w_\ord^\top \right)
    \leq \chi_r \epsilon n/p + n\oeps(\epsilon) + \on(n),
\end{equation*}
which is the first desired inequality.
For the second inequality, use this first inequality along with \Cref{mass_conservation,amounts_global}:
\begin{align*}
    & \left| \#_h \Fom_\lef\left( w_\lambda^\top , w^\rig \right)
    - \#_h \Fom_\lef\left( w_\ord^\top , w^\rig \right) \right|
    \\& \leq \left| \#_h \Fom_\bot\left( w_\lambda^\top , w^\rig \right)
    - \#_h \Fom_\bot\left( w_\ord^\top , w^\rig \right) \right|
    + \left| \#_h w_\lambda^\top - \#_h w_\ord^\top \right|
    \\& \leq \chi_r \epsilon n/p + n\oeps(\epsilon) + \on(n).
\end{align*}
This proves Claim 2, and by symmetry the following holds as well:

\bigskip

\noindent\underline{Claim 3}:
Define $w_\ord^\rig := r^{\lfloor n\epsilon s\beta_r \rfloor} \dots 1^{\lfloor n\epsilon s\beta_1 \rfloor}$.
Then for any word $w^\top$ and any $h\in\N^*$:
\begin{equation*}
    \left\{
    \begin{array}{ll}
    \left| \#_h \Fom_\lef\left( w^\top , w_\lambda^\rig \right)
    - \#_h \Fom_\lef\left( w^\top , w_\ord^\rig \right) \right|
    \leq \chi_r s\epsilon n/p + n\oeps(\epsilon) + \on(n) ;\vspace{0.5em}\\
    \left| \#_h \Fom_\bot\left( w^\top , w_\lambda^\rig \right)
    - \#_h \Fom_\bot\left( w^\top , w_\ord^\rig \right) \right|
    \leq \chi_r s\epsilon n/p + n\oeps(\epsilon) + \on(n).
    \end{array}
    \right.
\end{equation*}

\medskip

\noindent We can now finish the proof of \Cref{th_derivative_per}. Using Claims 2 and 3 we obtain for any $k\in\N^*$:
\begin{align*}
    & \left| \#_k \Fom_\bot\left( w_\lambda^\top , w_\lambda^\rig \right) 
    - \#_k \Fom_\bot\left( w_\ord^\top , w_\ord^\rig \right) \right|
    \\& \leq \left| \#_k \Fom_\bot\left( w_\lambda^\top , w_\lambda^\rig \right) 
    - \#_k \Fom_\bot\left( w_\ord^\top , w_\lambda^\rig \right) \right|
    + \left| \#_k \Fom_\bot\left( w_\ord^\top , w_\lambda^\rig \right)
    - \#_k \Fom_\bot\left( w_\ord^\top , w_\ord^\rig \right) \right|
    \\& \leq (t+s)\chi_r \epsilon n/p + n\oeps(\epsilon) + \on(n).
\end{align*}  
Then using the equality $\lambda_k\big( \lfloor nx\rfloor,\lfloor ny\rfloor \big) - \lambda_k\big( \lfloor n(x\m t\epsilon)\rfloor,\lfloor n(y\m s\epsilon)\rfloor \big) =  \#_k \Fom_\bot\left( w_\lambda^\top , w_\lambda^\rig \right) + \#_k w_\lambda^\rig$, (\ref{amounts_global}), \Cref{prop_Fbot_real_words_and_continuity_of_phi} and Property (\ref{phi_k}) we get:
\begin{align*}
    \left| \Big( \lambda_k\big( \lfloor nx\rfloor,\lfloor ny\rfloor \big) - \lambda_k\big( \lfloor n(x\m t\epsilon)\rfloor,\lfloor n(y\m s\epsilon)\rfloor \big) \Big)
    - \phi\big( (n\epsilon t\alpha_i)_{k\leq i\leq r}, (n\epsilon s\beta_i)_{k\leq i\leq r} \big) \right|
    \\ \leq (t+s)\chi_r \epsilon n/p + n\oeps(\epsilon) + \on(n).
\end{align*}
Dividing by $n$, using Property (\ref{homogeneity_phi}) and taking $\as{n}{\infty}$ we deduce:
\begin{align*}
    \left| \left( \lamt_k(x,y) - \lamt_k(x-t\epsilon,y-s\epsilon) \right)
    - \epsilon \phi\big( (t\alpha_i)_{k\leq i\leq r}, (s\beta_i)_{k\leq i\leq r} \big) \right|
    \leq (t+s)\chi_r \epsilon /p + \oeps(\epsilon),
\end{align*}
whence:
\begin{align*}
    \limsup_{\epsilon\rightarrow0^+}
    \left| \frac{\lamt_k(x,y) - \lamt_k(x-t\epsilon,y-s\epsilon)}{\epsilon}
    -  \phi\big( (t\alpha_i)_{k\leq i\leq r}, (s\beta_i)_{k\leq i\leq r} \big)  \right|
    \leq (t+s)\chi_r /p.
\end{align*}
Since this holds for any $p\in\N^*$, this concludes the proof of \Cref{th_derivative_per}.
\end{proof}

\begin{proof}[Proof of \Cref{cor_gradient_lambda}]
For each $k\in\lb1,r\rb$, the function $\lamt_k^\mu$ is Lipschitz according to \Cref{prop_lambda_cv}.
Thanks to Rademacher's theorem, this implies that $\lamt_k^\mu$ is differentiable almost everywhere on $(0,1)^2$.
If $(x,y)$ is a point at which $\lamt_k^\mu$ is differentiable and if we write $\alpha_k := \partial^-_x \lamt_k(x,y)$ and $\beta_k := \partial^-_y \lamt_k(x,y)$, then the left-hand side of \Cref{eq_derivative_per} is simply equal to $t\alpha_k+s\beta_k$, therefore we have $\phi\big( (t\alpha_i)_{k\leq i\leq r} , (s\beta_i)_{k\leq i\leq r} \big) = t\alpha_k+s\beta_k$.
Almost everywhere on $(0,1)^2$, this equation holds for all $k\in\lb1,r\rb$.
\Cref{lem_phi_si_diff} then implies $\alpha_k \cdot \beta_k = 0$ almost everywhere for all $k$, which concludes the proof.
\end{proof}

\subsection{Proofs of injectivity results}\label{section_proofs_injectivity}

\begin{proof}[Proof of \Cref{th_decomposition}]
Let us first prove the existence of a triple satisfying those properties. 
For each $k\in\N^*$ consider $A_k\in\mathcal{P}^{k\nearrow}(\carre)$ such that $\mu(A_k)=\LISt_k(\mu)$ and $B_k\in\mathcal{P}^{k\searrow}(\carre)$ such that $\mu(B_k)=\LDSt_k(\mu)$. 
Then define:
\begin{equation*}
    A := \bigcup_{k\in\N^*}A_k,\;\;
    B := \bigcup_{k\in\N^*}B_k,\;\;
    \mu^\incr := \mu\big(A\cap\cdot\big),\;\;
    \mu^\decr := \mu\big(B\cap\cdot\big),\;\;
    \mu^\sub := \mu\left( \left(A\cup B\right)^c \cap\cdot\right).
\end{equation*}
Clearly one has
\begin{equation*}
    \LISt_\infty\big(\mu^\incr\big)
    = \mu(A)
    =\mu^\incr\left(\carre\right)
    \quad\text{and}\quad
    \LDSt_\infty\big(\mu^\decr\big)
    = \mu(B)
    =\mu^\decr\left(\carre\right)
\end{equation*}
and the measure $\mu^\sub$ is singular with respect to both $\mu^\incr$ and $\mu^\decr$. 
Then notice that the intersection between a nondecreasing and a nonincreasing subset is entirely contained within the union of a vertical and a horizontal line. 
Since $\mu$ is a pre-permuton, it puts mass zero to such an intersection. As $A$ (resp.~$B$) is a countable union of nondecreasing (resp.~nonincreasing) subsets, it implies $\mu\big( A \cap B \big) = 0$.
Thus (i) holds, and $\mu= \mu^\incr + \mu^\decr + \mu^\sub$.

Now for each $k\in\N^*$, since $\mu^\incr, \mu^\decr$ are sub-measures of $\mu$, one has 
\begin{equation*}
    \LISt_k\big(\mu^\incr\big) \leq \LISt_k(\mu)
    \quad\text{and}\quad
    \LDSt_k\big(\mu^\decr\big) \leq \LDSt_k(\mu).
\end{equation*}
Moreover
\begin{equation*}
    \LISt_k\big(\mu^\incr\big) \geq \mu^\incr(A_k) = \mu(A_k) = \LISt_k(\mu)
    \quad\text{and}\quad
    \LDSt_k\big(\mu^\decr\big) \geq \mu^\decr(B_k) = \mu(B_k) = \LDSt_k(\mu),
\end{equation*}
hence (ii) holds. Subsequently
$\LISt_\infty\big(\mu^\incr\big)=\LISt_\infty(\mu)$
{and}
$\LDSt_\infty\big(\mu^\decr\big)=\LDSt_\infty(\mu)$.

Now consider a nondecreasing subset $A'$ that is disjoint of $A$. 
For each $k\in\N^*$ the subset $A_k \cup A'$ is $(k+1)$-nondecreasing, implying
\begin{equation*}
    \LISt_{k+1}(\mu) \geq \LISt_k(\mu) + \mu(A').
\end{equation*}
As a consequence:
\begin{equation*}
    \mu(A') \leq \LISt_{k+1}(\mu) - \LISt_k(\mu) \cv{k}{\infty} 0.
\end{equation*}
In particular $\mu^\decr(A')=\mu^\sub(A')=0$. Since this holds for any nondecreasing subset disjoint of $A$ and $\mu^\decr , \mu^\sub$ are singular with respect to $\mu^\incr$, this implies:
\begin{equation*}
    \LISt\big(\mu^\decr\big) 
    = \LISt\big(\mu^\sub\big) = 0
\end{equation*}
and in particular $\lamt^\mu=\lamt^{\mu^\incr}$.
Repeating this argument with nonincreasing subsets, we deduce (iii).

\bigskip

Let us turn to the uniqueness claim. Keep the previous definition of $\mu^\incr,\mu^\decr,\mu^\sub$ and write $\mu=\nu^\incr+\nu^\decr+\nu^\sub$ where $\nu^\incr, \nu^\decr,\nu^\sub$ are finite (positive) measures satisfying
\begin{equation*}
    \LISt_\infty\big(\nu^\incr\big)
    =\nu^\incr\left(\carre\right)
    =\LISt_\infty(\mu)
    \quad\text{and}\quad
    \LDSt_\infty\big(\nu^\decr\big)
    =\nu^\decr\left(\carre\right)
    =\LDSt_\infty(\mu).
\end{equation*}
Let $C$ be a countable union of nondecreasing subsets supporting the measure $\nu^\incr$. 
Then:
\begin{equation*}
    \nu^\incr\left(\carre\right) 
    = \nu^\incr(C) 
    \leq \mu(C) \leq \LISt_\infty(\mu).
\end{equation*}
Since the first and last term are equal, we deduce equality of all terms. 
Hence the sub-measure $\nu^\incr$ of $\mu$ is given by the restriction $\nu^\incr = \mu( C \cap \cdot )$.
Moreover Property (iii) implies
\begin{equation}\label{masse_pas_ailleurs}
    \mu^\decr(C) = \mu^\sub(C) = 0,
\end{equation}
thus
\begin{equation*}
    \LISt_\infty\big(\mu\big) =
    \mu(C) = \mu^\incr(C) 
    \leq \mu^\incr\big(\carre\big) =
    \LISt_\infty\big(\mu\big)
\end{equation*}
from which we deduce
\begin{equation}\label{masse_dedans}
    \mu^\incr(C) = \mu^\incr\left(\carre\right).
\end{equation}
\Cref{masse_pas_ailleurs,masse_dedans} together imply
\begin{equation*}
    \mu^\incr = \mu( C \cap \cdot ) = \nu^\incr
\end{equation*}
as desired. 
The proof of $\mu^\decr=\nu^\decr$ works the same, and we finally deduce $\mu^\sub=\nu^\sub$.
\end{proof}

\medskip

\begin{proof}[Proof of \Cref{prop_injectivity_lambda}]
    For any $(x,y)\in\carre$:
    \begin{equation*}
        \LISt_\infty\left( \mu\vert_{[0,x]\times[0,y]} \right)
        = \mu^\incr([0,x]\times[0,y]).
    \end{equation*}
    Hence $\lamt^\mu = \lamt^\nu$ implies 
    $$\mu^\incr([0,x]\times[0,y]) = \nu^\incr([0,x]\times[0,y])$$ 
    for any $(x,y)\in\carre$, and therefore $\mu^\incr=\nu^\incr$.
\end{proof}

\medskip

The idea to prove \Cref{prop_injectivity_RS_one} is that while we could not recover $\lamt^\mu$ from its PDE together with its boundary condition $\RSt(\mu)$, there is a simple explicit solution for the last row of the tableaux.

\begin{lemma}\label{lem_solving_lamt_r}
    Let $\mu$ be a pre-permuton satisfying $\LISt_r(\mu)=1$ for some $r\in\N^*$.
    Then for any $(x,y)\in\carre$:
    \begin{equation*}
        \lamt_r(x,y) = \lamt_r(x,1)\wedge\lamt_r(1,y).
    \end{equation*}
\end{lemma}

\begin{proof}
    Denote by $F=(F_1,F_2)$ the pair of distribution functions of each marginal of $\mu$, so that the push-forward $\hat{\mu}=F_*\mu$ is a permuton. 
    Since $F$ is nondecreasing in each coordinate, we have for any $k\in\lb1,r\rb$ and $(x',y')\in\carre$:
    \begin{equation*}
        \LISt_k\left( \hat{\mu}\vert_{[0,F_1(x')]\times[0,F_2(y')]} \right)
        = \LISt_k\left( \mu\vert_{[0,x']\times[0,y']} \right).
    \end{equation*}
    Thus $\lamt^\mu = \lamt^{\hat{\mu}} \circ F$. 
    
    For each $n\in\N^*$ let $\sigma_n\sim\sample_n(\mu)$.
    As noted before obtaining Property (\ref{phi_k}), for any words $w^\top$ and $w^\rig$, the number of $r$'s in $\Fom_\bot(w^\top,w^\rig)$ does not depend on letters less than $r$ in $w^\top$ and $w^\rig$.
    Let $1\leq i\leq i'\leq n$, $1\leq j\leq j'\leq n$, and $w_n^\top,w_n^\rig$ be the top- and right-words of $\sigma_n$'s edge labels on $\lb i,i'\rb\times\lb j,j'\rb$.
    Since these words almost surely only have labels at most $r$, the previous observation yields:
    \begin{equation*}
        \#_r \Fom_\bot\left( w_n^\top,w_n^\rig \right)
        = \#_r \Fom_\bot\left( r^{\#_r w_n^\top} ,\, r^{\#_r w_n^\rig}  \right)
        = \left( \#_r w_n^\top - \#_r w_n^\rig \right)^+.
    \end{equation*}
    Then
    \begin{align*}
        \lambda_r^{\sigma_n}(i',j') - \lambda_r^{\sigma_n}(i,j)
        &= \#_r \Fom_\bot\left( r^{\#_r w_n^\top} ,\, r^{\#_r w_n^\rig}  \right) + \#_r w_n^\rig
        \\&= \#_r w_n^\top \vee \#_r w_n^\rig
        \\&= \big( \lambda_r^{\sigma_n}(i',j')-\lambda_r^{\sigma_n}(i,j') \big) \vee \big( \lambda_r^{\sigma_n}(i',j')-\lambda_r^{\sigma_n}(i',j)\big),
    \end{align*}
    and in particular when $i'=j'=n$, almost surely:
    \begin{equation*}
        \lambda_r^{\sigma_n}(i,j) 
        = \lambda_r^{\sigma_n}(i,n) \wedge \lambda_r^{\sigma_n}(n,j).
    \end{equation*}
    Using \Cref{prop_lambda_cv} we get for any $(x,y)\in\carre$:
    \begin{equation*}
        \lambda_r^{\hat{\mu}}(x,y) 
        = \lambda_r^{\hat{\mu}}(x,1) \wedge \lambda_r^{\hat{\mu}}(1,y).
    \end{equation*}
    Using the equality $\lamt^\mu=\lamt^{\hat{\mu}}\circ F$ we finally deduce for any $(x,y)\in\carre$:
    \begin{equation*}
        \lamt_r^\mu(x,y) 
        = \lamt_r^{\hat{\mu}}\big(F_1(x),F_2(y)\big) 
        = \lamt_r^{\hat{\mu}}\big(F_1(x),1\big) \wedge \lamt_r^{\hat{\mu}}\big(1,F_2(y)\big)
        = \lamt_r^{\mu}(x,1) \wedge \lamt_r^{\mu}(1,y).
        \qedhere
    \end{equation*}
\end{proof}

\medskip

\begin{proof}[Proof of \Cref{prop_injectivity_RS_one}]
    Under the first hypothesis we have $\mu^\incr\left(\carre\right) = \LISt\big(\mu^\incr\big)$. 
    Applying \Cref{lem_solving_lamt_r} to the renormalization of $\mu^\incr$, we deduce for any $(x,y)\in\carre$:
    \begin{equation}\label{se_ramener_au_bord}
        \lamt_1^\mu(x,y) = \lamt_1^{\mu^\incr}(x,y) 
        = \lamt_1^{\mu^\incr}(x,1) \wedge \lamt_1^{\mu^\incr}(1,y)
        = \lamt_1^{\mu}(x,1) \wedge \lamt_1^{\mu}(1,y).
    \end{equation}
    Since $\RSt(\mu) = \RSt(\nu)$ implies
    \begin{equation*}
        \LISt_\infty(\nu)
        = \sum_{k\geq1} \lamt_k^\nu(1,1)
        = \sum_{k\geq1} \lamt_k^\mu(1,1)
        = \lamt_1^\mu(1,1)
        = \lamt_1^\nu(1,1)
        = \LISt_1(\nu),
    \end{equation*}
    $\nu$ also satisfies \Cref{se_ramener_au_bord}. 
    Therefore $\lamt^\mu = \lamt^\nu$, and we conclude with \Cref{prop_injectivity_lambda}.
\end{proof}

\section*{Acknowledgments}

I would like to express my sincere gratitude to my supervisor Valentin Féray for his invaluable guidance and support.
I would also like to thank 
two anonymous reviewers for their suggestions that greatly improved the presentation of this paper;
Blaise Colle, Corentin Correia, Emmanuel Kammerer, and Pierrick Sieste for fruitful discussions; 
and Rémi Maréchal, whose master's thesis laid the foundation of this work by proving 
Propositions \ref{prop_semicon_LISt} and \ref{prop_conv_LISk} when $k=1$ with a different but essentially equivalent definition of $\LISt_1$.
The author is partially supported by the Program “Future Leader” of the Initiative “Lorraine Université d’Excellence” (LUE), and by the Agence Nationale de la Recherche funding CORTIPOM ANR-21-CE40-0019.

\bibliographystyle{alpha}
\bibliography{bibli}

\end{document}